\documentclass[a4paper]{article}

\usepackage{a4}
\usepackage[english]{babel}
\usepackage[dvipsnames,usenames]{color}

\usepackage{bbm}

\usepackage{verbatim,graphics,indentfirst} 
\usepackage{url}
\usepackage{stmaryrd}
\usepackage[nosumlimits]{amsmath}
\usepackage{amssymb}
\usepackage{amsthm}
\usepackage{float}
\usepackage{booktabs}

\usepackage{latexsym}
\usepackage{enumerate}
\usepackage{graphicx}
\usepackage[T1]{fontenc}
\usepackage{lmodern}

\usepackage[utf8]{inputenc}
\usepackage{graphicx}
\usepackage{amsfonts}
\usepackage{nicefrac}
\usepackage{array}
\usepackage{appendix}

\usepackage{tikz}
\usepackage{tikz-qtree}
\usepackage[]{forest}

\usepackage{xspace}
\usepackage{marginnote}
\usepackage{bussproof}
\usepackage{cmll}

\usepackage{proof}

\usepackage{pdflscape}

\usepackage{tikz}

\usepackage{mathdots}

\newtheorem{theorem}{Theorem}[section]
\newtheorem{definition}[theorem]{Definition}
\newtheorem{proposition}[theorem]{Proposition}
\newtheorem{lemma}[theorem]{Lemma}

\newtheorem{example}[theorem]{Example}
\newtheorem{remark}[theorem]{Remark}

\newcommand{\conn}{\copyright}

\newcommand{\LThreeName}{\textbf{\textup{\L$_3$}}}
\newcommand{\LPName}{\textbf{\textup{LP}}}
\newcommand{\LPImpName}{\textbf{\textup{LP}$^\ararr{}$}}
\newcommand{\BKName}{\textbf{\textup{BK}}}
\newcommand{\PWKName}{\textbf{\textup{PK}}}
\newcommand{\JThreeName}{\textbf{\textup{J3}}}
\newcommand{\GThreeName}{\textbf{\textup{G3}}}
\newcommand{\SKName}{\textbf{\textup{K}}}
\newcommand{\SThreeName}{\textbf{\textup{S3}}}
\newcommand{\MSevenName}{\textbf{\textup{L7}}}
\newcommand{\MEightName}{\textbf{\textup{L8}}}
\newcommand{\SetteName}{\textbf{\textup{S}}}
\newcommand{\PostName}{\textbf{\textup{P3}}}
\newcommand{\BelnapDunnName}{\textbf{\textup{E}}}

\newcommand{\LName}{\mathbf{L}}

\newcommand{\lukConj}{\ensuremath{\land_{\mathsf{min}}}}
\newcommand{\lukDisj}{\ensuremath{\lor_{\mathsf{max}}}}
\newcommand{\lukImp}{\ensuremath{\to_{\textrm{\LThreeName}}}}
\newcommand{\lukNeg}{\ensuremath{\lnot_{\LThreeName}}}

\newcommand{\postNeg}{\ensuremath{\lnot_{\PostName}}}

\newcommand{\setteImp}{\ensuremath{\to_{\SetteName}}}
\newcommand{\setteNeg}{\ensuremath{\lnot_{\SetteName}}}

\newcommand{\KleeImp}{\ensuremath{\to_{\SKName}}}

\newcommand{\BKleeConj}{\ensuremath{\land_{\BKName}}}
\newcommand{\BKleeDisj}{\ensuremath{\lor_{\BKName}}}
\newcommand{\BKleeNeg}{\ensuremath{\lnot_{\BKName}}}
\newcommand{\BKleeImp}{\ensuremath{\to_{\BKName}}}

\newcommand{\WKleeImp}{\ensuremath{\to_{\PWKName}}}

\newcommand{\SobocNeg}{\ensuremath{\lnot_{\SThreeName}}}
\newcommand{\SobocImp}{\ensuremath{\to_{\SThreeName}}}

\newcommand{\LPImp}{\ensuremath{\to_{\LPName}}}

\newcommand{\JThreeImp}{\ensuremath{\to_{\JThreeName}}}
\newcommand{\LukPoss}{\ensuremath{\Diamond_{}}}

\newcommand{\HImp}{\ensuremath{\to_{\GThreeName}}}
\newcommand{\HNeg}{\ensuremath{\lnot_{\GThreeName}}}

\newcommand{\DSetOne}{\{\tv\}}
\newcommand{\DSetTwo}{\{\tv,\uv\}}

\newcommand{\EmptySet}{\varnothing}
\newcommand{\SetFmla}{{\upshape\textsc{Set-Fmla}}}
\newcommand{\SetSet}{{\upshape\textsc{Set-Set}}}

\newcommand{\MatA}{\mathbb{M}}

\newcommand{\SetSetCR}{\rhd}

\newcommand{\FmA}{\varphi}
\newcommand{\FmB}{\psi}
\newcommand{\FmC}{\gamma}
\newcommand{\FmD}{\delta}

\newcommand{\FmSetA}{\Phi}
\newcommand{\FmSetB}{\Psi}
\newcommand{\FmSetC}{\Pi}
\newcommand{\FmSetD}{\Gamma}
\newcommand{\FmSetE}{\Delta}

\newcommand{\LabSysNameA}{\mathsf{G}}
\newcommand{\LabSysNameTriv}{\mathsf{G}_{\mathsf{triv}}}

\newcommand{\CtxLabAOne}{\Gamma}
\newcommand{\CtxLabAHalf}{\Psi}
\newcommand{\CtxLabAZero}{\Phi}

\newcommand{\LabelValA}{\nu}
\newcommand{\LabFmSetA}{X}
\newcommand{\LabFmSetB}{Y}
\newcommand{\LabFmSetC}{Z}

\newcommand{\SetFmlaHRule}[3]{
    \AXC{\ensuremath{#1}}
    \LL{\ensuremath{#3}}
    \UIC{\ensuremath{#2}}
    \DP
}
\newcommand{\SetFmlaHName}{\mathcal{H}}

\newcommand{\ThreeLabName}[1]{\ell.{#1}}
\newcommand{\ThreeLabStreamName}[1]{s\ell.{#1}}

\newcommand{\MCRule}[2]{\frac{#1}{#2}}

\newcommand{\PropA}{p}
\newcommand{\PropB}{q}
\newcommand{\PropC}{r}
\newcommand{\PropD}{s}

\newcommand{\AlgInterp}[2]{{#1}_{#2}}

\newcommand{\SymbDef}{:=}

\newcommand{\ThreeValuesSet}{V_3}

\newcommand{\DesSetA}{D}

\newcommand{\GenSubs}[1]{S}

\newcommand{\TPart}{T}
\newcommand{\FPart}{F}

\newcommand{\ValuesSetComp}[2]{{#1}{\setminus}{#2}}

\newcommand{\AlgA}{\mathbf{A}}

\newcommand{\dv}{d}
\newcommand{\uva}{u_1}
\newcommand{\uvb}{u_2}

\newcommand{\Sep}{\mathrm{S}}
\newcommand{\SepB}{\mathrm{U}}

\newcommand{\LangAlgA}{\mathbf{L}_{\Sigma}(P)}
\newcommand{\LangSetA}{L_{\Sigma}(P)}

\newcommand{\LangSetAp}{L_{\Sigma}(\{\PropA\})}

\newcommand{\DesSeps}[1]{\Omega_{#1}}
\newcommand{\NDesSeps}[1]{\mho_{#1}}

\newcommand{\Hom}{\mathsf{Hom}}
\newcommand{\End}{\mathsf{End}}

\newcommand{\ValuationA}{v}

\newcommand{\CalcA}{\mathsf{R}}

\newcommand{\RuleA}{\mathsf{r}}

\newcommand{\AnalyticSetA}{\Theta}

\newcommand{\NSetSetCR}[1]{\blacktriangleright_{#1}}

\newcommand{\FmSetComp}[1]{{#1}^{\mathsf{c}}}

\newcommand{\SetCut}{\Pi}

\newcommand{\AnaSetSetCR}[2]{\SetSetCR_{#1}^{#2}}
\newcommand{\NAnaSetSetCR}[2]{\NSetSetCR{#1}^{#2}}

\newcommand{\Subf}[1]{\mathsf{subf}(#1)}
\newcommand{\SubfTheta}[2]{\mathsf{subf}^{#1}(#2)}
\newcommand{\Props}[1]{\mathsf{props}(#1)}

\newcommand{\PowerSet}[1]{\mathcal{P}(#1)}

\newcommand{\TreeA}{\mathsf{t}}
\newcommand{\NodeA}{\mathsf{n}}

\newcommand{\Star}{\star}

\newcommand{\FmlasTree}[1]{\mathsf{fmlas}(#1)}

\newcommand{\NodeLabel}[2]{\mathsf{lab}^{#1}(#2)}

\tikzset{
	treenode/.style = {align=center, inner sep=0pt, text centered},
	Ske/.style = {treenode, ellipse, double, draw=black,
		minimum width=6pt, thick},
	PIA/.style = {treenode, ellipse, black, draw=black,
		minimum width=6pt},
	Crit/.style = {treenode, rectangle, draw=black,
		minimum width=0.5em, minimum height=0.5em}
}

\newcommand{\fns}{\footnotesize}

\def\fCenter{{\mbox{$\ \vdash\ $}}}

\EnableBpAbbreviations

\newcommand{\aneg}{\ensuremath{\neg}\xspace}

\newcommand{\aand}{\ensuremath{\wedge}\xspace}
\newcommand{\aor}{\ensuremath{\vee}\xspace}
\newcommand{\ararr}{\ensuremath{\rightarrow}\xspace}

\newcommand{\alrarr}{\ensuremath{\leftrightarrow}\xspace}

\newcommand{\wdia}{\ensuremath{\Diamond}\xspace}

\newcommand{\tv}{\ensuremath{\mathbf 1}}
\newcommand{\tvb}{\ensuremath{\mathbf{\textcolor{blue}{1}}}}
\newcommand{\uv}{\ensuremath{\nicefrac{\mathbf{1}}{\mathbf{2}}}}
\newcommand{\uvg}{\ensuremath{\mathbf{\textcolor{green}{\nicefrac{1}{2}}}}}
\newcommand{\bv}{\ensuremath{\mathbf{0}}}
\newcommand{\bvr}{\ensuremath{\mathbf{\textcolor{red}{0}}}}
 
\newcolumntype{P}[1]{>{\centering\arraybackslash}p{#1}}

\newcommand{\smName}[1]{mc.{#1}}
\newcommand{\smcName}[1]{smc.{#1}}

 \newcommand{\half}{\frac{1}{2}}

\newcommand{\lf}{\lfloor}
\newcommand{\rf}{\rfloor}

\usepackage{hyperref}
\usepackage{framed}

\title{Generating proof systems for \\  three-valued propositional logics}
\author{Vitor Greati\thanks{(i) Universidade Federal do Rio Grande do Norte, Brazil and (ii) University of Groningen, The Netherlands. Work supported by the FWF project P33548.} \and Giuseppe Greco\thanks{Vrije Universiteit Amsterdam, The Netherlands.} \and Sérgio Marcelino\thanks{SQIG - Instituto de Telecomunica\c{c}\~oes,
Dep. Matem\'atica - Instituto Superior T\'ecnico,
Universidade de Lisboa, Portugal. Research funded by FCT/MCTES through national funds and when applicable co-funded by EU under the project UIDB/50008/2020.} \and Alessandra Palmigiano\thanks{(i) Vrije Universiteit Amsterdam, The Netherlands and (ii) Department of Mathematics and Applied Mathematics, University of Johannesburg, South Africa. The research of the second and fourth author has been funded in part by the NWO grant KIVI.2019.001.} \and Umberto Rivieccio\thanks{(i) Universidade Federal do Rio Grande do Norte, Brazil and (ii)  Universidad Nacional de Educación a Distancia, Spain.}}

\date{}

\begin{document}

\maketitle

\begin{abstract}




%
In general, providing an axiomatization for an arbitrary logic is a task that may require some ingenuity. In the case of logics defined by a finite logical matrix  
(three-valued logics being a particularly simple example), the generation of suitable
finite
axiomatizations 
can be completely automatized, essentially 
by expressing the matrix tables via inference rules.
In this chapter we illustrate how two
formalisms, the
\emph{3-labelled calculi}
of Baaz, Ferm\"uller and Zach
 and 
the
\emph{multiple-conclusion (or \SetSet{}) Hilbert-style calculi} of Shoesmith and Smiley,
may be uniformly employed to axiomatize logics defined by a three-valued logical matrix.
We discuss their main properties (related to completeness, decidability and proof search) and  make a systematic comparison between both formalisms. 
We observe that either of the following strategies are pursued: (i) expanding the \emph{metalanguage} 
of the formalism 
(via labels or types) or (ii) generalizing the usual notion of \emph{subformula property}
(as done in recent work by C. Caleiro and S. Marcelino)
while remaining as close as possible to the original language of the logic. In both cases,  desirable requirements are to guarantee the decidability of the associated symbolic procedure, as well as the possibility of an effective proof search.
%
The  generating procedure common to both formalisms can be described as follows: first (i)  convert the matrix semantics into rule form (we refer to this step as the \emph{generating subprocedure}) and then (ii) simplify the set of rules thus obtained, essentially relying on the defining properties of any Tarskian consequence relation (we refer to this step as the \emph{streamlining subprocedure}).
%
%
We illustrate through some examples that, if a minimal 
expressiveness assumption is met (namely, if the matrix defining the logic is \emph{monadic}), 
then it is straightforward to define  effective translations  guaranteeing
the equivalence between the
3-labelled and the \SetSet{}
approach. 
%
%
%
%
%
%
\end{abstract}


\section{Introduction}
\label{sec:Introduction}

Different types of logical calculi are better suited for different purposes. 
As is well known, so-called \emph{Gentzen-style} (or \emph{sequent}) \emph{calculi} are commonly employed whenever we are interested in automatizing proof search. On the other hand, so-called \emph{natural deduction calculi} are often taken as benchmarks for deciding when two derivations amount to the same argument, while \emph{Hilbert-style calculi} are a most useful tool
in characterizing the sets of formulas that are
closed with respect to the derivability relation of
a given logic. 
In this chapter, we shall be concerned with automatizing proof system generation, that is, with defining a \emph{procedure for generating decidable calculi} for a large class of logics, namely all logics that can be presented via a single finite-valued matrix and so, in particular, for three-valued logics. 

Before  getting into further technical details, let us spend a few words on what we mean by a \emph{three-valued logic}, as well as on the choice of the particular logics considered in the present chapter. In the literature on (non-classical) logics, a \emph{finite-valued} (propositional) logic is usually taken to be a logic determined or defined by some finite structure, typically a set whose elements
are interpreted as truth values in some intuitive sense. In turn, a \emph{propositional logic} is usually defined as
a binary relation  on the set of formulas -- inductively constructed from propositional variables by means of a given set of propositional connectives -- a relation
that satisfies the  postulates laid out by A.~Tarski~\cite[p.~31]{Ta56}; this is nowadays a standard notion
(though not the only existing one) 
in formal logic.


It has been shown that every propositional logic 
satisfying Tarski's postulates can be defined by means of a (possibly infinite) family of (possibly infinite) \emph{logical matrices}~\cite{W88}. It is thus fair to
affirm that
logical matrices constitute a most general method 
(i.e.~a standard \emph{semantics}) for introducing  arbitrary
(Tarskian) logics. 
From this standpoint, finite-valued logics
may be taken to be those logics that can be defined
by means of a \emph{finite} set of \emph{finite} logical matrices -- one cannot, without loss of generality, replace ``finite set of finite matrices'' by ``a single finite matrix'' (see~\cite{caleiro2018characterizing} for further discussion and counter-examples).
In this broad sense,
\emph{three-valued logics} are thus  those definable by some finite set of three-element logical matrices (also in this case we cannot replace ``finite set'' by 
``a single (three-element) matrix'').  
For technical reasons, in the present chapter
we shall narrow our focus to logics determined by \emph{a single  matrix}.
 It is to these logics, as we shall illustrate through plenty of examples, that
 the proof formalisms here considered  apply more straightforwardly\footnote{The   techniques we present are not limited to the single matrix case (regarding multiple conclusion calculi, see e.g.~\cite{marcelinowollic19}). 
 However, in order to deal with arbitrary three-valued (or indeed, more generally, finite-valued) logics, one needs more sophisticated tools (e.g.~from the theory of partial and non-deterministic matrices; cf.~Section \ref{sec:Conclusion}) that would not fit well within the scope of the present contribution.}.

By the above considerations, we shall    fix a three-element set,
say $\{ \bv , \uv, \tv \}$, as our standard space of truth values. A particular three-valued logic is then determined by a choice of propositional connectives (each  determined by the corresponding truth table) together with a choice of one or more sets of \emph{designated} truth values. 
All the non-trivial three-valued logics 
arise in this way when combining the logics obtained by designating either a singleton (say $\{ \tv \} $) or a two-element set (say $\{ \uv, \tv \}$).

Over an arbitrary propositional language, the number of three-valued truth tables (and therefore of possible logics) obtainable in the above-described way is obviously unbounded: in the  present chapter
we shall  consider a few examples chosen according to the criteria detailed below (see Section \ref{sec:CaseStudies}).
We shall introduce 3-labelled calculi~\cite{BaaFerZac98} and \SetSet{} Hilbert-style calculi~\cite{shoesmithsmiley1978} for a number of logics determined by a single three-valued matrix,
as well as what we call their \emph{streamlined} versions. We shall discuss their main properties (completeness, decidability, effective proof search) and  draw a first systematic comparison of both formalisms. 

\SetSet{} Hilbert-style calculi are a generalization of traditional Hilbert-style calculi (introduced by Frege and popularized by Hilbert) in the sense that their metalanguage is reduced to a minimum, and inference rules just manipulate \emph{sets of formulas} that constitute consecutions of (that is, belong to) the underlying logic. 
3-labelled calculi, on the other hand, are not Hilbert-style, for their language includes labels and rules manipulate \emph{sets of labelled formulas} or, equivalently,
\emph{3-sequents}.
The 3-labelled calculi generated for three-valued logics, as we will see, can actually be seen as generalizations of ``multiple-conclusion sequent calculi'' (see \cite[Chapter 3]{negri2008structural}), insofar as 
the distinction between structural and logical rules is in place, with zeroary rules being a generalized form of identity rules,
cut rules being a generalized form of (multi)cuts,
and each connective being associated with
its own set of introduction rules.

We observe, however, that the  core of the corresponding generating procedures is essentially the same, and amounts to first (i) converting the matrix semantics in rule form (we refer to this step as the \emph{generating subprocedure}) and then (ii) simplifying the set of rules thus obtained, essentially relying on the defining properties of any Tarskian consequence relation (we refer to this step as the \emph{streamlining subprocedure}). The key point where both generating subprocedures depart from each other has to do with the specific properties of the output. In the case of 3-labelled calculi, we can rely on labels corresponding to truth values, so turning matrices into rule form is a quite straightforward process. In the case of \SetSet{} Hilbert-style calculi, this additional information has to be conveyed in the logical language, hence the need of \emph{separators}, which are logical formulas capable of telling apart  truth values pairwise in a  sense that  will be formally specified in Section~\ref{Axiomatizing3valuedMatricesTheMonadicCase}. 

The chapter is organized as follows. In Section \ref{sec:Preliminaries}, we introduce basic
notions and terminology about algebras, languages, logics and logical matrices.
In Section \ref{sec:CaseStudies},
we provide an overview of the selected
three-valued logics that will be discussed
in subsequent sections. We devote special
attention to the traditional Hilbert-style calculi
available in the literature for
such logics, highlighting the fact
that for some of them no finite calculus
is currently known (whilst \SetSet{} finite calculi are guaranteed to exist by~\cite{shoesmithsmiley1978}).
In Section~\ref{sec:3LabelledCalculi},
we introduce the formalism of 3-labelled calculi and associated procedures for generating and streamlining analytic calculi
for three-valued logics.
We do the same in Section~\ref{sec:MultipleConclusionCalculi} for the formalism of \SetSet{} Hilbert-style calculi. Despite
the fact that the generating subprocedure
in the latter case demands an
expressiveness condition to be
satisfied by the
logical matrix at hand, we will
provide analytic finite axiomatizations for some non-sufficiently expressive
matrices in Subsection~\ref{sec:nonmonadic}
(note that the result of~\cite{shoesmithsmiley1978}
just referred to does not mention analyticity).
In Section~\ref{ComparisonBetween3labelledAndMultipleConclusionCalculi} we compare 
the generating procedures for 3-labelled
calculi and \SetSet{} Hilbert-style
calculi, showing possibilities of
translations from \SetSet{} 
to 3-labelled rules and proofs.
Finally, Section~\ref{sec:Conclusion}
closes the chapter with some concluding remarks
and suggestions for further research.


\section{Preliminaries}
\label{sec:Preliminaries}



A \emph{(propositional) signature}
is a family
$\Sigma \SymbDef \{\Sigma_k\}_{k \in \omega}$,
where each $\Sigma_k$ is
a collection of $k$-ary \emph{connectives}
(we qualify a connective as \emph{nullary} when $k=0$, as \emph{unary} when $k=1$, and as \emph{binary} when $k=2$).
A signature is finite when $\bigcup_{k \in \omega}\Sigma_k$
is finite.
In this work, the considered examples are in propositional signatures
containing the unary connectives $\neg, \circ, \bullet$ and $\Diamond$,
and the binary connectives $\land, \lor$ and $\to$.
%
A \emph{$\Sigma$-algebra} is a structure
$\mathbf{A} \SymbDef \langle A, \AlgInterp{\cdot}{\mathbf{A}} \rangle$, where $A$ is a non-empty
set called the \emph{carrier} or \emph{universe} of $\mathbf{A}$ and, for each $\conn \in \Sigma_k$ and $k \in \omega$, $\AlgInterp{\conn}{\mathbf{A}} : A^k \to A$ is the \emph{interpretation} of $\conn$
in $\mathbf{A}$. 
A $\Sigma$-algebra is \emph{finite} when both $\Sigma$ and its carrier are finite.
Given a denumerable set $P$ of
\emph{propositional variables}, 
the absolutely free algebra over
$\Sigma$ freely generated by $P$,
or simply the
\emph{(propositional) language} over $\Sigma$ (generated by $P$),
is denoted by $\LangAlgA$,
and its members are called \emph{$\Sigma$-formulas}.
The set of all subformulas of a given formula $\FmA \in \LangSetA$ 
will be denoted by $\Subf{\FmA}$,
and the set of all propositional variables
occurring in $\FmA$
will be denoted by $\Props{\FmA}$.
Given $\FmSetA \subseteq \LangSetA$, we let
$\FmSetComp{\FmSetA} \SymbDef \LangSetA{\setminus}\FmSetA$.

The collection of homomorphisms between
two $\Sigma$-algebras $\mathbf{A}$
and $\mathbf{B}$ is denoted by $\Hom(\mathbf{A},\mathbf{B})$.
Furthermore, the set
of endomorphisms on $\mathbf{A}$ is denoted by
$\End(\mathbf{A})$
and each one of the members $\sigma \in \End(\LangAlgA)$
is called a \emph{substitution}.
In case $\PropA_1,\ldots,\PropA_n$
are the only propositional variables occurring in $\FmA \in \LangSetA$,
we say that $\FmA$ is $n$-ary and denote by $\FmA_\mathbf{A}$
the $n$-ary operation on $A$
such that,
for all $a_1,\ldots,a_n \in A$,
$\FmA_\mathbf{A}(a_1,\ldots,a_n) = h(\FmA)$, for
an $h \in \Hom(\LangAlgA, \mathbf{A})$ with
$h(\PropA_i) = a_i$ for each
$1 \leq i \leq n$.
Instead of using ``$n$-ary'', we may conveniently use the word \emph{unary}, when $n=1$, and \emph{binary}, when $n=2$,
as we do for connectives.
Also, if $\FmB_1,\ldots,\FmB_n \in \LangSetA$,
we let $\FmA(\FmB_1,\ldots,\FmB_n)$ denote the formula
$\AlgInterp{\FmA}{{\LangAlgA}}(\FmB_1,\ldots,\FmB_n)$.
Moreover, when $\Theta$ is a set of $n$-ary formulas,
we let $\Theta(\FmB_1,\ldots,\FmB_n) \SymbDef \{\FmA(\FmB_1,\ldots,\FmB_n) \mid \FmA \in \Theta\}$.



In what follows, we report the usual definitions of propositional logics as consequence relations over a propositional language.

\begin{definition}
\label{def:SetFmla}
A \emph{(finitary)} \SetFmla{} \emph{logic} is a
relation $\vdash \, \subseteq \, \PowerSet{\LangSetA} \times \LangSetA$
satisfying the following `Tarski'  conditions~\cite[p.~31]{Ta56}, for all $\FmSetA,\FmSetB,\FmSetC, \{\FmA,\FmB\} \subseteq \LangSetA$:
\begin{table}[H]
    \begin{tabular}{ll}
         (R)eflexivity & $\FmSetA, \FmA \vdash \FmA$ \\
         (M)onotonicity & if $\FmSetA \vdash \FmB$, then $\FmSetA, \FmSetB \vdash \FmB$ \\
         (T)ransitivity & if $\FmSetA,\FmSetC \vdash \FmB$ and $\FmSetA \vdash \FmC$ for each $\FmC \in \FmSetC$, then $\FmSetA \vdash \FmB$ \\
         (S)tructurality & if $\FmSetA \vdash \FmB$ and $\sigma \in \End(\LangAlgA)$, then $\sigma(\FmSetA) \vdash \sigma(\FmB)$ \\
         (F)initariness & if $\FmSetA \vdash \FmB$,
         then $\FmSetA^\prime \vdash \FmB$
         for some finite $\FmSetA^\prime \subseteq \FmSetA$
    \end{tabular}
\end{table}
\end{definition}

\begin{definition}
\label{def:SetSet}
A \emph{(finitary)} \SetSet{} \emph{logic} is a
relation $\SetSetCR \, \subseteq \, \PowerSet{\LangSetA} \times \PowerSet{\LangSetA}$
satisfying the following `Scott' conditions~\cite{shoesmithsmiley1978}, for all $\FmSetA,\FmSetA^\prime, \FmSetB, \FmSetB^\prime \subseteq \LangSetA$:
\begin{table}[H]
    \begin{tabular}{ll}
         (O)verlap & if $\FmSetA \cap \FmSetB \neq \EmptySet$, then $\FmSetA \SetSetCR \FmSetB$ \\
         (D)ilution & $\FmSetA \SetSetCR \FmSetB$, then $\FmSetA,\FmSetA^\prime \SetSetCR \FmSetB,\FmSetB^\prime$\\
         (C)ut for sets & if $\FmSetA, \SetCut \SetSetCR \FmSetComp{\SetCut}, \FmSetB$ for all $\SetCut \subseteq \LangSetA$, then $\FmSetA \SetSetCR \FmSetB$\\
         (S)tructurality & if $\FmSetA \SetSetCR \FmSetB$ and $\sigma \in \End(\LangAlgA)$, then $\sigma(\FmSetA) \SetSetCR \sigma(\FmSetB)$ \\
         (F)initariness & if $\FmSetA \SetSetCR \FmSetB$,
         then $\FmSetA^\prime \SetSetCR \FmSetB^\prime$
         for some finite $\FmSetA^\prime \subseteq \FmSetA$
         and $\FmSetB^\prime \subseteq \FmSetB$
    \end{tabular}
\end{table}
\end{definition}
\noindent Every \SetSet{} logic $\SetSetCR$ induces a \SetFmla{} logic $\vdash_{\SetSetCR}$
such that $\FmSetA \vdash_{\SetSetCR} \FmB$
if, and only if, $\FmSetA \SetSetCR \{\FmB\}$.
The complement of a \SetSet{} logic
$\SetSetCR{}$ is denoted by $\NSetSetCR{}$. 

A \emph{(deterministic logical) $\Sigma$-matrix} $\MatA$
is a structure $\langle \mathbf{A}, D \rangle$
where $\mathbf{A}$
is a $\Sigma$-algebra,
the universe $A$ is the set of \emph{truth values of $\MatA$},
and the members of $D \subseteq A$ are called \emph{designated values}. We will write $\overline{D}$ to refer to $A{\setminus{}}D$.
A matrix is \emph{finite} when its underlying algebra
is finite.
In this chapter, we will most of the time work with
the set $\ThreeValuesSet \SymbDef \{\bv, \uv, \tv\}$
of truth values.
The mappings in
$\Hom(\LangAlgA, \mathbf{A})$ are called \emph{$\MatA$-valuations}. 
Every $\Sigma$-matrix
determines a \SetSet{} logic
$\SetSetCR_\MatA$
such that $\FmSetA \SetSetCR_{\MatA} \FmSetB$
if{f} $h(\FmSetA) \cap \overline{D} \neq \varnothing$
or $h(\FmSetB) \cap D \neq \varnothing$
for every $\MatA$-valuation $h$,
as well as a \SetFmla{} logic
$\vdash_\MatA$
with $\FmSetA \vdash_\MatA \FmB$ iff $\FmSetA \SetSetCR_\MatA \{\FmB\}$.
Given a \SetSet{} logic~$\SetSetCR$ (resp.\ a \SetFmla{} logic~$\vdash$), 
if $\SetSetCR\;\subseteq\;\SetSetCR_{\MatA}$ (resp.\ $\vdash \;\subseteq\; \vdash_{\MatA}$), we shall say that~$\MatA$ \emph{is a model of}
$\SetSetCR$ (resp.\ $\vdash$),
and if the converse also holds
we shall say that~$\MatA$ \emph{characterises}
$\SetSetCR$ (resp.\ $\vdash$).

The pairs $(\FmSetA,\FmB) \in \PowerSet{\LangSetA} \times \LangSetA$ (resp. $(\FmSetA,\FmSetB) \in \PowerSet{\LangSetA} \times \PowerSet{\LangSetA}$) are called
\SetFmla{} \emph{statements} (resp. \SetSet{} \emph{statements}), and we say
that such statements hold according
to a \SetFmla{} logic $\vdash$ (resp. a \SetSet{} logic $\SetSetCR$) when
$\FmSetA \vdash \FmB$ (resp. when $\FmSetA \SetSetCR \FmSetB$).
In what follows, we use `H-calculus' for
``Hilbert-style calculus''
and `H-rule' for ``Hilbert-style rule of inference''.
A \emph{\SetFmla{} H-calculus} is a collection
of \SetFmla{} statements $(\FmSetA, \FmB)$, which, in this context,
are called \emph{\SetFmla{} H-rules}
and denoted by $$\SetFmlaHRule{\FmSetA}{\FmB}{},$$
being $\FmSetA$ the \emph{antecedent} (or \emph{set of premises}) and $\FmB$,
the \emph{succedent} (or \emph{conclusion}) of the rule.
H-rules with empty antecedents are called \emph{axioms}.
As usual, given a \SetFmla{} H-calculus $\SetFmlaHName$,
we say that the \SetFmla{} statement
$(\FmSetA, \FmB)$ is \emph{provable} in $\SetFmlaHName$
whenever there is a finite sequence of formulas
$\FmA_1,\ldots,\FmA_n$, where $\FmA_n = \FmB$
and
each $\FmA_i$ is either a premise (that is, a formula in $\FmSetA$) or follows
from previous formulas in the sequence by an
instance of a H-rule of $\SetFmlaHName$.
The relation $\vdash_{\SetFmlaHName} \, \subseteq \, \PowerSet{\LangSetA} \times \LangSetA$
defined such that $\FmSetA \vdash_\SetFmlaHName \FmB$
if, and only if, the statement $(\FmSetA, \FmB)$ is provable
in $\SetFmlaHName$ is easily seen to be a finitary
\SetFmla{} logic.
In Section~\ref{sec:MultipleConclusionCalculi},
we shall introduce \SetSet{} H-calculi~(based on \cite{shoesmithsmiley1978}), which generalize \SetFmla{} H-calculi by allowing for sets of formulas
to appear as succedents of the rules of inference.

\section{ Prominent three-valued logics}
\label{sec:CaseStudies}


For the most part of this chapter, we shall be dealing with logics that have an independent 
interest (i.e.~not purely formal but also philosophical or historical), having been introduced by renowned logicians with particular applications or motivations in mind: this is for instance the case of the time-honoured logical systems \LThreeName, \GThreeName{} and \SKName, which bear the names of  J.~\L ukasiewicz, K.~G\"odel and S. C. Kleene. 

In Section \ref{sec:nonmonadic}, we shall also consider a few logics whose interest is mostly theoretical, in that they will help us exemplify certain advantages or difficulties related to the \SetSet{} approach: this will be, in particular, the case of logics defined by non-monadic matrices, such as the implicative fragment of \L ukasiewicz logic.

A third criterion guiding our choice (partly overlapping with the first) has been to include the three-valued logics that appear in other chapters of the present book,  in particular in the contributions by A. Avron \& A. Zamansky~\cite{avronzamanssamebook} and by F. Paoli \& M. Pra Baldi~\cite{poaliprabaldisamebook}. 


\begin{table}
\centering
\begin{tabular}{llp{4cm}}
\toprule
Logic & $D$ & Interpretations \\
\midrule
\L ukasiewicz (\LThreeName) & $\DSetOne$ & $\lukConj$, $\lukDisj$, $\lukNeg$, $\lukImp$ \\
Logic of Paradox (\LPName) & $\DSetTwo$ & $\lukConj$, $\lukDisj$, $\lukNeg$ \\
Logic of Paradox with implication (\LPImpName) & $\DSetTwo$ & $\lukConj$, $\lukDisj$, $\lukNeg$, $\LPImp$ \\
Bochvar-Kleene (\BKName) & $\DSetOne$ & $\BKleeConj$, $\BKleeDisj$, $\lukNeg$, $\BKleeImp$ \\
Paraconsistent Weak Kleene (\PWKName) & $\DSetTwo$ & $\BKleeConj$, $\BKleeDisj$, $\lukNeg$, $\BKleeImp$ \\ 
Sette's (\SetteName)&$\DSetTwo$&$\setteImp$, $\setteNeg$\\
Sobociński (\SThreeName) & $\DSetTwo$ & 
$\lukNeg$, $\SobocImp$ \\
Da Costa-D'Ottaviano's (\JThreeName) & $\DSetTwo$ & $\lukConj$, $\lukDisj$, $\lukNeg$, $\JThreeImp$, $\LukPoss \SymbDef \setteNeg$\\
G\"odel (\GThreeName) & $\DSetOne$ & $\lukConj$, $\lukDisj$, $\HNeg$, $\HImp$ \\
Strong Kleene (\SKName) & $\DSetOne$ & $\lukConj$, $\lukDisj$, $\lukNeg$, $\KleeImp$ \\
Post's (\PostName)&$\DSetOne$&$\lukConj,\lukDisj,\postNeg$\\
\MSevenName & $\DSetOne$ & $\bullet_7$ \\ 
\MEightName & $\DSetOne$ & $\bullet_8$ \\
\bottomrule
\end{tabular}
\caption{Summary of the three-valued logics considered in this chapter.
Each row refers to a logic,
first informing its name,
then the set $D \subseteq \ThreeValuesSet$
of designated truth values
and the interpretations of
the connectives in the corresponding logical matrix, associated with
the truth tables presented
in Figure~\ref{fig:truth-tables2}.}
\label{tab:usecases}
\end{table}
%
%
%
%
%

\begin{framed}
\begin{figure}[H]
\begin{table}[H]
    \centering
    \begin{tabular}{@{}c|ccc@{}}
        \toprule
         \lukConj & \bv & \uv & \tv\\
        \midrule
         \bv & \bv & \bv & \bv\\
         \uv & \bv & \uv & \uv \\
         \tv & \bv & \uv & \tv\\
        \bottomrule
    \end{tabular}
    \quad
    \begin{tabular}{@{}c|ccc@{}}
        \toprule
         \lukDisj & \bv & \uv & \tv\\
        \midrule
         \bv & \bv & \uv & \tv\\
         \uv & \uv & \uv & \tv \\
         \tv & \tv & \tv & \tv\\
        \bottomrule
    \end{tabular}
    \quad 
    \begin{tabular}{@{}c|ccc@{}}
        \toprule
         \lukImp & \bv & \uv & \tv\\
        \midrule
         \bv & \tv & \tv & \tv\\
         \uv & \uv & \tv & \tv\\
         \tv & \bv & \uv & \tv\\
        \bottomrule
    \end{tabular}
    \quad 
    \begin{tabular}{@{}c|c@{}}
        \toprule
         & \lukNeg \\
         \midrule
         \bv & \tv\\
         \uv & \uv\\
         \tv & \bv\\
         \bottomrule
    \end{tabular}
\end{table}

\begin{table}[H]
    \centering
    \begin{tabular}{@{}c|ccc@{}}
        \toprule
         \KleeImp & \bv & \uv & \tv\\
        \midrule
         \bv & \tv & \tv & \tv\\
         \uv & \uv & \uv & \tv\\
         \tv & \bv & \uv & \tv\\
        \bottomrule
    \end{tabular}
    \quad
    \begin{tabular}{@{}c|ccc@{}}
        \toprule
         $\LPImp$ & \bv & \uv & \tv\\
        \midrule
         \bv & \tv & \tv & \tv\\
         \uv & \bv & \uv & \tv\\
         \tv & \bv & \uv & \tv\\
        \bottomrule
    \end{tabular}
\end{table}

\begin{table}[H]
    \centering
    \begin{tabular}{@{}c|ccc@{}}
        \toprule
         \BKleeConj & \bv & \uv & \tv\\
        \midrule
         \bv & \bv & \uv & \bv\\
         \uv & \uv & \uv & \uv \\
         \tv & \bv & \uv & \tv\\
        \bottomrule
    \end{tabular}
    \quad
    \begin{tabular}{@{}c|ccc@{}}
        \toprule
         \BKleeDisj & \bv & \uv & \tv\\
        \midrule
         \bv & \bv & \uv & \tv\\
         \uv & \uv & \uv & \uv \\
         \tv & \tv & \uv & \tv\\
        \bottomrule
    \end{tabular}
    \quad 
    \begin{tabular}{@{}c|ccc@{}}
        \toprule
         \BKleeImp & \bv & \uv & \tv\\
        \midrule
         \bv & \tv & \uv & \tv\\
         \uv & \uv & \uv & \uv\\
         \tv & \bv & \uv & \tv\\
        \bottomrule
    \end{tabular}
\end{table}

\begin{table}[H]
    \centering
    \begin{tabular}{@{}c|ccc@{}}
        \toprule
         \setteImp & \bv & \uv & \tv\\
        \midrule
         \bv & \tv & \tv & \bv\\
         \uv & \tv & \tv & \bv\\
         \tv & \tv & \tv & \tv\\
        \bottomrule
    \end{tabular}\quad
    \begin{tabular}{@{}c|c@{}}
        \toprule
         & \setteNeg \\
         \midrule
         \bv & \bv\\
         \uv & \tv\\
         \tv & \tv\\
         \bottomrule
    \end{tabular}\quad
    \begin{tabular}{@{}c|ccc@{}}
        \toprule
         \SobocImp & \bv & \uv & \tv\\
        \midrule
         \bv & \tv & \tv & \tv\\
         \uv & \bv & \uv & \tv\\
         \tv & \bv & \bv & \tv\\
        \bottomrule
    \end{tabular}
\end{table}

\begin{table}[H]
    \centering
        \begin{tabular}{@{}c|ccc@{}}
        \toprule
         \JThreeImp & \bv & \uv & \tv\\
        \midrule
         \bv & \tv & \tv & \tv\\
         \uv & \bv & \uv & \tv\\
         \tv & \bv & \uv & \tv\\
        \bottomrule
    \end{tabular}
%
\quad
    \begin{tabular}{@{}c|ccc@{}}
        \toprule
             \HImp & \bv & \uv & \tv\\
        \midrule
         \bv & \tv & \tv & \tv\\
         \uv & \bv & \tv & \tv\\
         \tv & \bv & \uv & \tv\\
        \bottomrule
    \end{tabular}
    \quad 
    \begin{tabular}{@{}c|c@{}}
        \toprule
         & \HNeg \\
         \midrule
         \bv & \tv\\
         \uv & \bv\\
         \tv & \bv\\
         \bottomrule
    \end{tabular}
\end{table}
\begin{table}[H]
    \centering
    \begin{tabular}{@{}c|c@{}}
        \toprule
         & \postNeg \\
         \midrule
         \bv & \tv\\
         \uv & \bv\\
         \tv & \uv\\
         \bottomrule
    \end{tabular}\quad
    \begin{tabular}{c|ccc}
        \toprule
         $\bullet_7$ & $\bv$ & $\uv$ & $\tv$ \\
         \midrule
         $\bv$& \uv & \tv & \tv\\
         $\uv$& \tv& \tv&\tv\\
         $\tv$& \uv& \tv&\tv\\
         \bottomrule
    \end{tabular}
    \quad
    \begin{tabular}{c|ccc}
        \toprule
         $\bullet_8$ & $\bv$ & $\uv$ & $\tv$ \\
         \midrule
         $\bv$& \tv & \tv & \tv\\
         $\uv$& \tv& \tv&\tv\\
         $\tv$& \uv& \tv&\tv\\
         \bottomrule
    \end{tabular}
\end{table}
\caption{Three-valued interpretations for the connectives of the logics listed in Table~\ref{tab:usecases}.}
\label{fig:truth-tables2}
\end{figure}
\end{framed}

The three-valued logics of interest
are all listed in Table \ref{tab:usecases}.
We provide below a more detailed
description of them, retrieving from the literature a 
(not necessarily finite or analytic) 
\SetFmla{} H-calculus, when available.

\paragraph{\L ukasiewicz (\L$_3$).}

The logic \LThreeName{} was introduced 
more than a century ago (1920)
by Jan
\L ukasiewicz,
in a short paper~\cite{lukz1920} written in Polish 
(for an English translation, see~\cite{lukz1920en}).
 \L ukasiewicz's 
intention was to challenge classical logic (which is based on the Aristotelian principle of contradiction) by
introducing a third truth value meant to represent propositions that are neither false nor true but only \emph{possible}.
Although there is nowadays consensus that \LThreeName{} does not successfully model the logic of possibility (this task having been taken over by modal logic), \L ukasiewicz's proposal was later generalized to  $n$-valued logics (for each $n < \omega$)
and even to an infinite-valued one,
leading over time to a huge research output in  the area of mathematical logic  (more recently, in particular, in the mathematical fuzzy logic community:~see e.g.~\cite{hajek98}).

Various inter-definabilities hold among the  
connectives of \L ukasiewicz logics, which in many presentations also include
a 
\emph{multiplicative disjunction} ($\oplus$)
and a \emph{multiplicative conjunction} ($\odot$)
given, in terms of our primitive connectives, by
$\FmA \oplus \FmB : = \neg \FmA \to \FmB$
and $\FmA \odot \FmB := \neg (\FmA \to \neg \FmB)$. 
Every logic in the  \L ukasiewicz family can be presented
in the language 
(first used by 
M.~Wajsberg to axiomatize the
infinite-valued \L ukasiewicz logic)
having only
the implication $(\to)$ and the negation $(\neg)$ as primitive: one defines
$\FmA \lor \FmB : = (\FmA \to \FmB) \to \FmB$
and, as in \SKName{} below,
$\FmA \land \FmB := \neg (\neg \FmA \lor \neg \FmB)$. The implication,
in turn, can be given by
$\FmA \to \FmB := \neg \FmA \oplus \FmB$
but is not 
definable by the classical formula
$\neg \FmA \lor \FmB$ nor, indeed, by any formula constructed from the remaining connectives   chosen for our presentation (conjunction, the disjunction and the negation). 

The following is an adequate (sound and complete) 
calculus for \LThreeName{}, which can be found in~\cite{Avr91}:

\begin{table}[H]
\begin{tabular}{l}
$\SetFmlaHRule{}{\FmA \ararr (\FmB \ararr \FmA)}{(\LThreeName1)}$\\[.3cm]
$\SetFmlaHRule{}{(\FmA \ararr \FmB) \ararr ((\FmB \ararr \FmC) \ararr (\FmA \ararr \FmC))}{(\LThreeName2)}$\\[.3cm]
$\SetFmlaHRule{}{(\FmA \ararr (\FmB \ararr \FmC)) \ararr (\FmB \ararr (\FmA \ararr \FmC))}{(\LThreeName3)}$\\[.3cm]
$\SetFmlaHRule{}{((\FmA \ararr \FmB) \ararr \FmB) \ararr ((\FmB \ararr \FmA) \ararr \FmA)}{(\LThreeName4)}$\\[.3cm]
$\SetFmlaHRule{}{((((\FmA \ararr \FmB) \ararr \FmA) \ararr \FmA) \ararr (\FmB \ararr \FmC)) \ararr (\FmB \ararr \FmC)}{(\LThreeName5)}$\\[.3cm]
$\SetFmlaHRule{}{\FmA \aand \FmB \ararr \FmA}{(\LThreeName6)}$\\[.3cm]
$\SetFmlaHRule{}{\FmA \aand \FmB \ararr \FmB}{(\LThreeName7)}$\\[.3cm]
$\SetFmlaHRule{}{(\FmA \ararr \FmB) \ararr ((\FmA \ararr \FmC) \ararr (\FmA \ararr \FmB \aand \FmC))}{(\LThreeName8)}$\\[.3cm]
$\SetFmlaHRule{}{\FmA \ararr (\FmA \aor \FmB)}{(\LThreeName9)}$\\[.3cm]
$\SetFmlaHRule{}{\FmB \ararr (\FmA \aor \FmB)}{(\LThreeName10)}$\\[.3cm]
$\SetFmlaHRule{}{(\FmA \ararr \FmC) \ararr ((\FmB \ararr \FmC) \ararr (\FmA \aor \FmB \ararr \FmC))}{(\LThreeName11)}$\\[.3cm]
$\SetFmlaHRule{}{(\aneg \FmB \ararr \aneg \FmA) \ararr (\FmA \ararr \FmB)}{(\LThreeName12)}$\\[.3cm]
$\SetFmlaHRule{\FmA \quad  \FmA \ararr \FmB}{\FmB}{(\LThreeName13)}$
\end{tabular}
\end{table}

\paragraph{Strong Kleene (\SKName{}).}

In his  book~\cite{Kl50},
the prominent mathematician 
S.C.~Kleene introduced two
three-valued logics that nowadays bear his name.
The first  we consider
($\SKName$) is defined by the same truth tables as \LThreeName{} 
as far as the connectives $\land, \lor$ and $\neg$
are concerned, whereas it differs in the implication  (usually not included in the primitive language), which may be given by
$\FmA \to \FmB : = \neg \FmA \lor \FmB$.
Thus, formally, $\SKName$ is
easily shown to coincide with the $\{ \land, \lor, \neg \}$-fragment of \LThreeName{}. The meaning
 Kleene had in mind for the third value is, however, different from the \L ukasiewicz case, for $\uv$  is now interpreted as \emph{undefined} (rather than \emph{possible}, as in \LThreeName{}). Such an interpretation can be justified within the theory of  partial recursive functions, to which Kleene in fact gave  major contributions. 
 
 The notion of partial function generalizes that of ordinary function in that the former may be undefined for some arguments, that is, we admit   that no output may be produced
for certain inputs. Thus, for a given function $f$ on integers and a given input $x$, we have three cases: either $f(x)= 5$ or 
$f(x) \neq 5$ or $f(x)$ is undefined; within $\SKName$,
a proposition corresponding to the latter situation would be labelled by the  truth value $\uv$. Notice also that the truth tables of $\SKName$ reflect the following observation: even if the value of some atomic proposition (say $p$) is undefined, we may still be able to assign a classical value to a compound proposition 
(say $p \lor q$) in which $p$ appears, for it suffices to know that $q$ is true. This reasoning also makes perfect sense from a computational point of view: a program that computes the value of $p \lor q$ may output 
the value $\tv$ as soon as it successfully manages to establish that (e.g.) $q$ has value $\tv$, even if the computation involving $p$ has not terminated yet. 

A convenient way to axiomatize
 $\SKName$ (as well as $\LPName$ considered further on) is to start from a common weakening of both logics, namely four-valued Belnap-Dunn logic, \textbf{B}. The latter (on which we shall say more in a little while) is axiomatized by the following \SetFmla{} H-calculus
 introduced in \cite{Pyn95} (an almost identical calculus was independently introduced in \cite{Fon97}):


\begin{table}[H]
\centering
\setlength\tabcolsep{1.5pt}
\begin{tabular}{lll}
\AXC{$\FmA \aand \FmB$}
\LL{($\SKName1$)}
\UIC{$\FmA$}
\DP 
 & 
\AXC{$\FmA \aand \FmB$}
\LL{($\SKName2$)}
\UIC{$\FmB \aand \FmA$}
\DP
 &
\AXC{$\FmA$}
\AXC{$\FmB$}
\LL{($\SKName3$)}
\BIC{$\FmA \aand \FmB$}
\DP
 \\

 & & \\
 
\AXC{$\FmA$}
\LL{($\SKName4$)}
\UIC{$\FmA \aor \FmB$}
\DP 
 & 
\AXC{$\FmA \aor \FmB$}
\LL{($\SKName5$)}
\UIC{$\FmB \aor \FmA$}
\DP
 & 
\AXC{$\FmA \aor (\FmB \aor \FmC)$}
\LL{($\SKName6$)}
\UIC{$(\FmA \aor \FmB) \aor \FmC$}
\DP
 \\
 
 & & \\

\AXC{$\FmA \aor (\FmB {\aand} \FmC)$}
\LL{($\SKName7$)}
\UIC{$(\FmA \aor \FmB) \aand (\FmA \aor \FmC)$}
\DP
 &
\AXC{$(\FmA \aor \FmB) \aand (\FmA \aor \FmC)$}
\LL{($\SKName8$)}
\UIC{$\FmA \aor (\FmB {\aand} \FmC)$}
\DP
 & 
\AXC{$\FmA \aor \FmA$}
\LL{($\SKName9$)}
\UIC{$\FmA$}
\DP
\\

 & & \\
 
\AXC{$\FmA \aor \FmB$}
\LL{($\SKName10$)}
\UIC{$\aneg\aneg \FmA \aor \FmB$}
\DP
 & 
\AXC{$\aneg\aneg \FmA \aor \FmB$}
\LL{($\SKName11$)}
\UIC{$\FmA \aor \FmB$}
\DP
 &
\AXC{$(\aneg \FmA \aand \aneg \FmB) \aor \FmC$}
\LL{($\SKName12$)}
\UIC{$\aneg (\FmA \aor \FmB) \aor \FmC$}
\DP
 \\

 & & \\
 
\AXC{$\aneg (\FmA \aor \FmB) \aor \FmC$}
\LL{($\SKName13$)}
\UIC{$(\aneg \FmA \aand \aneg \FmB) \aor \FmC$}
\DP
 & 
\AXC{$(\aneg \FmA \aor \neg \FmB) \aor \FmC$}
\LL{($\SKName14$)}
\UIC{$\aneg (\FmA \aand \FmB) \aor \FmC$}
\DP
 & 
\AXC{$\aneg (\FmA \aand \FmB) \aor \FmC$}
\LL{($\SKName15$)}
\UIC{$(\aneg \FmA \aor \neg \FmB) \aor \FmC$}
\DP
\\
\end{tabular}
\end{table}

\noindent 
As shown in~\cite{AlbPreRiv17}, strong Kleene logic may be obtained by adding to the preceding calculus the following rule: 

\begin{table}[H]
\begin{tabular}{lll}
\centering
\setlength\tabcolsep{1.5pt}
\AXC{$(\FmA \aand \aneg \FmA) \aor \FmB$}
\LL{($\SKName16$)}
\UIC{$\FmB$}
\DP&&
\end{tabular}
\end{table}

\noindent Notice that $\SKName$ does not have theorems, so it cannot be presented via a calculus with closed assumptions (i.e.~axioms). 

\paragraph{Logic of Paradox (\LPName{}).}

 This logic
was introduced 
(under the name \emph{calculus of antinomies}) in the 1960s by F.~G.~Asenjo~\cite{asenjo1966},
but was studied in greater depth and made popular a decade later by G.~Priest, who also proposed the name we are adopting here: \emph{Logic of Paradox}.
The intention of both authors was to develop a logical framework that might allow one to handle logical paradoxes 
(also known as \emph{antinomies}) without trivializing the whole system, as happens in the classical setting. 
Antinomies are propositions that, like the famous Russell's paradox, appear to be \emph{both true and false} at the same time. From the point of view of a particular logic,
we can say that an antinomy is a proposition $p$ such that both $p$ and $\neg p$ assume a designated value. The logic devised by Asenjo-Priest is therefore 
(a most prominent example of) a
\emph{paraconsistent logic}, 
i.e.~one that rejects the following \emph{ex contradictione quodlibet} rule:

\begin{table}[H]
\begin{tabular}{lll}
\centering
\setlength\tabcolsep{1.5pt}
\AXC{$\FmA \aand \aneg \FmA$}
\LL{($\mathbf{ECQ}$)}
\UIC{$\FmB$}
\DP&&
\end{tabular}
\end{table}
Semantically, achieving such a behavior is easy: one may just adopt the
same truth tables as Kleene's 
\SKName{} while having $\uv$
designated  together with
$\tv$. This simple strategy, while necessarily renouncing some derivations that hold classically, allows \LPName{} to actually
retain all the tautologies of 
classical logic (see below).

As mentioned earlier, the Belnap-Dunn logic $\BelnapDunnName$ is, formally, a  weakening common to  
\LPName{} and \SKName{}. 
It may be interesting to mention that,
in fact, the four-valued matrix which defines $\BelnapDunnName$ is  obtained by simultaneously enlarging 
the set $\{ \bv, \tv \}$ with 
two non-classical truth values: the former (usually called \textbf{n}, for \emph{neither} true nor false)  behaves as $\uv$ does in \SKName{}
(\textbf{n} is a fixpoint of the negation operator, but is not designated) and the latter 
(usually called \textbf{b}, for \emph{both} true and false) 
 behaves as $\uv$ does in \LPName{}
(\textbf{b} is also a fixpoint of the negation operator but it is designated).

As shown in~\cite[Thm.~3.4]{AlbPreRiv17},  \LPName{} can be axiomatized through the \SetFmla{} H-calculus extending $\BelnapDunnName$ by the \emph{exluded middle} axiom:

\begin{table}[H]
    \begin{tabular}{l}
         $\SetFmlaHRule{}{\FmA \lor \neg\FmA}{(\LPName1)}$
    \end{tabular}
\end{table}

\noindent As mentioned earlier, the tautologies of \LPName{} are precisely the ones of classical logic \cite[Thm.~III.8,~p.~228]{Priest}, so classical logic and \LPName{} only differ as far as proper inferences (that is, those with nonempty sets of premises) are concerned.


\paragraph{Logic of Paradox with implication ($\LPImpName{}$).} 

As in
Kleene's 
\SKName{}, one can define within \LPName{} an implication connective
by the term 
$\neg \FmA \lor \FmB$. Such an implication, however,
will not enjoy the Deduction Theorem,  and indeed it will not even satisfy \emph{modus ponens} (adding the latter rule to \LPName{} is already sufficient to regain classical logic). A stronger result is shown in~\cite[Prop.~8]{arieli2011}:
 no  connective definable in \LPName{} 
can possibly enjoy the Deduction Theorem. It is, however,
possible to expand  \LPName{}
with a new implication that enjoys
the Deduction Theorem and still does not force the resulting logic to become classical. This logic,
which we shall denote by 
\LPImpName{},
seems to have been first considered in~\cite{batens1980}.
A \SetFmla{} H-calculus for  
\LPImpName{} may be obtained as follows.
Let us start from the system
HB$_e$ from~\cite{Avr91}:



\begin{table}[H]
\begin{tabular}{l}
$\SetFmlaHRule{}{\FmA \ararr (\FmB \ararr \FmA)}{(\LPImpName1)}$\\[0.3cm]
$\SetFmlaHRule{}{(\FmA \ararr (\FmB \ararr \FmC)) \ararr ((\FmA \ararr \FmB) \ararr (\FmA\ararr \FmC))}{(\LPImpName2)}$\\[0.3cm]
$\SetFmlaHRule{}{((\FmA \ararr \FmB) \ararr \FmA) \ararr \FmA}{(\LPImpName3)}$\\[0.3cm]
$\SetFmlaHRule{}{(\FmA \aand \FmB) \ararr \FmA}{(\LPImpName4)}$\\[0.3cm]
$\SetFmlaHRule{}{(\FmA \aand \FmB) \ararr \FmB}{(\LPImpName5)}$\\[0.3cm]
$\SetFmlaHRule{}{\FmA \ararr (\FmB \ararr (\FmA \aand \FmB))}{(\LPImpName6)}$\\[0.3cm]
$\SetFmlaHRule{}{\FmA \ararr (\FmA \aor \FmB)}{(\LPImpName7)}$\\[0.3cm]
$\SetFmlaHRule{}{\FmB \ararr (\FmA \aor \FmB)}{(\LPImpName8)}$\\[0.3cm]
$\SetFmlaHRule{}{(\FmA \ararr \FmC) \ararr ((\FmB \ararr \FmC) \ararr ((\FmA \aor \FmB) \ararr \FmC))}{(\LPImpName9)}$\\[0.3cm]
$\SetFmlaHRule{}{\aneg (\FmA \aor \FmB) \alrarr \aneg \FmA \aand \aneg \FmB}{(\LPImpName10)}$\\[0.3cm]
$\SetFmlaHRule{}{\aneg (\FmA \aand \FmB) \alrarr \aneg \FmA \aor \aneg \FmB}{(\LPImpName11)}$\\[0.3cm]
$\SetFmlaHRule{}{\aneg \aneg \FmA \alrarr \FmA}{(\LPImpName12)}$\\[0.3cm]
$\SetFmlaHRule{}{\aneg (\FmA \ararr \FmB) \alrarr \FmA \aand \aneg \FmB}{(\LPImpName13)}$\\[0.3cm]
$\SetFmlaHRule{\FmA \quad \FmA \ararr \FmB}{\FmB}{(\LPImpName14)}$
\end{tabular}
\end{table}

\noindent The logic \LPImpName{} can be obtained from HB$_e$ by adding either of the following axioms (see \cite[Sec.~4]{Avr91}, where \LPImpName{} is called \textbf{Pac}).

\begin{table}[H]
    \begin{tabular}{l}
        $\SetFmlaHRule{}{\FmA \lor \neg\FmA}{(\LPImpName15)}$ \\[0.3cm]
        $\SetFmlaHRule{}{(\FmB \to \FmA) \to ((\FmB \to \neg \FmA) \to \neg \FmB)}{(\LPImpName16)}$
    \end{tabular}
\end{table}

\paragraph{Paraconsistent Weak Kleene (\PWKName).}

This logic seems to have been  considered already by S.~Halld\'en
in his 1949 monograph~\cite{hallden1949}, 
and two decades later by A.~Prior~\cite{prior1967}, 
but has only recently been studied in depth (see e.g.~\cite{BonGilPaoPer17}, 
and also~\cite{pailos2018} which
explores applications to the theory of truth).

The connectives and  truth tables of \PWKName{} coincide with those of \BKName{} (see the next subsection), and so does the interpretation of the middle value. Logical consequence, however, is defined in \PWKName{}
in terms of preservation of non-falsity rather than (as in \BKName{}) of preservation of truth. This  choice leads to designating $\{ \tv, \uv \}$, which entails (as we have seen with \LPName{}) that  \PWKName{} is indeed a paraconsistent logic. 
In fact the valid formulas of
\PWKName{} coincide with those of 
\LPName{} which, as observed earlier, are precisely the classical tautologies
(see e.g.~\cite[p.~257]{BonGilPaoPer17}).
However, 
\PWKName{}  and \LPName{} differ in their derivations. To see this,
let us recall that 
\PWKName{} can be axiomatized 
as follows.
Take any complete \SetFmla{} H-calculus for classical logic (having \emph{modus ponens} as its only non-nullary rule)
and, while keeping all the axioms,
replace \emph{modus ponens}
by the following restricted version (see~\cite[Prop.~4]{BonGilPaoPer17}):

\begin{table}[H]
\begin{tabular}{l}
$\SetFmlaHRule{\FmA \quad \FmA \ararr \FmB}{\FmB}{(\PWKName1)}$ \ \ provided  $\Props{\FmA} \subseteq \Props{\FmB}$. 
\end{tabular}
\end{table}

This is in keeping with the observation that 
\PWKName{} is the \emph{left variable inclusion companion}
of classical logic (in the sense of~\cite[pp.~49-50]{bonzio2021plonka}, to which we refer for the relevant definitions and further details); likewise,
the logic \BKName{} dealt with in the sequel is the right variable inclusion companion of classical logic.

Using the preceding characterization,
one sees that \PWKName{}
validates, for instance, the inference
$\FmA \land \neg \FmA \vdash  \FmA  \land \FmB$
which is easily refutable 
in  \LPName{}. On the other hand,
\LPName{} satisfies all rules of 
the Belnap-Dunn logic \textbf{B},
while 
it is easy to see that \PWKName{} does not satisfy, for instance, the rule
($\SKName1$) shown earlier. Hence, \LPName{}
and \PWKName{} are, as consequence relations, mutually  incomparable.


As shown in~\cite[Example~26]{bonzio2021plonka},
\PWKName{}
can also be axiomatized via 
a \SetFmla{} H-calculus having no variable restriction on the applicability of rules; this  system, however, involves countably many rule schemata. The situation is thus 
similar to the one of \BKName{},
and also in this case it is an open 
question whether \PWKName{} 
can be axiomatized by a finite H-calculus.
And just like in that case the divide between \SetFmla{} and \SetSet{} is visible,
after many decades we still do not know if these logics are finitely axiomatizable in \SetFmla{} but a finite \SetSet{} analytic axiomatization (in the sense explained in Section~\ref{sec:MultipleConclusionCalculi}) can easily be obtained,
see \cite{caleiro2020}.


 
\paragraph{Bochvar-Kleene ($\BKName$).}

This logic was introduced in 1937 by Bochvar~\cite{bochvar1981}, 
who had in mind 
potential 
applications to
the study of paradoxes, future contingent statements
and presuppositions (see e.g.~\cite{ferguson2014}
for a more recent computational interpretation of \BKName{}).
The third value is intended to
represent nonsensical statements
(or corrupted data, in the  interpretation given by Kleene for his equivalent logic),
and has an \emph{infectious}
behaviour: any interaction 
of $\uv$ with the other
truth values delivers
 $\uv$ itself.

It is worth clarifying that 
Bochvar's original presentation of \BKName{} employs
two different propositional languages, an  \emph{external} 
and an \emph{internal} one. We shall here consider the logic  of the internal language (shown in Figure~\ref{fig:truth-tables2}),
which is functionally equivalent to Kleene's system of
`weak' connectives~\cite{Kl50}. In Bochvar's paper,
the external connectives belong
to a metalanguage employed to assert or deny propositions which are built using the internal language.
Thus, for instance, according to the truth table of the internal conjunction $\BKleeConj$,
the value of a proposition such as  $\FmA \land \FmB$ may be
$\tv$,  $\uv$ or $\bv$.
However,
the value of the metalinguistic affirmation of 
$\FmA \land \FmB$ (this is denoted $\vdash \FmA \land \FmB$ by Bochvar)
may only be classical,
and  it will be $\tv$ precisely when the value of $\FmA \land \FmB$ is $\tv$, and $\bv$ otherwise. The external conjunction of two propositions $\FmA$ and $\FmB$ is interpreted as the affirmation
``$\vdash \FmA$ and $\vdash \FmB$ '',
which follows the usual Boolean truth tables. The semantics of external disjunction, conjunction and implication is defined in a similar way.

Although we are here  exclusively concerned with the internal connectives,
the preceding explanation was necessary to highlight the following fact: while (finite) \SetFmla{} H-calculi for \BKName{} presented in the external language  (and even for $n$-valued generalizations of \BKName{}: see~\cite{grigolia2008}) are well known, 
the first  \SetFmla{} H-calculus axiomatizing \BKName{} in the pure internal language is the one 
displayed below, which includes countably many rule schemata. 
This axiomatization was 
introduced in~\cite{bonzio2021}, in which
\BKName{} is studied from an algebraic
perspective as the right variable inclusion companion of classical logic (see also the chapter  by Paoli \& Pra Baldi~\cite{poaliprabaldisamebook} in this volume).
\begin{table}[H]
    \centering
    \begin{tabular}{ll}
        \AXC{$\FmC \ast (\FmA \to \FmA)$}
        \LL{(\BKName1)}
        \UIC{$\FmA \to \FmA$}
        \DP
        &
        \AXC{$\FmC \ast (\FmA \to (\FmB \to \FmA))$}
        \LL{(\BKName2)}
        \UIC{$\FmA \to (\FmB \to \FmA)$}
        \DP\\[0.4cm]
        \multicolumn{2}{l}{
        \AXC{$\FmC \ast (\FmA \to (\FmB \to \FmD) \to (\FmA \to \FmB) \to (\FmA\to\FmD))$}
        \LL{(\BKName3)}
        \UIC{$\FmA \to (\FmB \to \FmD) \to (\FmA \to \FmB) \to (\FmA\to\FmD)$}
        \DP}\\[0.4cm]
        \AXC{$\FmC \ast ((\neg\FmA \to \neg\FmB) \to (\FmB\to\FmA))$}
        \LL{(\BKName4)}
        \UIC{$(\neg\FmA \to \neg\FmB) \to (\FmB\to\FmA)$}
        \DP
        &
        \AXC{$\FmA \ast \FmB, \FmA \to \FmB$}
        \LL{(\BKName5)}
        \UIC{$\FmB$}
        \DP\\[0.4cm]
        \AXC{$\FmA, \FmB$}
        \LL{(\BKName6)}
        \UIC{$\FmA \ast \FmB$}
        \DP
        &
        \AXC{$\FmA \ast \FmB$}
        \LL{(\BKName7)}
        \UIC{$\FmA$}
        \DP
        \\[0.4cm]
        \AXC{$\FmA, \neg\FmA$}
        \LL{(\BKName8)}
        \UIC{$\FmB$}
        \DP
        &
        \AXC{$\chi(\delta, \vec{z})$}
        \doubleLine\LL{(\BKName9)}
        \UIC{$\chi(\epsilon, \vec{z})$}
        \DP,
        for every $\delta \approx \epsilon \in \Theta$\\
    \end{tabular}
\end{table}
\noindent where $\FmA \ast \FmB \SymbDef \FmA \land (\FmA \lor \FmB)$ and $\Theta \SymbDef \{
            \FmA \ast \FmA \approx \FmA,
            \FmA \ast (\FmB \ast \FmC) \approx (\FmA \ast \FmB) \ast \FmC,
            \FmA \ast (\FmB \ast \FmC) \approx \FmA \ast (\FmC \ast \FmB),
            \conn(\FmA_1,\ldots,\FmA_k) \ast \FmB \approx
            \conn(\FmA_1 \ast \FmB, \ldots, \FmA_k \ast \FmB),
            \FmB \ast \conn(\FmA_1,\ldots, \FmA_k) \approx \conn(\FmB \ast \FmA_1, \ldots, \FmB\ast \FmA_k)
        \}$.

As with \PWKName{}, it is currently an open problem whether \BKName{} can or cannot be axiomatized by  
a finite \SetFmla{} H-calculus. 
In fact,
 \BKName{}  and \PWKName{} are both examples that illustrate the gap that exists between the \SetFmla{} and the \SetSet{}
approach to the axiomatization of logics.

Using the \SetSet{} formalism, it is easy to produce a finite axiomatization for \BKName{}
as a dual of the calculus for
\PWKName{} presented in the preceding section obtained by reversing all rules and swapping all occurrences of $\land$ and $\lor$. This result can be established thanks to a construction introduced in~\cite{caleiro2020}, which allows one to update the matrix semantics
of a given logic (by adding infectious values), 
obtaining a semantics for its left and right variable inclusion companions so as to preserve finite-valuedness.
Under a certain expressiveness requirement on the resulting matrix, this construction can then be used (together with the results from~\cite{marcelino2019,marcelinowollic19}) to obtain finite analytic \SetSet{} axiomatizations for the variable inclusion companions of the logic under consideration.
In the case of \BKName{},
the
claimed result then follows from the  fact that \BKName{}
is the right variable inclusion companion of classical logic~\cite[Thm.~4, p.~258]{urquhart2001}. 


{\color{red}

}

{\color{red}
 }

\paragraph{Sobociński ($\SThreeName$)}

This logic was introduced in 1952 by  B.~Sobociński~\cite{sobocinski1952}  
as an (admittedly oversimplified) example of a logic possessing a ``strict implication'' analogous to that of more widely-known (non-finitely valued) modal systems. The propositional language of \SThreeName{} consists just of the (\L ukasiewicz) negation
and the above-mentioned implication, which does not coincide with any of the implications of the other logics considered in the present chapter.
Given the mostly theoretical interest in  \SThreeName{},
Sobociński does not propose an interpretation for the truth values; from a formal point of view, however, \SThreeName{} 
is closely related to relevance logics,
and in particular to the one known as \textbf{R}-Mingle (see~\cite{dunn1970}). 
In fact, as shown in~\cite{parks1972},
the set of tautologies of
\SThreeName{} coincides with those of the implication-negation fragment of 
\textbf{R}-Mingle.



The following \SetFmla{} Hilbert-style axiomatization 
of $\SThreeName$ was presented in~\cite{sobocinski1952}
and further explored in~\cite{parks1972}:
\begin{table}[H]
    \begin{tabular}{l}
        \AXC{}
        \LL{(\SThreeName1)}
        \UIC{$(\FmA \to \FmB) \to ((\FmB \to \FmC) \to (\FmA \to \FmC))$}
        \DP
        \\[0.3cm]
        \AXC{}
        \LL{(\SThreeName2)}
        \UIC{$\FmA \to (\FmA \to (\FmB \to \FmB))$}
        \DP
        \\[0.3cm]
        \AXC{}
        \LL{(\SThreeName3)}
        \UIC{$(\FmA \to (\FmA\to\FmB))\to(\FmA \to \FmB)$}
        \DP
        \\[0.3cm]
        \AXC{}
        \LL{(\SThreeName4)}
        \UIC{$\FmA \to (\FmB \to (\neg {\FmB} \to \FmA))$}
        \DP
        \\[0.3cm]
        \AXC{}
        \LL{(\SThreeName5)}
        \UIC{$(\neg \FmA \to \neg \FmB) \to (\FmB \to \FmA)$}
        \DP
        \\[0.3cm]
        \AXC{$\FmA \quad \FmA \to \FmB$}
        \LL{(\SThreeName6)}
        \UIC{$\FmB$}
        \DP
    \end{tabular}
\end{table}

%

\paragraph{Da Costa and D'Ottaviano's (\JThreeName).}

This logic
was proposed in 1970 by N.~da Costa and I.M.L. D'Ottaviano~\cite{dottaviano1987} as a formal solution to a question originally posed by Stanis\l aw Ja\'skowski  (hence  the name \JThreeName{}), who is   one of the  earliest pioneers of paraconsistent logic (together with N.~da Costa himself).
As in the case of \LPName{}, the middle value is designated: this reflects the  observation that \JThreeName{}
is also paraconsistent, in that both a proposition and its negation may be assigned a designated value. The authors of~\cite{dottaviano1987} give two possible interpretations for propositions that are evaluated to the middle value: we can either regard them as \emph{true} (in which case we have to admit, with G.~Priest, that there may be ``true contradictions'') or as \emph{provisionally true}. In the latter case it is assumed that contradictions may exist within a theory at a certain stage of its development, but  that these  will be hopefully eliminated later on. 
This reasoning justifies the presence in \JThreeName{}
of a modal-like ``possibility'' connective $\Diamond$ and its truth table: the logic is optimistic in
that $\Diamond \uv = \tv$, that is to say, the proposition ``it is possible that $A$'' is true (has value $\tv$) when 
$p$  is provisionally true
($\uv$). 

Taking the conjunction, the negation and $\Diamond$
as primitive,
one defines within \JThreeName{} 
the remaining connectives 
as follows:
\begin{itemize}
\item[] $\FmA \aor \FmB \SymbDef \aneg (\aneg \FmA \aand \aneg \FmB)$ 
\item[] $\FmA \ararr \FmB \SymbDef \aneg \wdia \FmA \aor \FmB$ 
\item[] $\FmA \alrarr \FmB \SymbDef (\FmA \ararr \FmB) \aand (\FmB \ararr \FmA)$
\item[] $\circ \FmA \SymbDef \aneg (\wdia \FmA \aand \wdia \aneg \FmA)$ 
\end{itemize}
As our notation suggests, the above-defined connective $\circ$ used in our axiomatization is an example of what in paraconsistent logic  is known
as a \emph{consistency operator}:
this means that the value of $\circ \FmA$ is designated (in which case, it must be $\tv$) if and only if the value of $\FmA$ is classical (see e.g.~\cite{carnielli2000}). 

It is easy to verify that, formally, all the connectives of 
\JThreeName{} either coincide with those of \LThreeName{} (as in the case of $\land, \lor$ and $\neg$) or are definable through them, for 
$\Diamond$ is itself definable in \LThreeName{}
by 
$\Diamond \FmA : = \neg \FmA \to_{\LThreeName} \FmA$.
Conversely, the \L ukasiewicz implication (and hence all the connectives of \LThreeName{}) is definable in \JThreeName{}
by: 
$$
\FmA \to_{\LThreeName} \FmB := 
(\Diamond \neg \FmA \lor \FmB) 
\land 
(\neg \FmA \lor \Diamond \FmB).
$$
This is shown in the MSc thesis of  J.~Marcos~\cite[p.~58]{marcos1999},
who must also be credited for bringing this fact to our attention. 

We may therefore affirm that 
the matrices defining 
\LThreeName{} and \JThreeName{}
are equivalent up to propositional language;
however, the  different 
choices of designated elements
determine two distinct logics.
To see this, it suffices to observe that 
 \LThreeName{} (alone)
validates 
$\FmA \vdash \neg \Diamond \neg \FmA$
whereas $\FmA \lor \neg \FmA$  is a tautology (only) in \JThreeName{}
(see e.g.~\cite[Example~3.20]{fontaal2016}). 

A \SetFmla{} H-calculus axiomatizing \JThreeName{} was introduced in~\cite{EpsDOt90}:

\begin{table}[H]
\begin{tabular}{l}
$\SetFmlaHRule{}{\FmA \ararr (\FmB \ararr \FmA)}{(\JThreeName1)}$\\[0.3cm]
$\SetFmlaHRule{}{(\FmA \ararr (\FmB \ararr \FmC)) \ararr ((\FmA \ararr \FmB) \ararr (\FmA\ararr \FmC))}{(\JThreeName2)}$\\[0.3cm]
$\SetFmlaHRule{}{((\FmA \ararr \FmB) \ararr \FmA) \ararr \FmA}{(\JThreeName3)}$\\[0.3cm]
$\SetFmlaHRule{}{\aneg \aneg {\FmA} \alrarr \FmA}{(\JThreeName4)}$\\[0.3cm]
$\SetFmlaHRule{}{\circ \FmA \alrarr \circ \aneg \FmA}{(\JThreeName5)}$\\[0.3cm]
$\SetFmlaHRule{}{\aneg \wdia \FmA \alrarr (\aneg \FmA \aand \circ \FmA)}{(\JThreeName6)}$\\[0.3cm]
$\SetFmlaHRule{}{\circ (\wdia \FmA)}{(\JThreeName7)}$\\[0.3cm]
$\SetFmlaHRule{}{((\aneg (\FmA \aand \FmB) \aand \circ (\FmA \aand \FmB)) \aand \FmB) \ararr (\aneg \FmA \aand \circ \FmA)}{(\JThreeName8)}$\\[0.3cm]
$\SetFmlaHRule{}{(\aneg \FmA \aand \circ \FmA) \ararr (\aneg (\FmA \aand \FmB) \aand \circ (\FmA \aand \FmB))}{(\JThreeName9)}$\\[0.3cm]
$\SetFmlaHRule{}{((\FmA \aand \FmB) \aand \circ (\FmA \aand \FmB)) \alrarr ((\FmA \aand \circ \FmA) \aand (\FmB \aand \circ \FmB))}{(\JThreeName10)}$\\[0.3cm]
$\SetFmlaHRule{\FmA \quad \FmA \ararr \FmB}{\FmB}{(\JThreeName11)}$
\end{tabular}
\end{table}



\paragraph{G\"odel (G3)}

Like \L ukasiewicz's,
G\"{o}del logics form an infinite family
of $n$-valued propositional systems
 which are all extensions of a (unique)
 infinite-valued 
 common generalization
 today known as \emph{G\"{o}del-Dummett logic}.
The origin of
these systems can indeed be traced back to K.~G\"{o}del~\cite{godel1932},
who introduced the corresponding  family of finite matrices as a means to prove that intuitionistic logic is not  finite-valued.
G\"{o}del-Dummett logic, the infinite-valued member of the family, was first considered
by M.~Dummett~\cite{dummet1959},
who showed that it can be obtained
as the axiomatic strengthening of intuitionistic logic 
by  the \emph{prelinearity}
axiom $(\GThreeName1)$ shown below. Each $n$-valued G\"{o}del logic can be axiomatized by further adding to G\"{o}del-Dummett logic a single characteristic formula, which in the case of 
the three-valued logic $\GThreeName$ is the axiom denoted
$(\GThreeName2)$ below (see e.g.~\cite[Thm.~6.9 and Lemma 6.10]{Ono19}).

\begin{table}[H]
\begin{tabular}{l}
$\SetFmlaHRule{}{(\FmA \ararr \FmB) \aor (\FmB \ararr \FmA)}{(\GThreeName1)}$\\[0.3cm]
$\SetFmlaHRule{}{\FmA \aor (\FmA \ararr (\FmB\aor\neg\FmB))}{(\GThreeName2)}$
\end{tabular}
\end{table}

The logics in the G\"{o}del family are  mostly regarded (and have been extensively studied) as \emph{fuzzy logics}, 
with intended application 
to reasoning with vague predicates (see e.g.~\cite{hajek98}). 
Accordingly,  non-classical truth values such as $\uv$ are interpreted as 
degrees of truth intermediate
between $\bv$ (absolute falsity) and
$\tv$ (absolute truth).
As axiomatic extensions of the intuitionistic propositional calculus,
all 
G\"{o}del logics 
inherit the standard algebraic semantics of intuitionistic logic,
which is the class of Heyting algebras (aka relatively pseudo-complemented lattices).
$\GThreeName$ is no exception, and
its truth tables are determined by the algebraic counterparts of the propositional connectives: 
the conjunction and disjunction 
 (as in the case of \LThreeName{}, \SKName{} etc.) are $\min$ and $\max$,
 while the implication and negation correspond, respectively, to the relative pseudo-complement and the pseudo-complement of Heyting algebras.


\paragraph{Post ($\PostName$)}


The systems introduced by Emil Post  in  his seminal 1921 paper~\cite{post1921} are among the earliest
many-valued logics in the literature. The intended application (to problems 
of representability of functions) is purely mathematical, so the non-classical  values do not come with an intended semantical interpretation. Post uniformly defines an $n$-valued logic for each natural number $n$, and establishes that each  logic in this family is functionally complete and is axiomatized by a finite Hilbert-style system. Given our focus on three-valued logics, we shall here only consider
the case where $n=3$.
This logic shares the conjunction and the disjunction with \LThreeName{}, 
but has a \emph{cyclic} negation 
that we denote by $\postNeg$
(Figure~\ref{fig:truth-tables2}).
We will show in Example~\ref{ex:postaxiomat}
that this negation and thus the whole
logic has a simple finite (and analytic) \SetSet{} axiomatization.


%

\paragraph{Logics $\MSevenName$ and $\MEightName$.}
These two systems were
constructed \emph{ad hoc}
to provide examples of
three-valued logics not axiomatizable by any 
finite \SetFmla{} Hilbert-style calculus~\cite[Lemma 3.3]{pala1994}.
We will
see, however, that there is a simple \SetSet{} analytic axiomatization for $\MSevenName$, obtainable
directly from the recipe we present
in Section~\ref{sec:MultipleConclusionCalculi}.
The logic $\MEightName$,  
which lacks
the expressiveness requirement for
the application of our recipe, is
nevertheless finitely axiomatizable in \SetSet{} by~\cite[Theorem 19.12]{shoesmithsmiley1978}.


\section{The formalism of 3-labelled calculi}
\label{sec:3LabelledCalculi}
We begin by discussing what kinds
of objects are manipulated
by rules of inference in 3-labelled calculi.
A \emph{3-labelled formula} is an expression of the form `$\LabelValA\colon \FmA$', where $\LabelValA \in \{\bv, \uv, \tv\}$ is a \emph{label} and $\FmA$ is a formula.
The
\emph{3-labelled calculi} \cite{BaaFerZac98, avronzamanssamebook} 
constitute a family of proof systems whose rules of inference
manipulate finite sets of labelled formulas, usually referred to as \emph{3-sequents}. Henceforth, we will use the letters $\LabFmSetA, \LabFmSetB$ and $\LabFmSetC$ to denote 3-sequents and write $\LabFmSetA, \LabFmSetB$ to mean $\LabFmSetA \cup \LabFmSetB$.
We may eliminate the truth values from the notation of a 3-sequent $\LabFmSetA$ by writing it as a triple of sets of formulas
as follows:
\begin{equation*}
\FmSetA_{\bv} \mid \FmSetA_{\uv} \mid \FmSetA_{\tv},
\end{equation*}
where each formula in $\FmSetA_\LabelValA$ 
is labelled with $\LabelValA$ in $\{\bv, \uv, \tv\}$. 
We may write
$\CtxLabAZero, \CtxLabAHalf$ and $\CtxLabAOne$
instead of $\FmSetA_{\bv}$,
$\FmSetA_{\uv}$ and $\FmSetA_{\tv}$, respectively, in case
there is no risk of confusion when writing rules
and derivations.
The 3-labelled calculi of interest
in this work are those containing
the structural rules
of \emph{identity}, \emph{weakening}
and \emph{cut}, 
which do not depend on the connectives
present in the language,
as well as 
\emph{logical rules},
which vary according to the language
and describe the behaviour of the
connectives.
We describe these rules 
in more detail below.

\begin{paragraph}{Identity}
The only nullary (or \emph{axiomatic}) rule:

\begin{center}
\AXC{$\ $}
\RL{\fns Id$_{\FmA}$}
\UIC{$\FmA \mid \FmA \mid \FmA$}
\DP
\end{center}
\end{paragraph}
\begin{paragraph}{\textsc{Weakening}}
We may derive from any 3-sequent another
one by adding an arbitrary formula
in one of the three positions:
\begin{table}[H]
\centering
\begin{tabular}{ccc}
\AXC{$\CtxLabAZero \mid \CtxLabAHalf \mid \CtxLabAOne$}
\RL{\fns $W_\bv$}
\UIC{$\CtxLabAZero, \FmA \mid \CtxLabAHalf \mid \CtxLabAOne$}
\DP
&
\AXC{$\CtxLabAZero \mid \CtxLabAHalf \mid \CtxLabAOne$}
\RL{\fns $W_{\uv}$}
\UIC{$\CtxLabAZero\mid \CtxLabAHalf, \FmA  \mid \CtxLabAOne$}
\DP
 &
\AXC{$\CtxLabAZero \mid \CtxLabAHalf \mid \CtxLabAOne$}
\RL{\fns $W_{\tv}$}
\UIC{$\CtxLabAZero\mid \CtxLabAHalf  \mid \CtxLabAOne, \FmA$}
\DP
\end{tabular}
\end{table}
\end{paragraph}
\begin{paragraph}{Cut}
We may cut a formula whenever it appears
in different positions in two given 3-sequents:

\begin{table}[H]
    \centering
    \begin{tabular}{c}
        \AXC{$\CtxLabAZero, \FmA \mid \CtxLabAHalf \mid \CtxLabAOne$}
        \AXC{$\CtxLabAZero \mid \CtxLabAHalf, \FmA \mid \CtxLabAOne$}
        \RL{\fns Cut$_{\bv, \uv}$}
        \BIC{$\CtxLabAZero \mid \CtxLabAHalf \mid \CtxLabAOne$}
        \DP
        \\[0.4cm]
        \AXC{$\CtxLabAZero, \FmA \mid \CtxLabAHalf \mid \CtxLabAOne$}
        \AXC{$\CtxLabAZero \mid \CtxLabAHalf \mid \CtxLabAOne, \FmA$}
        \RL{\fns Cut$_{\bv, \tv}$}
        \BIC{$\CtxLabAZero \mid \CtxLabAHalf \mid \CtxLabAOne$}
        \DP
        \\[0.4cm]
        \AXC{$\CtxLabAZero \mid \CtxLabAHalf, \FmA \mid \CtxLabAOne$}
        \AXC{$\CtxLabAZero \mid \CtxLabAHalf \mid \CtxLabAOne, \FmA$}
        \RL{\fns Cut$_{\uv, \tv}$}
        \BIC{$\CtxLabAZero \mid \CtxLabAHalf \mid \CtxLabAOne$}
        \DP
    \end{tabular}
\end{table}
\end{paragraph}

\begin{paragraph}{Logical rules}
Let $\LabFmSetA_1,\ldots,\LabFmSetA_m$
be sets of 3-labelled formulas such that,
for each $\LabelValA' : \FmB \in \LabFmSetA_i$, we have $\FmB \in \{\FmA_1,\ldots,\FmA_k\} \subseteq \LangSetA$.
Also, let $\conn$ be a $k$-ary connective and $\LabelValA \in \ThreeValuesSet$.
Then the following is a logical rule:

\begin{center}
\AXC{$(\CtxLabAZero \mid \CtxLabAHalf \mid \CtxLabAOne)[\LabFmSetA_1] \quad \cdots \quad (\CtxLabAZero \mid \CtxLabAHalf \mid \CtxLabAOne)[\LabFmSetA_m]$}
\UIC{$(\CtxLabAZero \mid \CtxLabAHalf \mid \CtxLabAOne)[\{\LabelValA : \conn(\FmA_1,\ldots,\FmA_k)\}]$}
\DP
\end{center}

\noindent
where the expression 
$(\CtxLabAZero \mid \CtxLabAHalf \mid \CtxLabAOne)[\LabFmSetB]$
refers to the 3-sequent
resulting from adding the 3-labelled
formulas in $\LabFmSetB$
to $\CtxLabAZero \mid \CtxLabAHalf \mid \CtxLabAOne$.
The 3-labelled formula
$\LabelValA : \conn(\FmA_1,\ldots,\FmA_k)$
is called \emph{the principal formula}
of the rule,
each $\FmA_i$, where $1 \leq i \leq k$,
is an \emph{auxiliary formula},
the sets
$\CtxLabAZero, \CtxLabAHalf$
and $\CtxLabAOne$
are the \emph{contexts}
of the rule, and their formulas are called \emph{side formulas}.
\end{paragraph}

A \emph{proof} of a 3-labelled sequent $\CtxLabAZero \mid \CtxLabAHalf \mid \CtxLabAOne$
in a 3-labelled calculus is, as one could expect from
the usual practice in sequent-style calculi,
a finite tree of 3-sequents rooted
in $\CtxLabAZero \mid \CtxLabAHalf \mid \CtxLabAOne$
built using the rules of inference of the calculus,
whose leaves result from an application of an instance of a
rule with an empty set of premises.
See Example~\ref{ex:derivationthreelabelled}
for a proof of a 3-sequent using a 3-labelled calculus
for the logic $\LThreeName$.
A 3-labelled calculus, differently
from what happens with Hilbert-style
systems, induce more than one
logic, depending on the logical
matrix we intend to capture. 
Usually the algebra is fixed, and
the different induced logics vary with the
choice of designated elements:

\begin{definition}
\label{def:crofthreelabelled}
Given a 3-labelled system 
$\LabSysNameA$
and a nonempty set $D \subseteq \ThreeValuesSet$,
we associate a consequence relation
$\SetSetCR_{\LabSysNameA}^{D}$
to $\LabSysNameA$ and $D$
such that
\begin{gather*}
     \FmSetA \SetSetCR_{\LabSysNameA}^{D}
     \FmSetB
     \text{ if, and only if, } \\
     \bigcup_{\LabelValA \in \ThreeValuesSet{\setminus}D } \{\LabelValA:\FmA \mid \FmA \in \FmSetA\}
     \,\cup\, \bigcup_{\LabelValA \in D}\{ \LabelValA:\FmB \mid \FmB \in \FmSetB\}
     \text{ is provable in $\LabSysNameA$}.
\end{gather*}
In the three-valued case, we will usually
deal with $D=\{\tv\}$
and $D=\{\uv,\tv\}$, for which cases
the above expression specializes to:
\begin{itemize}
    \item $\FmSetA \SetSetCR_{\LabSysNameA}^{\{\tv\}}
     \FmSetB \;
     \text{ if, and only if, } \;
     \FmSetA \mid \FmSetA \mid \FmSetB
     \text{ is provable in $\LabSysNameA$}$, and
     \item $\FmSetA \SetSetCR_{\LabSysNameA}^{\{\uv, \tv\}}
     \FmSetB \;
     \text{ if, and only if, } \;
     \FmSetA \mid \FmSetB \mid \FmSetB
     \text{ is provable in $\LabSysNameA$}$.
\end{itemize}
\end{definition}

Let $\MatA = \langle \AlgA, D \rangle$ be a three-valued logical matrix.
We say that a 3-labelled calculus
$\LabSysNameA$ \emph{axiomatizes}
$\MatA$ in case $\SetSetCR_{\MatA} \; = \; \SetSetCR_{\LabSysNameA}^{D}$.
An $\MatA$-valuation $\ValuationA$ \emph{satisfies}
a 3-sequent $\LabFmSetA$
in case $\ValuationA(\FmA) = \LabelValA$
for some 3-labelled formula $\LabelValA: \FmA \in \LabFmSetA$.
A 3-labelled rule
\emph{holds} in $\MatA$
if for no $\MatA$-valuation
$\ValuationA$
it is the case that the
3-sequent premises are satisfied
by $\ValuationA$
while the conclusion 3-sequent
is not satisfied by $\ValuationA$.
The next lemma shows that we may
use  3-labelled logical rules to give a proof-theoretical characterization
of each entry of the interpretations (truth tables) 
of the connectives in $\MatA$.
%


\begin{lemma}
A 3-labelled logical rule
of the form
\begin{center}
\AXC{$(\CtxLabAZero \mid \CtxLabAHalf \mid \CtxLabAOne)[\{\LabelValA_1 : \FmA_1\}] \quad \cdots \quad (\CtxLabAZero \mid \CtxLabAHalf \mid \CtxLabAOne)[\{\LabelValA_k : \FmA_k\}]$}
\RL{$\conn:\LabelValA_1,\ldots,\LabelValA_k,\LabelValA$}
\UIC{$(\CtxLabAZero \mid \CtxLabAHalf \mid \CtxLabAOne)[\{\LabelValA : \conn(\FmA_1,\ldots,\FmA_k)\}]$}
\DP
\end{center}
 holds in $\MatA := \langle \AlgA, D \rangle$
if, and only if,
$\AlgInterp{\conn}{\AlgA}(\LabelValA_1,\ldots,\LabelValA_k) = \LabelValA$.
\end{lemma}

It is natural, then, to expect that, by inspecting the entries of the truth tables of $\MatA$
and generating the corresponding rules
according to the previous result,
we should get a 3-labelled axiomatization
for $\MatA$. The next result matches
precisely
this expectation, and gives more:
the resulting calculus 
is \emph{cut-free}
and \emph{analytic}, that is, respectively,
every provable consecution has
a proof where cuts are not employed
and only subformulas of the 
proved statement occur in the proof.

\begin{theorem}[\cite{BaaFerZac98, avronzamanssamebook}]
\label{the:axiomatization_three_labelled}
Let $\MatA = \langle \AlgA, D \rangle$ be a three-valued matrix
whose determined logic we denote by $\LName$. Then the 3-labelled calculus $\ThreeLabName{\LName}$
whose logical rules are given by
\begin{center}
\AXC{$(\CtxLabAZero \mid \CtxLabAHalf \mid \CtxLabAOne)[\,\LabelValA_1 : \FmA_1\,] \quad \cdots \quad (\CtxLabAZero \mid \CtxLabAHalf \mid \CtxLabAOne)[\,\LabelValA_k : \FmA_k\,]$}
\RL{$\conn:\LabelValA_1,\ldots,\LabelValA_k$}
\UIC{$(\CtxLabAZero \mid \CtxLabAHalf \mid \CtxLabAOne)[\conn_\AlgA(\LabelValA_1,\ldots,\LabelValA_k) : \conn(\FmA_1,\ldots,\FmA_k)\,]$}
\DP,
\end{center}
for all $\LabelValA_1,\ldots,\LabelValA_k \in \ThreeValuesSet$, $\conn \in \Sigma_k$
and $k \in \omega$,
axiomatizes
$\MatA$, being, in addition, cut-free
and analytic.
\end{theorem}

In the following,
in order to emphasize that the rules above are in a one-to-one correspondence with matrix entries,
we will write the rule names as $\conn^{\LabelValA_1\LabelValA_2}_{\LabelValA}$ to identify the matrix entries for binary connectives $\conn$ (also for~unary connectives, in which case there will be no $\LabelValA_2$), where the indices $\LabelValA_1\LabelValA_2$ refers to the list of input values of the corresponding matrix entry (that govern the position of the auxiliary formulas in the premise of a rule), and $\LabelValA$ refers to the corresponding output value in the matrix (which governs the position of the principal formula in the conclusion of the rule). Moreover, we color the output value in rules' names to emphasize all rules sharing the same output per each connective.

\begin{example}
\label{ex:lukrules}
We illustrate below the case of
$\ThreeLabName{\LThreeName}$,
beginning with $\neg$.
From the facts that
$\neg_{\LThreeName}(\tv) = \bv$,
$\neg_{\LThreeName}(\uv) = \uv$
and
$\neg_{\LThreeName}(\bv) = \tv$,
we obtain, respectively, the following 3-labelled rules:

\begin{table}[H]
\centering
    \begin{tabular}{ccc}
    \AXC{$\CtxLabAZero, \FmA \mid \CtxLabAHalf \mid \CtxLabAOne$}
    \RL{\fns $\aneg^\bv_\tvb$}
    \UIC{$\CtxLabAZero \mid \CtxLabAHalf \mid \CtxLabAOne, \aneg \FmA$}
    \DP
    &
    \AXC{$\CtxLabAZero \mid \CtxLabAHalf, \FmA \mid \CtxLabAOne$}
    \RL{\fns $\aneg^{\uv}_{\uvg}$}
    \UIC{$\CtxLabAZero \mid \CtxLabAHalf, \aneg \FmA \mid \CtxLabAOne$}
    \DP
    &
    \AXC{$\CtxLabAZero \mid \CtxLabAHalf \mid \CtxLabAOne, \FmA$}
    \RL{\fns $\aneg^\tv_\bvr$}
    \UIC{$\CtxLabAZero, \aneg \FmA \mid \CtxLabAHalf \mid \CtxLabAOne$}
    \DP
    \end{tabular}
\end{table}

\noindent For $\land$, we obtain:
\def\defaultHypSeparation{\,}
\begin{table}[H]
    \small
    \centering
    \begin{tabular}{ll}
        \AXC{$\CtxLabAZero, \FmA \mid \CtxLabAHalf \mid \CtxLabAOne$}
        \AXC{$\CtxLabAZero, \FmB \mid \CtxLabAHalf \mid \CtxLabAOne$}
        \RL{\fns $\aand^{\bv\bv}_\bvr$}
        \BIC{$\CtxLabAZero, \FmA \aand \FmB \mid \CtxLabAHalf \mid \CtxLabAOne$}
        \DP 
        &  
        \AXC{$\CtxLabAZero, \FmA \mid \CtxLabAHalf \mid \CtxLabAOne$}
        \AXC{$\CtxLabAZero \mid \CtxLabAHalf, \FmB \mid \CtxLabAOne$}
        \RL{\fns $\aand^{\bv\uv}_\bvr$}
        \BIC{$\CtxLabAZero, \FmA \aand \FmB \mid \CtxLabAHalf \mid \CtxLabAOne$}
        \DP
        \\[.4cm]
        \AXC{$\CtxLabAZero, \FmA \mid \CtxLabAHalf \mid \CtxLabAOne$}
        \AXC{$\CtxLabAZero \mid \CtxLabAHalf \mid \CtxLabAOne, \FmB$}
        \RL{\fns $\aand^{\bv\tv}_\bvr$}
        \BIC{$\CtxLabAZero, \FmA \aand \FmB \mid \CtxLabAHalf \mid \CtxLabAOne$}
        \DP
        &        
        \AXC{$\CtxLabAZero \mid \CtxLabAHalf \mid \CtxLabAOne, \FmA$}
        \AXC{$\CtxLabAZero, \FmB \mid \CtxLabAHalf \mid \CtxLabAOne$}
        \RL{\fns $\aand^{\tv\bv}_\bvr$}
        \BIC{$\CtxLabAZero, \FmA \aand \FmB \mid \CtxLabAHalf \mid \CtxLabAOne$}
        \DP\\[.4cm]
        \AXC{$\CtxLabAZero \mid \CtxLabAHalf \mid \CtxLabAOne, \FmA$}
        \AXC{$\CtxLabAZero \mid \CtxLabAHalf, \FmB \mid \CtxLabAOne$}
        \RL{\fns $\aand^{\tv\uv}_\uvg$}
        \BIC{$\CtxLabAZero \mid \CtxLabAHalf, \FmA \aand \FmB \mid \CtxLabAOne$}
        \DP
        & 
        \AXC{$\CtxLabAZero \mid \CtxLabAHalf \mid \CtxLabAOne, \FmA$}
        \AXC{$\CtxLabAZero \mid \CtxLabAHalf \mid \CtxLabAOne, \FmB$}
        \RL{\fns $\aand^{\tv\tv}_\tvb$}
        \BIC{$\CtxLabAZero \mid \CtxLabAHalf \mid \CtxLabAOne, \FmA \aand \FmB$}
        \DP\\[.4cm]
        \AXC{$\CtxLabAZero \mid \CtxLabAHalf, \FmA \mid \CtxLabAOne$}
\AXC{$\CtxLabAZero, \FmB \mid \CtxLabAHalf \mid \CtxLabAOne$}
\RL{\fns $\aand^{\uv\bv}_\bvr$}
\BIC{$\CtxLabAZero, \FmA \aand \FmB \mid \CtxLabAHalf \mid \CtxLabAOne$}
\DP
&
\AXC{$\CtxLabAZero \mid \CtxLabAHalf, \FmA \mid \CtxLabAOne$}
\AXC{$\CtxLabAZero \mid \CtxLabAHalf, \FmB \mid \CtxLabAOne$}
\RL{\fns $\aand^{\uv\uv}_\uvg$}
\BIC{$\CtxLabAZero \mid \CtxLabAHalf, \FmA \aand \FmB \mid \CtxLabAOne$}
\DP
\\[.4cm]
\AXC{$\CtxLabAZero \mid \CtxLabAHalf, \FmA \mid \CtxLabAOne$}
\AXC{$\CtxLabAZero \mid \CtxLabAHalf \mid \CtxLabAOne, \FmB$}
\RL{\fns $\aand^{\uv\tv}_\uvg$}
\BIC{$\CtxLabAZero \mid \CtxLabAHalf, \FmA \aand \FmB \mid \CtxLabAOne$}
\DP
&
    \end{tabular}
\end{table}

\noindent We leave to the reader the completion
of this calculus for $\LThreeName$
in terms of the rules for
$\lor$ and $\to$.

%
%
%

\end{example}

The previous result
is clearly
a recipe to axiomatize
three-valued matrices
via 3-labelled calculi. We call it
the \emph{(3-labelled) generating subprocedure}.
Notice that each logical rule introduces the 3-labelled formula $\LabelValA: \conn(\FmA_1, \ldots, \FmA_n)$ in the conclusion, and the auxiliary formulas are immediate subformulas of the principal formula occurring exactly once, each of which in a different premise.
Nonetheless, providing an actual derivation is quite cumbersome. In the case of three-valued logics, we will have $9$ logical introduction rules for each connective, one for each entry of the concerned three-valued logical matrix. 

A second calculus may be obtained from the previous one via
what we we call here a \emph{streamlining sub-procedure} --- the result is called \emph{streamlined 3-labelled calculi}. 
In this case, we are able to get, for instance, a streamlined system with exactly three logical rules per each  $k$-ary connective $\conn$, one for each $\LabelValA \in \ThreeValuesSet$, whose effect in a derivation is the
introduction of the 3-labelled formula $\LabelValA: \conn(\FmA_1, \ldots, \FmA_k)$. 
A streamlined calculus satisfying this property will be denoted here
by
$\ThreeLabStreamName{\LName}$, where $\LName$ is the concerned three-valued logic.
Again, the auxiliary formulas in the premise of logical rules are immediate subformulas of the principal formula, the number of premises is finite but not fixed, but, now, the very same auxiliary formula could occur more than once and in different sequent components. 
We provide below, for self-containment,
the steps to produce a streamlined calculus and point
to the references where the correctness of the simplification
steps is proved.

\begin{definition}
\label{def:3LabelledSreamliningProcedure}
The \emph{streamlined version} of a 3-labelled calculus is obtained applying the following streamlining subprocedure (see \cite{avronzamanssamebook} in the present book):
\begin{description}
\item[(P1)] If a rule $r$ is derivable from other rules without using cuts, then $r$ can be deleted.
\item[(P2)] If $r \SymbDef \frac{\LabFmSetA}{\LabFmSetB}$ is a rule with $\LabFmSetA$ a finite set of premises, $r' \SymbDef \frac{\LabFmSetA'}{\LabFmSetB}$ is sound, and $\LabFmSetA' \subset \LabFmSetA$, then $r$ can be replaced by $r'$.
\item[(P3)] Rules $r \SymbDef \frac{\LabFmSetA_1 \,\cdots\, \LabFmSetA_m}{\LabFmSetB}$ and $r' \SymbDef \frac{\LabFmSetA'_1 \cdots \LabFmSetA'_n}{\LabFmSetB}$ can be replaced by $r'' \SymbDef \frac{\{\LabFmSetA_i \cup \LabFmSetA'_j\}}{T}$ with $1\leq i \leq m$ and $1\leq j \leq n$, and vice-versa.
\end{description}
\end{definition}

Here we just recall the following key result for the correctness of the steps above refer the reader to \cite{avronzamanssamebook} in this book for the proof:

\begin{proposition}
If $\LabSysNameA'$ and $\LabSysNameA$ are 3-labelled calculi and $\LabSysNameA'$ is obtained from $\LabSysNameA$ applying the streamlining procedure in Definition~\ref{def:3LabelledSreamliningProcedure}, then: 
\begin{itemize}
\item $\LabSysNameA'$ and $\LabSysNameA$ are equivalent (one derives the rules of the other); 
\item if $\LabSysNameA$ is sound and complete for a matrix $\MatA$, then so is $\LabSysNameA'$; and
\item if $\LabSysNameA$ admits cut-elimination, then so does $\LabSysNameA'$.
\end{itemize}
\end{proposition}

\begin{example}
The rules for $\land$ presented in Example~\ref{ex:lukrules}
can be simplified using the steps above,
resulting in the following rules:
\begin{table}[H]
    \centering
    \begin{tabular}{c}
        \AXC{$\CtxLabAZero, \FmA, \FmB \mid \CtxLabAHalf \mid \CtxLabAOne$}
        \RL{\fns $\aand_\bvr$}
        \UIC{$\CtxLabAZero, \FmA \aand \FmB \mid \CtxLabAHalf \mid \CtxLabAOne$}
        \DP
        \\[.4cm]
        \AXC{$\CtxLabAZero \mid \CtxLabAHalf, \FmA \mid \CtxLabAOne, \FmA \ \ \ \CtxLabAZero \mid \CtxLabAHalf, \FmA, \FmB \mid \CtxLabAOne \ \ \ 
        \CtxLabAZero \mid \CtxLabAHalf, \FmB \mid \CtxLabAOne, \FmB$}
        \RL{\fns $\aand_\uvg$}
        \UIC{$\CtxLabAZero \mid \CtxLabAHalf, \FmA \aand \FmB \mid \CtxLabAOne$}
        \DP
        \\[.4cm]
        \AXC{$\CtxLabAZero \mid \CtxLabAHalf \mid \CtxLabAOne, \FmA$}
        \AXC{$\CtxLabAZero \mid \CtxLabAHalf \mid \CtxLabAOne, \FmB$}
        \RL{\fns $\aand_\tvb$}
        \BIC{$\CtxLabAZero \mid \CtxLabAHalf \mid \CtxLabAOne, \FmA \aand \FmB$}
        \DP
    \end{tabular}
\end{table}
\noindent Similarly, we may obtain
one rule for each truth value introducing
each of the connectives $\lor$ and $\to$, resulting 
in the following rules:
\begin{table}[H]
    \centering
    \begin{tabular}{c}
        \AXC{$\CtxLabAZero, \FmA \mid \CtxLabAHalf \mid \CtxLabAOne$}
        \AXC{$\CtxLabAZero, \FmB \mid \CtxLabAHalf \mid \CtxLabAOne$}
        \RL{\fns $\aor_\bvr$}
        \BIC{$\CtxLabAZero, \FmA \aor \FmB \mid \CtxLabAHalf \mid \CtxLabAOne$}
        \DP
        \\[.4cm]
        \AXC{$\CtxLabAZero, \FmA \mid \CtxLabAHalf, \FmA \mid \CtxLabAOne \ \ \ \CtxLabAZero \mid \CtxLabAHalf, \FmA, \FmB \mid \CtxLabAOne \ \ \ 
        \CtxLabAZero, \FmB \mid \CtxLabAHalf, \FmB \mid \CtxLabAOne$}
        \RL{\fns $\aor_\uvg$}
        \UIC{$\CtxLabAZero \mid \CtxLabAHalf, \FmA \aor \FmB \mid \CtxLabAOne$}
        \DP
        \\[.4cm]
        \AXC{$\CtxLabAZero \mid \CtxLabAHalf \mid \CtxLabAOne, \FmA, \FmB$}
        \RL{\fns $\aor_\tvb$}
        \UIC{$\CtxLabAZero \mid \CtxLabAHalf \mid \CtxLabAOne, \FmA \aor \FmB$}
        \DP\\[.4cm]
        \AXC{$\CtxLabAZero \mid \CtxLabAHalf \mid \CtxLabAOne, \FmA$}
        \AXC{$\CtxLabAZero, \FmB \mid \CtxLabAHalf \mid \CtxLabAOne$}
        \RL{\fns $\ararr_\bvr$}
        \BIC{$\CtxLabAZero, \FmA \ararr \FmB \mid \CtxLabAHalf \mid \CtxLabAOne$}
        \DP
        \\[.4cm]
        \AXC{$\CtxLabAZero, \FmB \mid \CtxLabAHalf \mid \CtxLabAOne, \FmA$}
        \AXC{$\CtxLabAZero \mid \CtxLabAHalf, \FmA, \FmB \mid \CtxLabAOne$}
        \RL{\fns $\ararr_\uvg$}
        \BIC{$\CtxLabAZero \mid \CtxLabAHalf, \FmA \ararr \FmB \mid \CtxLabAOne$}
        \DP
        \\[.4cm]
        \AXC{$\CtxLabAZero, \FmA \mid \CtxLabAHalf, \FmA \mid \CtxLabAOne, \FmB$}
        \AXC{$\CtxLabAZero, \FmA \mid \CtxLabAHalf, \FmB \mid \CtxLabAOne, \FmB$}
        \RL{\fns $\ararr_\tvb$}
        \BIC{$\CtxLabAZero \mid \CtxLabAHalf \mid \CtxLabAOne, \FmA \ararr \FmB$}
        \DP
    \end{tabular}
\end{table}
\noindent Taking these rules together with
the rules for $\neg$ in Example~\ref{ex:lukrules}
gives us a 3-labelled calculus for $\LThreeName$ that we call
$\ThreeLabStreamName{\LThreeName}$.
\end{example}

\begin{example}
\label{ex:derivationthreelabelled}
Let us prove that $\FmA \to (\FmB \to \FmA)$ is a tautology (that is, receives a designated value under every valuation)
of $\LThreeName$. Since $\ThreeLabStreamName{\LThreeName}$
is complete for this logic, we just need to
derive the 3-sequent $\varnothing \mid \varnothing \mid \FmA \to (\FmB \to \FmA)$ in
$\ThreeLabStreamName{\LThreeName}$ (recall
Definition~\ref{def:crofthreelabelled}
and the fact that the set of designated values
of the logical matrix for $\LThreeName$ is $\{\tv\}$).
Notice that we use $W^\ast$ to mean 
a sequence of applications of weakening rules.
\begin{figure}[H]
    \centering
    \small
    \AxiomC{}
    \RL{\fns Id$_{\FmA}$}
    \UIC{$\FmA\mid\FmA\mid\FmA$}
    \RL{$W^\ast$}
    \UnaryInfC{$\FmA,\FmB \mid \FmA,\FmB \mid \FmA$}
    \AxiomC{}
    \RL{\fns Id$_{\FmA}$}
    \UIC{$\FmA\mid\FmA\mid\FmA$}
    \RL{$W$}
    \UnaryInfC{$\FmA,\FmB \mid \FmA \mid \FmA$}
    \RL{$\ararr_\tvb$}
    \BinaryInfC{$\FmA \mid \FmA \mid \FmB\to\FmA$}
    %
    \AxiomC{$\mathcal{D}_1$}
    %
    \AxiomC{$\mathcal{D}_2$}
    \RL{$\ararr_{\uvg}$}
    \BinaryInfC{$\FmA \mid \FmB\to\FmA \mid \FmB\to\FmA$}
    \RL{$\ararr_\tvb$}
    \BinaryInfC{$\varnothing \mid \varnothing \mid \FmA \to (\FmB \to \FmA)$}
    \DP
\end{figure}
\noindent where $\mathcal{D}_1$ and $\mathcal{D}_2$ are the following
subtrees:
\begin{table}[H]
    \scriptsize
    \begin{tabular}{cc}
    \centering
    \AxiomC{}
    \RL{\fns Id$_{\FmA}$}
    \UIC{$\FmA\mid\FmA\mid\FmA$}
    \RL{$W^\ast$}
    \UnaryInfC{$\FmA,\FmB \mid \FmB,\FmA \mid \FmA$}
    \AxiomC{}
    \RL{\fns Id$_{\FmA}$}
    \UIC{$\FmA\mid\FmA\mid\FmA$}
    \RL{$W^\ast$}
    \UnaryInfC{$\FmA,\FmB \mid \FmB,\FmA \mid \FmA$}
    \RL{$\ararr_\tvb$}
    \BinaryInfC{$\FmA \mid \FmB,\FmA \mid \FmB\to\FmA$}
    \DP
&
    \AxiomC{}
    \RL{\fns Id$_{\FmB}$}
    \UIC{$\FmB\mid\FmB\mid\FmB$}
    \RL{$W^\ast$}
    \UnaryInfC{$\FmA,\FmB \mid \FmB \mid \FmA,\FmB$}
    \AxiomC{}
    \RL{\fns Id$_{\FmA}$}
    \UIC{$\FmA\mid\FmA\mid\FmA$}
    \RL{$W^\ast$}
    \UnaryInfC{$\FmA,\FmB \mid \FmA \mid \FmB,\FmA$}
    \RL{$\ararr_\tvb$}
    \BinaryInfC{$\FmA \mid \varnothing \mid \FmB,\FmB\to\FmA$}
    \DP
    \end{tabular}
\end{table}

\end{example}

We should emphasize that different streamlined calculi may
be obtained depending on how and to which extent the simplification steps outlined above are applied.
The algorithm presented in Theorem~\ref{the:axiomatization_three_labelled} 
was implemented in the system MultiLog~\cite{baaz1993}, which
can be used by anyone for producing a streamlined
3-labelled calculus $\ThreeLabStreamName{\LName}$
--- that is, one with one rule per each
truth value and each connective ($3c$ rules, for $c$ connectives) ---
given a three-valued logical matrix for $\LName$ as input.

\section{Multiple-conclusion Hilbert-style calculi}
\label{sec:MultipleConclusionCalculi}

In this section, we are interested in providing finite and analytic multiple-conclusion (or \SetSet{}) 
Hilbert-style axiomatizations for three-valued logical matrices.
As we shall detail in a moment, this proof formalism generalizes
traditional Hilbert-style calculi in allowing sets of
formulas in the conclusion of the rules of inference.
Putting aside the trivial cases, we will focus on
two tasks: first, on specifying a general method
for axiomatizing any three-valued matrix that satisfies a
criterion of sufficient expressiveness called \emph{monadicity}; and, second, 
on axiomatizing some specific three-valued matrices that
do not meet such criterion. 
The result proved in \cite[Theorem 19.12]{shoesmithsmiley1978}
guarantees that every finite logical matrix (in particular, three-valued matrices)
has an adequate (sound and complete) finite \SetSet{} axiomatization, 
but finding one for the non-monadic matrices
can be a rather 
involved task, as we will see.
In~\cite{marcelino2019}, the recipe for the monadic case was generalized to deal with non-determinism and the resulting finite calculi were shown to satisfy a generalized subformula property, yielding what we call \emph{$\Theta$-analiticity}, a fact that had remained unnoticed for many decades.

We begin by giving the formal definitions related to the 
proof formalism of \SetSet{} calculi together with some examples.

\begin{definition}
\label{def:mcInferenceRule}
A \emph{\SetSet{} inference rule} $\RuleA$ is a pair $(\FmSetA,\FmSetB)\in\PowerSet{\LangSetA}\times\PowerSet{\LangSetA}$ that we denote by $\MCRule{\FmSetA}{\FmSetB}$, where
$\FmSetA$ is the \emph{antecedent} and
$\FmSetB$ is \emph{succedent} of $\RuleA$. 
If $\sigma$ is a substitution, then $\RuleA^\sigma \SymbDef \MCRule{\sigma(\FmSetA)}{\sigma(\FmSetB)}$
is said to be an \emph{instance} of $\RuleA$. 
A \emph{subrule} of
$\RuleA$
is a rule $\MCRule{\FmSetA^\prime}{\FmSetB^\prime}$ with
$\FmSetA^\prime \subseteq \FmSetA$
and $\FmSetB^\prime \subseteq \FmSetB$. 
An inference rule is \emph{finitary} when 
its antecedent and succedent are finite.
A \emph{\SetSet{} calculus} $\CalcA$ is a collection
of \SetSet{} inference rules.
\end{definition}

\begin{example} 
Consider the calculus $\CalcA^{\to\neg}_{\SThreeName}$ consisting of the following rules of inference: 
\begin{gather*}
\MCRule{\PropA}{\neg\neg\PropA}\RuleA^{\SobocNeg}_{1} \quad
\MCRule{\neg\neg\PropA}{\PropA}\RuleA^{\SobocNeg}_{2} \quad
\MCRule{}{\PropA, \neg\PropA}\RuleA^{\SobocNeg}_{3}\\[.3em]
\MCRule{\neg\PropA, \PropB}{\PropA\to\PropB}\RuleA^{\SobocImp}_{1}
\quad
\MCRule{\PropA, \neg\PropB}{\neg(\PropA\to\PropB)}\RuleA^{\SobocImp}_{2}
\quad
\MCRule{\PropA\to\PropB, \neg\PropB}{\neg\PropA}\RuleA^{\SobocImp}_{3}
\\[.3em]
\MCRule{\neg(\PropA\to\PropB)}{\PropA}\RuleA^{\SobocImp}_{4}
\quad
\MCRule{\neg(\PropA\to\PropB)}{\neg\PropB}\RuleA^{\SobocImp}_{5}
\quad
\MCRule{\PropA,\PropA\to\PropB}{\PropB}\RuleA^{\SobocImp}_{6}
\end{gather*}
Notice that what distinguishes $\CalcA^{\to\neg}_{\SThreeName}$ from a standard Hilbert-style calculus
is the last rule, whose succedent is a two-formula set.
This calculus, according to the recipe we are about to present, is an axiomatization 
(see Definition~\ref{def:axiomatizes})
of the $\{\to,\neg\}$-fragment of the
logic $\SThreeName$.
\end{example}

As mentioned in Section~\ref{sec:CaseStudies},
the logics $\BKName$ and $\PWKName$ are not known to be
finitely axiomatizable by means of a \SetFmla{} H-calculus.
The next couple of examples show that this matter is
much simpler in the \SetSet{} formalism, as these logics
fall in the scope of the axiomatization recipe
we will provide in a moment.

\begin{example} 
Consider the \SetSet{} calculus $\CalcA_{\BKName}$ consisting of the following rules of inference: 
\begin{gather*}
\MCRule{\PropA}{\neg\neg\PropA}\RuleA^{\BKleeNeg}_{1} \quad
\;
\MCRule{\neg\neg\PropA}{\PropA}\RuleA^{\BKleeNeg}_{2} \quad
\;
\MCRule{\PropA, \neg\PropA}{}\RuleA^{\BKleeNeg}_{3}\\
\MCRule{\PropA,\PropB}{\PropA\land\PropB}\RuleA^{\BKleeConj}_{1}\quad
\MCRule{\PropA\land\PropB}{\PropA}\RuleA^{\BKleeConj}_{2}\quad
\MCRule{\PropA\land\PropB}{\PropB}\RuleA^{\BKleeConj}_{3}\quad
\MCRule{\neg\PropA,\neg\PropB}{\neg(\PropA\land\PropB)}\RuleA^{\BKleeConj}_{4}
\\
\MCRule{\neg(\PropA\land\PropB)}{\PropA,\neg\PropA}\RuleA^{\BKleeConj}_{5}\quad
\MCRule{\neg(\PropA\land\PropB)}{\PropB,\neg\PropB}\RuleA^{\BKleeConj}_{6}\quad
\MCRule{\PropA,\neg\PropB}{\neg(\PropA\land\PropB)}\RuleA^{\BKleeConj}_{7}\quad
\MCRule{\neg\PropA, \PropB}{\neg(\PropA\land\PropB)}\RuleA^{\BKleeConj}_{8}
\\
\MCRule{\neg\PropA,\neg\PropB}{\PropA\to\PropB}\RuleA^{\BKleeImp}_{1}
\quad
\MCRule{\neg\PropA,\PropB}{\PropA\to\PropB}\RuleA^{\BKleeImp}_{2}
\quad
\MCRule{\PropA,\PropB}{\PropA\to\PropB}\RuleA^{\BKleeImp}_{3}
\quad
\MCRule{\PropA,\neg\PropB}{\neg(\PropA\to\PropB)}\RuleA^{\BKleeImp}_{4}\\
\MCRule{\PropA\to\PropB}{\neg\PropA,\PropA}\RuleA^{\BKleeImp}_{5}
\quad
\MCRule{\neg(\PropA\to\PropB)}{\neg\PropB}\RuleA^{\BKleeImp}_{6}
\quad
\MCRule{\PropA\to\PropB}{\neg\PropB,\PropB}\RuleA^{\BKleeImp}_{7}
\quad 
\MCRule{\neg(\PropA\to\PropB)}{\PropA}\RuleA^{\BKleeImp}_{8}
\quad
\MCRule{\PropA\to\PropB,\PropA}{\PropB}\RuleA^{\BKleeImp}_{9}\\
\MCRule{\neg\PropA,\neg\PropB}{\neg(\PropA\lor\PropB)}\RuleA^{\BKleeDisj}_{1}
\quad
\MCRule{\neg(\PropA\lor\PropB)}{\neg\PropA}\RuleA^{\BKleeDisj}_{2}
\quad
\MCRule{\PropA\lor\PropB}{\neg\PropA,\PropA}\RuleA^{\BKleeDisj}_{3}
\quad
\MCRule{\neg(\PropA\lor\PropB)}{\neg\PropB}\RuleA^{\BKleeDisj}_{4}\\
\MCRule{\PropA\lor\PropB}{\neg\PropB,\PropB}\RuleA^{\BKleeDisj}_{5}
\quad
\MCRule{\neg\PropA,\PropB}{\PropA\lor\PropB}\RuleA^{\BKleeDisj}_{6}
\quad
\MCRule{\PropA,\neg\PropB}{\PropA\lor\PropB}\RuleA^{\BKleeDisj}_{7}
\quad
\MCRule{\PropA,\PropB}{\PropA\lor\PropB}\RuleA^{\BKleeDisj}_{8}
\quad
\MCRule{\PropA\lor\PropB}{\PropA,\PropB}\RuleA^{\BKleeDisj}_{9}
\end{gather*}
This calculus axiomatizes the
logic $\BKName$, according to the
procedure we describe in Subsection~\ref{Axiomatizing3valuedMatricesTheMonadicCase}. In rule $\RuleA^{\BKleeNeg}_{3}$,
we observe that
\SetSet{} rules may have an empty succedent.
As we will see in the next definition, the
effect of an application of a rule with this
property is the discontinuation of a derivation branch.
We should emphasize that the empty succedent in a rule is not replaceable
by the version of the same rule having a single fresh propositional variable in the succedent
(for instance, using the rule $\MCRule{\neg \PropA, \PropA}{\PropB}$ instead of $\MCRule{\neg \PropA, \PropA}{}$ in the present example).
Such modification would prevent the derivation
of valid statements with empty succedents,
leading thus to a different \SetSet{} logic.
\end{example}

\begin{example}
Call $\CalcA_{\PWKName}$ the \SetSet{} calculus
given by:
\begin{itemize}
    \item all rules $\RuleA^{\BKleeNeg}_{i}$
     from the previous example turned upside down;
    \item all rules $\RuleA^{\BKleeConj}_{i}$ and
    $\RuleA^{\BKleeDisj}_{j}$ from the previous example
    turned upside down, with
    $\lor$ in place of $\land$ (and vice-versa);
    and
    \item the following rules governing $\to$:
    \begin{gather*}
        \MCRule{\neg\PropA}{\PropA\to\PropB}\RuleA^{\WKleeImp}_{1}
        \quad
        \MCRule{\PropB}{\PropA\to\PropB}\RuleA^{\WKleeImp}_{2}
        \quad
        \MCRule{}{\PropA,\PropA\to\PropB}\RuleA^{\WKleeImp}_{3}
        \quad
        \MCRule{\PropA,\neg\PropA}{\neg(\PropA\to\PropB)}\RuleA^{\WKleeImp}_{4}\\
        \MCRule{\PropB,\neg\PropB}{\neg(\PropA\to\PropB)}\RuleA^{\WKleeImp}_{5}
        \quad
        \MCRule{\PropA}{\neg(\PropA\to\PropB),\PropB}\RuleA^{\WKleeImp}_{6}
        \quad
        \MCRule{\neg(\PropA\to\PropB)}{\neg\PropA,\neg\PropB}\RuleA^{\WKleeImp}_{7}
        \quad
        \MCRule{\PropA\to\PropB}{\neg\PropA,\PropB}\RuleA^{\WKleeImp}_{8}\\
        \MCRule{\neg(\PropA\to\PropB)}{\PropA,\neg\PropB}\RuleA^{\WKleeImp}_{9}
        \quad
        \MCRule{\neg(\PropA\to\PropB)}{\PropA,\PropB}\RuleA^{\WKleeImp}_{10}
    \end{gather*}
\end{itemize}
\noindent By the recipe in Subsection~\ref{Axiomatizing3valuedMatricesTheMonadicCase},
this is a \SetSet{} axiomatization for $\PWKName$.
\end{example}

We now show how the notions of
Hilbert-style derivation and proof are generalized
in order to deal with the novelty of
allowing for sets of
formulas in the succedent of a 
rule of inference.

\begin{definition}
An \emph{$\CalcA$-derivation} is a rooted directed
tree $\TreeA$ 
such that every node is 
labelled with sets of formulas or with
a discontinuation symbol $\Star$, and in which
every non-leaf node (that is, a node with child nodes) $\NodeA$ in $\TreeA$ is an \emph{expansion of $\NodeA$
by a rule instance} $\RuleA^\sigma$ of $\CalcA$. This means that
the antecedent of $\RuleA^\sigma$ is contained in the label of $n$
and that
$n$ has exactly one child node for
each formula $\FmB$ in the succedent of $\RuleA^\sigma$,
which is, in turn, labelled with the same formulas as those of $\NodeA$
plus $\FmB$. In case $\RuleA^\sigma$ has an empty
succedent, then $\NodeA$ has a single child node
labelled with $\Star$. Figure~\ref{fig:derivationscheme}
illustrates how derivations 
using only finitary inference rules are graphically represented.
We denote by $\NodeLabel{\TreeA}{\NodeA}$
the label of the node $\NodeA$ in the tree
$\TreeA$.

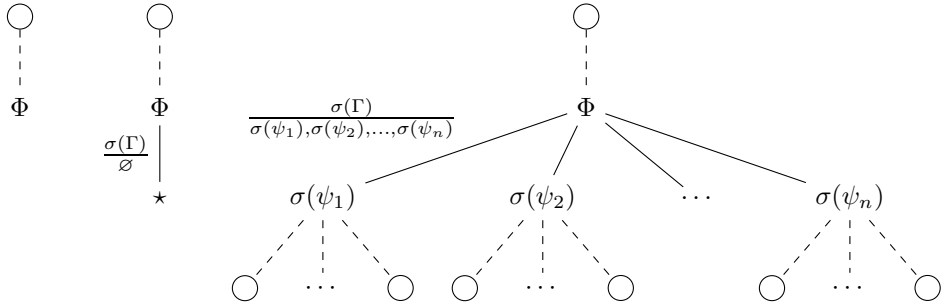
\begin{figure}[H]
    \centering
    \begin{tikzpicture}[every tree node/.style={},
       level distance=1.2cm,sibling distance=1cm,
       edge from parent path={(\tikzparentnode) -- (\tikzchildnode)}, baseline]
        \Tree[.\node[style={draw,circle}] {};
            \edge[dashed];
            [.\node[style={}] {$\FmSetA$};
            ]
        ]
    \end{tikzpicture}\qquad
    \begin{tikzpicture}[every tree node/.style={},
       level distance=1.2cm,sibling distance=1cm,
       edge from parent path={(\tikzparentnode) -- (\tikzchildnode)}, baseline]
        \Tree[.\node[style={draw,circle}] {};
            \edge[dashed];
            [.\node[style={}] {$\FmSetA$};
            \edge node[auto=right] {$\MCRule{\sigma(\FmSetD)}{\EmptySet}$};
            [.{$\Star$}
            ]
            ]
        ]
    \end{tikzpicture}\qquad
    \begin{tikzpicture}[every tree node/.style={},
       level distance=1.2cm,sibling distance=.5cm,
       edge from parent path={(\tikzparentnode) -- (\tikzchildnode)}, baseline]
        \Tree[.\node[style={draw,circle}] {};
            \edge[dashed];
            [.\node[style={}] {$\FmSetA$};
            \edge node[auto=right] {$\MCRule{\sigma(\FmSetD)}{\sigma(\FmB_1),\sigma(\FmB_2),\ldots,\sigma(\FmB_n)}$};
            [.$\sigma(\FmB_1)$
                \edge[dashed];
                [.\node[style={draw,circle}]{};
                ]
                \edge[dashed];
                [.\node[style={}]{$\cdots$};
                ]
                \edge[dashed];
                [.\node[style={draw,circle}]{};
                ]
            ]
            [.$\sigma(\FmB_2)$
                \edge[dashed];
                [.\node[style={draw,circle}]{};
                ]
                \edge[dashed];
                [.\node[style={}]{$\cdots$};
                ]
                \edge[dashed];
                [.\node[style={draw,circle}]{};
                ]
            ]
            [.$\ldots$
            ]
            [.$\sigma(\FmB_n)$
                \edge[dashed];
                [.\node[style={draw,circle}]{};
                ]
                \edge[dashed];
                [.\node[style={}]{$\cdots$};
                ]
                \edge[dashed];
                [.\node[style={draw,circle}]{};
                ]
            ]
            ]
        ]
    \end{tikzpicture}
    \caption{Graphical representation of $\CalcA$--derivations
    using finitary inference rules, in the form of a leaf node, a discontinued node and an expanded node, respectively.}
    \label{fig:derivationscheme}
\end{figure}
\end{definition}

 In Figure~\ref{fig:derivationscheme}, the dashed edges and blank circles represent other nodes and edges that may exist in the derivation. Notice that, in the case of expanded
    nodes, we omit the formulas inherited from the parent
    node, exhibiting only the ones introduced by the
    applied rule of inference. Also, we have written the rule instance
    of the rule being applied in the second and third trees, but in practice only the name
    of the rule of inference will be written for simplicity.
    In these cases, we emphasize that a precondition
    for the application of the rule instance is
    that $\sigma(\FmSetD) \subseteq \FmSetA$
    (that is, the premises of the rule
    being applied are satisfied).

\begin{definition}
Given $\FmSetE \subseteq \LangSetA$, a node $\NodeA$ of an $\CalcA$-derivation $\TreeA$ is said to be
    \emph{$\FmSetE$-closed} when 
    $\NodeLabel{\TreeA}{\NodeA} \cap \FmSetE \neq \EmptySet$ or
    when $\NodeA$ is labelled with $\Star$.
    If all leaf nodes of $\TreeA$ are $\FmSetE$-closed,
    then $\TreeA$ is said to be $\FmSetE$-closed.
\end{definition}

\begin{definition}
An \emph{$\CalcA$-proof} of a
    statement $(\FmSetA,\FmSetB)$ is a
    $\FmSetB$-closed
    $\CalcA$-derivation whose root
    is labelled with a subset of $\FmSetA$.
\end{definition}

The notion of \SetSet{} calculus
succeeds in capturing the notion of
\SetSet{} logics, as the following
result states:

\begin{proposition}
A \SetSet{} calculus $\CalcA$ induces a
\SetSet{} logic $\SetSetCR_{\CalcA}$
such that $\FmSetA \SetSetCR_{\CalcA} \FmSetB$
if, and only if, there is an $\CalcA$-proof
of $(\FmSetA,\FmSetB)$.
This is the smallest logic
(according to $\subseteq$)
containing the rules of inference
of $\CalcA$.
\end{proposition}
\begin{proof}
See~\cite[Theorem 1]{shoesmithsmiley1978}.
\end{proof}

\begin{example}
We provide below an $\CalcA^{\to\neg}_{\SThreeName}$-proof 
and an $\CalcA^{\land\neg}_{\BKName}$-proof
witnessing, respectively,
that $\neg\PropA,\PropB,\PropC \SetSetCR_{\CalcA^{\to\neg}_{\SThreeName}} (\PropB \to \PropA)\to(\PropD\to\PropC),\PropA\to\PropD$
and $\neg(\PropA\land\PropB),\PropA \SetSetCR_{\CalcA^{\land\neg}_{\BKName}} \neg\PropB$.
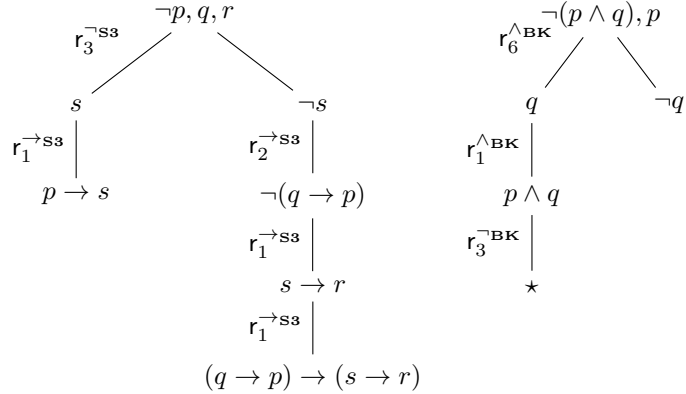
\begin{figure}[H]
\centering
\begin{tikzpicture}[every tree node/.style={},
   level distance=1.2cm,sibling distance=1cm,
   edge from parent path={(\tikzparentnode) -- (\tikzchildnode)}, baseline]
    \Tree[.\node[style={}] {$\neg\PropA,\PropB,\PropC$};
        \edge[] node[auto=right]{$\RuleA^{\SobocNeg}_{3}$};
        [.{$\PropD$}
            \edge[] node[auto=right]{$\RuleA^{\SobocImp}_{1}$};
            [.{$\PropA\to\PropD$}
            ]
        ]
        [.{$\neg\PropD$}
            \edge[] node[auto=right]{$\RuleA^{\SobocImp}_{2}$};
            [.{$\neg(\PropB\to\PropA)$}
                \edge[] node[auto=right]{$\RuleA^{\SobocImp}_{1}$};
                [.{$\PropD\to\PropC$}
                    \edge[] node[auto=right]{$\RuleA^{\SobocImp}_{1}$};
                    [.{$(\PropB \to \PropA)\to(\PropD\to\PropC)$}
                    ]
                ]
            ]
        ]
    ]
\end{tikzpicture}\quad
\begin{tikzpicture}[every tree node/.style={},
   level distance=1.2cm,sibling distance=1cm,
   edge from parent path={(\tikzparentnode) -- (\tikzchildnode)}, baseline]
    \Tree[.\node[style={}] {$\neg(\PropA\land\PropB),\PropA$};
        \edge[] node[auto=right]{$\RuleA^{\BKleeConj}_{6}$};
        [.{$\PropB$}
            \edge[] node[auto=right]{$\RuleA^{\BKleeConj}_{1}$};
            [.{$\PropA\land\PropB$}
                \edge[] node[auto=right]{$\RuleA^{\BKleeNeg}_{3}$};
                [.{$\Star$}
                ]
            ]
        ]
        [.{$\neg\PropB$}
        ]
    ]
\end{tikzpicture}
\end{figure}

\end{example}

%

Knowing in advance that the subformulas of a statement of interest 
are enough for providing
a proof of the statement
is an important feature of a
calculus, usually referred to as \emph{analyticity},
which was already observed in the
3-labelled systems introduced in the
previous section.
The next definition introduces a broader notion of
analyticity, called \emph{$\AnalyticSetA$-analyticity},
where $\AnalyticSetA \subseteq \LangSetAp$ ---
that is, a set of unary formulas. The main idea
is to allow for not only subformulas of a statement 
to appear in an analytic proof,
but also formulas resulting from substitutions of those
subformulas over the formulas in $\AnalyticSetA$. 
For example, if $\AnalyticSetA = \{\PropC, \neg\PropC\}$
,
a $\AnalyticSetA$-analytic proof 
witnessing that $\neg\PropA$
follows from $\neg(\PropA\land\PropB)$ would
use only formulas in $\{\PropA,\PropB,\neg\PropA,\neg\PropB,\neg\neg\PropA,\PropA\land\PropB,\neg(\PropA\land\PropB),\neg\neg(\PropA\land\PropB)\}$.
We formalize the notion of $\Theta$-analyticity
below.

\begin{definition}
Let $\AnalyticSetA \subseteq \LangSetAp$ and $\CalcA$ be a 
\SetSet{} calculus. We say that an $\CalcA$-proof $\TreeA$
of $(\FmSetA,\FmSetB)$
is \emph{$\AnalyticSetA$-analytic} in case only formulas
in $\SubfTheta{\Theta}{\FmSetA\cup\FmSetB}\SymbDef \Subf{\FmSetA\cup\FmSetB} \cup \{\sigma(\FmA) \mid \FmA\in\AnalyticSetA, \sigma: P \to \Subf{\FmSetA\cup\FmSetB}\}$ appear in the labels of the nodes
of $\TreeA$.
We write, in this case, that
$\FmSetA \AnaSetSetCR{\CalcA}{\AnalyticSetA} \FmSetB$.
When 
we have $\FmSetA \AnaSetSetCR{\CalcA}{\AnalyticSetA} \FmSetB$
whenever $\FmSetA \SetSetCR_{\CalcA}\FmSetB$ (notice that the converse always hold),
then $\CalcA$ is said to be \emph{$\AnalyticSetA$-analytic}.
\end{definition}

Observe that, in case $\AnalyticSetA = \EmptySet$
or $\AnalyticSetA = \{\PropA\}$,
we obtain the standard notion of analyticity, according to which
only the subformulas of the statement are enough to produce
a proof of that very statement in case such proof exists.

As explained in \cite{marcelino2019},
finite and finitary $\Theta$-analytic calculi are associated to a proof-search algorithm and a countermodel-search algorithm,
and consequently to a decision procedure for the corresponding \SetSet{} logics. 
Briefly put, if we want to know whether $\FmSetA \;\SetSetCR_\CalcA\; \FmSetB$,
where $\CalcA$ is a finite and finitary $\Theta$-analytic
\SetSet{} calculus, 
obtaining a proof when the answer is positive and a (symbolic) countermodel otherwise,
we may attempt to build a derivation in
the following way: start from
a single node labelled with $\FmSetA$
and search for a rule instance
of $\CalcA$ not used in the same branch
with formulas in the set
$\SubfTheta{\Theta}{\FmSetA\cup\FmSetB}$
whose premises are in~$\FmSetA$. If there is one,
expand that node by creating a
child node labelled with $\FmSetA\cup\{\FmB\}$ for each formula $\FmB$ 
in the succedent of the chosen rule instance
and repeat this step for each new node.
In case it fails in finding a rule instance for applying to some node,
we may conclude that no proof exists,
and from each non-$\FmSetB$-closed branch we may extract a countermodel.
In case every branch eventually gets $\FmSetB$-closed, the resulting
tree is a proof of the desired statement.
This procedure is exponential-time in the size of the statement of interest in general
and, when $\CalcA$ is \SetFmla{}
(all inference rules have a single formula
in the succedent), the running-time is
polynomial.

\begin{definition}
\label{def:axiomatizes}
Let $\MatA$ be a logical matrix and $\CalcA$ be a multiple
conclusion calculus. Then $\CalcA$ is \emph{sound}
for $\MatA$ when $\SetSetCR_{\CalcA} \subseteq \SetSetCR_{\MatA}$ and it is \emph{complete}
for $\MatA$ when the converse holds.
When $\CalcA$ is sound and complete for $\MatA$,
we say that $\CalcA$ is \emph{adequate for} or \emph{axiomatizes} $\MatA$.
\end{definition}

\begin{example}
\label{ex:trivialsetset}
The trivial three-valued matrices $\MatA$ are those having as sets of designated values either 
$\EmptySet$ or $\{\bv, \uv, \tv\}$.
It is not hard to see that,
in the first case, the calculus containing
the single rule $\MCRule{\PropA}{\EmptySet}$
axiomatizes $\MatA$; likewise, in the second case,
the rule $\MCRule{\EmptySet}{\PropA}$ is enough.
Note that, in the case of 3-labelled calculi,
the calculus $\LabSysNameTriv$ having only the rules of Identity and
Weakening can be easily seen to axiomatize
the trivial matrices (refer to Definition~\ref{def:crofthreelabelled} for the definitions of the consequence relations in these cases).
\end{example}


In view of Example~\ref{ex:trivialsetset}, from now on 
we shall only 
consider matrices with
either one or two designated values. We may 
reduce even more
the configurations to be considered by noticing that
a phenomenon observable in 
Example~\ref{ex:trivialsetset}
occurs in general:
if $\MatA_1 \SymbDef \langle \AlgA,\DesSetA_1 \rangle$
and $\MatA_2 \SymbDef \langle \AlgA,\ValuesSetComp{A}{\DesSetA_1} \rangle$, then $\FmSetA \SetSetCR_{\MatA_1} \FmSetB$ if{f}
$\FmSetB \SetSetCR_{\MatA_2} \FmSetA$, thus
an axiomatization for $\MatA_1$ can be readily turned into (or used as) an axiomatization
for $\MatA_2$. 
Indeed, the procedure described in the sequel
will  deliver a calculus for $\MatA_1$
from which a calculus for $\MatA_2$ may be produced
by simply turning the inference rules of the former upside down.
In view of this duality,
in the three-valued setting
we will first restrict our attention only to matrices having a single designated
value, and will then discuss how to extend the
approach to the other cases.

\subsection{Axiomatizing three-valued matrices: the monadic case}
\label{Axiomatizing3valuedMatricesTheMonadicCase}

We say that $\MatA \SymbDef \langle \AlgA, D \rangle$ is \emph{monadic} whenever there is a unary formula
$\Sep(\PropA) \in \LangSetAp$,
sometimes called a \emph{separator}, for each pair of truth values $x, y \in A$, such that
$\Sep_{\AlgA}(x) \in D$ and $\Sep_{\AlgA}(y) \in \ValuesSetComp{A}{D}$
or vice-versa.
Throughout this section, we fix a three-valued matrix $\MatA \SymbDef \langle \AlgA, D \rangle$ with $A \SymbDef \{\dv,\uva,\uvb\}$ and $D \SymbDef \{\dv\}$.
Notice that the propositional variable $\PropA$ separates $\dv$ from $\uva$ and $\uvb$, so monadicity
in this situation
requires just the existence of
a separator $\Sep$ for the non-designated values $\uva$ and $\uvb$.
Granted the existence of this separator, then,
for each $x \in \{\uva,\uvb\}$, let $(\DesSeps{x}(\PropA),\NDesSeps{x}(\PropA))$ be a quasi-partition (that is, a partition allowing for empty components) of $\{\PropA, \Sep(\PropA)\}$ such that $\FmB \in \DesSeps{x}(\PropA)$
iff $\FmB_\AlgA(x) \in D$ and $\FmB \in \NDesSeps{x}(\PropA)$
iff $\FmB_\AlgA(x) \in \ValuesSetComp{A}{D}$;
in addition, set $\DesSeps{\dv}(\PropA) \SymbDef \{\PropA\}$ and $\NDesSeps{\dv}(\PropA) \SymbDef \EmptySet$.
It suf{f}ices then to consider only two possible configurations
for each $\DesSeps{x}(\PropA)$ and $\NDesSeps{x}(\PropA)$, depending on whether 
$\Sep_{\AlgA}(\uva) \in D$ and $\Sep_{\AlgA}(\uvb) \in \ValuesSetComp{A}{D}$,
or vice-versa, which we present in tabular form below:

\begin{table}[H]
    \centering
    \begin{tabular}{c|cc}
        \toprule
         $x$&$\DesSeps{x}(\PropA)$ & $\NDesSeps{x}(\PropA)$ \\
         \midrule
         $\dv$ & $\{\PropA\}$ & $\EmptySet$\\
         $\uva$ & $\{\Sep(\PropA)\}$ & $\{\PropA\}$\\
         $\uvb$ & $\EmptySet$ & $\{\PropA, \Sep(\PropA)\}$\\
         \bottomrule
    \end{tabular}\qquad
    \begin{tabular}{c|cc}
        \toprule
         $x$&$\DesSeps{x}(\PropA)$ & $\NDesSeps{x}(\PropA)$ \\
         \midrule
         $\dv$ & $\{\PropA\}$ & $\EmptySet$\\
         $\uva$ & $\EmptySet$ & $\{\PropA, \Sep(\PropA)\}$\\
         $\uvb$ & $\{\Sep(\PropA)\}$ & $\{\PropA\}$\\
         \bottomrule
    \end{tabular}
    \caption{The configuration at the left applies when $\Sep_{\AlgA}(\uva) \in D$ and $\Sep_{\AlgA}(\uvb) \in \ValuesSetComp{A}{D}$, while the one at the right applies when
    $\Sep_{\AlgA}(\uvb) \in D$ and $\Sep_{\AlgA}(\uva) \in \ValuesSetComp{A}{D}$.}
    \label{tab:seps_configs}
\end{table}

In any case, the following lemma, which is a particular case of the general result given in~\cite{marcelino2019}, guarantees that the
separators uniquely characterizes each truth value:

\begin{lemma}
\label{lem:caractvalues}
For all $\MatA$-valuation $\ValuationA$, $x \in A$ and $\FmA \in \LangSetA$,
\[
v(\FmA) = x \text{ if{f} }
v(\DesSeps{x}(\FmA)) \subseteq \{\dv\} \text{ and }
v(\NDesSeps{x}(\FmA)) \subseteq \{\uva, \uvb\}.
\]
\end{lemma}
\begin{proof}
The left-to-right direction easily follows from the definitions
of $\DesSeps{x}(\PropA)$ and $\NDesSeps{x}(\PropA)$.
For the converse, suppose that
$\ValuationA(\FmA) = y \neq x$.
Assume that
$\SepB \in \LangSetAp$ separates $x$ and $y$,
with $\SepB(\PropA) \in \DesSeps{x}(\PropA)$. Then
$\ValuationA(\SepB(\FmA)) = \SepB_\AlgA(\ValuationA(\FmA)) = \SepB_\AlgA(y) \in \{\uva,\uvb\}$,
hence $v(\DesSeps{x}(\FmA)) \not\subseteq \{\dv\}$.
The case when $\SepB(\PropA) \in \NDesSeps{x}(\PropA)$
is analogous.
\end{proof}

\begin{example}
Clearly, in a three-valued $\Sigma$-matrix $\langle \ThreeValuesSet, \{\tv\} \rangle$, 
when there is $\neg \in \Sigma_1$ interpreted either as $\lukNeg$ or $\HNeg$,
$\neg \PropA$
is a separator for $\uv$ and $\bv$,
since 
$\lukNeg(\uv) = {\uv} \notin \{\tv\}$ and $\lukNeg(\bv) = \tv \in \{\tv\}$.
\end{example}



Below we provide discriminators applicable to some of the prominent logics presented in Section~\ref{sec:CaseStudies}.
\begin{table}[H]
    \centering
    \begin{tabular}{c|c|c}
        \toprule
        \multicolumn{3}{c}{$\lukNeg$}\\
        \midrule
         $x$ & $\Omega_x^{\DSetOne}(p)$ & $\mho_x^{\DSetOne}(p)$\\
         \midrule
         $\bv$ & $\neg p$ & $p$\\
         $\uv$& $\EmptySet$ & $p, \neg p$\\
         $\tv$ & $p$ & $\EmptySet$\\
         \bottomrule
    \end{tabular}
    \quad
    \begin{tabular}{c|c|c}
        \toprule
        \multicolumn{3}{c}{$\lukNeg$}\\
        \midrule
         $x$ & $\Omega_x^{\DSetTwo}(p)$ & $\mho_x^{\DSetTwo}(p)$\\
         \midrule
         $\bv$ & $\EmptySet$ & $p$\\
         $\uv$& $p, \neg p$ & $\EmptySet$\\
         $\tv$ & $p$ & $\neg p$\\
         \bottomrule
    \end{tabular}
    \caption{Discriminators for matrices with $\lukNeg$.}
    \label{tab:lukdisc}
\end{table}

\begin{table}[H]
    \centering
    \begin{tabular}{c|c|c}
        \toprule
        \multicolumn{3}{c}{$\HNeg$}\\
        \midrule
         $x$ & $\Omega_x^{\DSetOne}(p)$ & $\mho_x^{\DSetOne}(p)$\\
         \midrule
         $\bv$ & $\neg p$ &$p$\\
         $\uv$& $\EmptySet$ & $p, \neg p$\\
         $\tv$ &$p$ & $\EmptySet$\\
         \bottomrule
    \end{tabular}
    \caption{Discriminator for matrices with $\HNeg$.}
\end{table}

The next example shows that non-monadicity
is not hard to witness: a slight change
in the matrix of $\GThreeName$ would do.
In Subsection~\ref{sec:nonmonadic},
we will see other examples 
of non-monadic matrices, as well
as \SetSet{} axiomatizations for them.

\begin{example}
The matrix resulting from the one of $\GThreeName$ by replacing the set of designated
values by $\DSetTwo$ is not monadic.
One can see this by producing by
composition 
all the unary operators on $\ThreeValuesSet$
using the connectives of $\GThreeName$,
and noticing that none of them
can separate $\tv$ from $\uv$.
Actually, a shorter path for such conclusion is by noticing that this matrix is not even reduced, hence it cannot be monadic~\cite{shoesmithsmiley1978}.
\end{example}

From now on, let $\AnalyticSetA \SymbDef \{\PropA, \Sep(\PropA)\}$
and assume, without loss of generality, that the leftmost configuration presented
in Table~\ref{tab:seps_configs} defines a discriminator for $\MatA$.
We now define a recipe to generate a finite $\AnalyticSetA$-analytical \SetSet{} calculus for $\MatA$.

\begin{definition}
\label{def:monadiccalc}
Define $\CalcA_\MatA$ 
as the \SetSet{} calculus containing, for each $k \in \omega$,
$\conn \in \Sigma^k$  and $x_1,\ldots,x_k \in A$, the
following rule schemas:
\begin{align*}
    \RuleA_{\conn, x_1,\ldots,x_k, T_y} \SymbDef \MCRule{\bigcup_{i=1}^k \DesSeps{x_i}(\PropA_i)}{\bigcup_{i=1}^k \NDesSeps{x_i}(\PropA_i), T_y(\conn(\PropA_1,\ldots,\PropA_k))},& \text{ for each } T_y \in \DesSeps{y}(\PropA), and\\
    \RuleA_{\conn, x_1,\ldots,x_k, F_y} \SymbDef \MCRule{\bigcup_{i=1}^k \DesSeps{x_i}(\PropA_i), F_y(\conn(\PropA_1,\ldots,\PropA_k))}{\bigcup_{i=1}^k \NDesSeps{x_i}(\PropA_i)},& \text{ for each } F_y \in \NDesSeps{y}(\PropA),\\
    &\text{where }y \SymbDef \conn_\AlgA(x_1,\ldots,x_k).
\end{align*}
\end{definition}

\begin{theorem}
\label{the:soundness}
$\CalcA_\MatA$ is sound with respect to $\MatA$, that is, 
$\SetSetCR_{\CalcA_{\MatA}} \subseteq \SetSetCR_{\MatA}$.
\end{theorem}
\begin{proof}\sloppy
Let $\conn \in \Sigma^k$, $x_1,\ldots,x_k \in A$ and $y \SymbDef \AlgInterp{\conn}{\AlgA}(x_1,\ldots,x_k)$.
Suppose that $\ValuationA$ is an $\MatA$-valuation
such that $\ValuationA(\bigcup_{i=1}^k \DesSeps{x_i}(\PropA_i)) \subseteq \{\dv\}$
and $\ValuationA(\bigcup_{i=1}^k \NDesSeps{x_i}(\PropA_i)) \subseteq \{\uva,\uvb\}$ (notice that a purported counter-model
for any of the schemas in Definition~\ref{def:monadiccalc}
must show this behaviour). Then, by Lemma~\ref{lem:caractvalues},
we have $v(\PropA_i) = x_i$ for all $1 \leq i \leq k$.
Now, looking at a schema of the form $\RuleA_{\conn, x_1,\ldots,x_k, T_y}$, with $y \SymbDef \AlgInterp{\conn}{\AlgA}(x_1,\ldots,x_k)$,
a counter-model for it would have $\ValuationA(T_y(\conn(\PropA_1,\ldots,\PropA_k))) \in \{\uva,\uvb\}$,
but then, by Lemma~\ref{lem:caractvalues} again,
this would mean that $v(\conn(\PropA_1,\ldots,\PropA_k)) \neq y$,
contradicting our hypothesis. The same reasoning may be
applied to the schemas of the form $\RuleA_{\conn, x_1,\ldots,x_k, F_y}$.
\end{proof}

For what follows, recall
that
we denote the 
complement of a \SetSet{} logic
$\SetSetCR{}$ by $\NSetSetCR{}$.

\begin{lemma}
\label{lem:comphelper}
Let $\SetCut \subseteq \LangSetA$
such that $\SetCut \NAnaSetSetCR{\CalcA_{\MatA}}{\AnalyticSetA} \FmSetComp{\SetCut}$. Then there is
an $\MatA$-valuation $\ValuationA$
such that $v(\SetCut) \subseteq \{\dv\}$
and $v(\FmSetComp\SetCut) \subseteq \{\uva,\uvb\}$.
\end{lemma}
\begin{proof}
Consider the mapping $\ValuationA' : P \to A$
such that
\[
\ValuationA'(\PropA) \SymbDef
\begin{cases}
    \dv & \PropA \in \SetCut\\
    \uva & \PropA \not\in \SetCut, \Sep(\PropA) \not\in \SetCut\\
    \uvb & \PropA \not\in \SetCut, \Sep(\PropA) \in \SetCut\\
\end{cases}
\]
We will prove by induction on the structure of formulas that the extension $\ValuationA$ of $\ValuationA'$
to the full language
is such that $P(\FmA) \SymbDef ``\ValuationA(\FmA) = x$ if{f}
$\DesSeps{x}(\FmA) \subseteq \SetCut$
and $\NDesSeps{x}(\FmA) \subseteq \FmSetComp\SetCut$''
for all $\FmA \in \LangSetA$.
As a consequence, since 
$\DesSeps{\dv}(\PropA) = \{\PropA\}$ and 
$\NDesSeps{\dv}(\PropA) =\EmptySet$,
we will have that $v(\FmA) = \dv$ if{f}
$\FmA \in \SetCut$, proving the present lemma.
Proceeding with the induction, then,
the base case is obvious from the definition
of $\ValuationA'$. In the inductive step,
let $\FmA \SymbDef \conn(\FmB_1,\ldots,\FmB_k)$,
for a $\conn \in \Sigma^k$, and
assume that $P(\FmB_i)$ holds for all $1 \leq i \leq k$.
From the left to the right,
suppose that $\ValuationA(\conn(\FmB_1,\ldots,\FmB_k)) = x$,
hence $\AlgInterp{\conn}{\AlgA}(\ValuationA(\FmB_1),\ldots,\ValuationA(\FmB_k)) = x$, and call $x_i \SymbDef v(\FmB_i)$.
By the induction hypothesis, we have that
$\DesSeps{x_i}(\FmB_i) \subseteq \SetCut$
and $\NDesSeps{x_i}(\FmB_i) \subseteq \FmSetComp\SetCut$,
and, by (the contrapositive of) monotonicity of $\AnaSetSetCR{\CalcA_{\MatA}}{\AnalyticSetA}$,
we have
$\bigcup_i\DesSeps{x_i}(\FmB_i) \NAnaSetSetCR{\CalcA_{\MatA}}{\AnalyticSetA} \bigcup_i\NDesSeps{x_i}(\FmB_i)$. Assume, by contradiction, that
$\FmC \in \DesSeps{x}(\FmA)\setminus\SetCut$,
thus $\bigcup_i\DesSeps{x_i}(\FmB_i) \NAnaSetSetCR{\CalcA_{\MatA}}{\AnalyticSetA} \bigcup_i\NDesSeps{x_i}(\FmB_i), \FmC$,
but this goes against the rule 
$\RuleA_{\conn, x_1,\ldots,x_k, \FmC_x}$
present in $\CalcA_\MatA$. A similar argument
applies when we assume that
$\FmC \in \NDesSeps{x}(\FmA)\setminus\FmSetComp\SetCut$.
The other direction easily follows by
monotonicity of $\AnaSetSetCR{\CalcA_{\MatA}}{\AnalyticSetA}$ again and Lemma~\ref{lem:caractvalues}.
\end{proof}


\begin{theorem}
\label{the:completeness}
$\CalcA_\MatA$ is a finite $\{\PropA, \Sep(\PropA)\}$-analytical axiomatization for $\MatA$.
\end{theorem}
\begin{proof}
Finiteness of $\CalcA_\MatA$ is obvious since
$\MatA$ is a finite logical matrix.
Soundness follows 
by Theorem~\ref{the:soundness} and the fact that
$\AnaSetSetCR{\CalcA_\MatA}{\AnalyticSetA} \subseteq \SetSetCR_{\CalcA_\MatA}$.
It remains to prove completeness, i.e.,
that $\SetSetCR_\MatA \subseteq \AnaSetSetCR{\CalcA_\MatA}{\AnalyticSetA}$.
Suppose that $\FmSetA \NAnaSetSetCR{\CalcA_\MatA}{\AnalyticSetA} \FmSetB$,
then, by cut for sets, there is $\SetCut \subseteq \LangSetA$
such that
$\SetCut, \FmSetA \NAnaSetSetCR{\CalcA_\MatA}{\AnalyticSetA} \FmSetB, \FmSetComp\SetCut$. 
Then Lemma~\ref{lem:comphelper} applies and
the provided $\MatA$-valuation makes $v(\FmSetA) \subseteq \{\dv\}$
and $v(\FmSetB) \subseteq \{\uva,\uvb\}$, as desired.
\end{proof}

It is worth emphasizing the modularity of the
proposed approach: each connective is axiomatized by its own set of rule schemas,
thus axiomatizing a matrix expanded with new connectives is a matter
of adding new rule schemas corresponding to the new connectives.
The examples to be presented in a moment will show that
Definition~\ref{def:monadiccalc} may deliver redundant and 
complex rule schemas. This fact motivates the next result,
which introduces some streamlining steps that highly simplify
the generated axiomatization.
The proof uses, in essence, the properties
of \SetSet{} logics.
In what follows, we say that a rule schema is a \emph{dilution} of another when the former can be obtained from the latter by adding formulas in the antecedent or in the succedent (recall property (D) in Definition~\ref{def:SetSet}). In that case, we may also say that the latter is a \emph{subschema} of the former.
In what follows, denote by
$\FmlasTree{\TreeA}$ the set of formulas
appearing in the $\CalcA$-derivation $\TreeA$.

\begin{theorem}
\label{the:streamlining}
Let $\CalcA$ be a $\AnalyticSetA$-analytical calculus
for a matrix $\MatA$
and consider the following streamlining subprocedure:
\begin{description}
    \item[(S1)] Remove all instances of overlap.
    \item[(S2)] Remove all rule schemas that are dilutions of other rule schemas in $\CalcA$.
    \item[(S3)] Replace occurrences of pairs of rule schemas
    $\MCRule{\FmSetA, \FmA}{\FmSetB}$
    and $\MCRule{\FmSetA}{\FmSetB, \FmA}$
    by $\MCRule{\FmSetA}{\FmSetB}$.
    \item[(S4)] Replace a rule of inference by one of its subrules that is
    derivable 
    using the rules of the calculus.
    \item[(S5)] Remove a rule of inference $\RuleA \SymbDef \MCRule{\FmSetA}{\FmSetB}$ that has a proof $\TreeA$ 
    using the others rules of $\CalcA$,
    such that,
    for all $\FmSetD \subseteq \LangSetA$
    and all
    $\sigma : P \to \Subf{\FmSetD}$,
    if $\sigma(\FmSetA),\sigma(\FmSetB) \subseteq \SubfTheta{\AnalyticSetA}{\FmSetD}$, then
    $\sigma(\FmlasTree{\TreeA}) \subseteq \SubfTheta{\AnalyticSetA}{\FmSetD}$.
    
\end{description}
Then \textbf{(S1)}, \textbf{(S2)}, \textbf{(S3)}, \textbf{(S4)} 
and \textbf{(S5)}
preserve adequacy
for $\MatA$
and $\AnalyticSetA$-analyticity.
\end{theorem}

In order to use a terminology
that is closer to that
used in 3-labelled calculi,
if $\LName$ is determined
by $\MatA$, then 
we will sometimes denote $\CalcA_\MatA$
by $\smName{\LName}$
and the streamlined version
of this calculus by
$\smcName{\LName}$.
We now exemplify the axiomatization procedure and the streamlining procedures just described, producing \SetSet{} axiomatizations for the monadic matrices among those three-valued matrices introduced in
the beginning of the present work. Assume that $\dv = \tv$, $\uva = \uv$ and $\uvb = \bv$ for the next examples.

\begin{example}
\label{ex:lukneg}
Consider the three-valued matrix with $\{\tv\}$ as the set of designated values over a signature containing only
$\neg$, 
interpreted as $\lukNeg$.
Notice that the configuration
at the left of Table~\ref{tab:lukdisc}
applies in this case. According to Definition~\ref{def:monadiccalc}, then,
the following rule schemas axiomatize $\neg$ in this matrix:
\begin{gather*}
\RuleA_{\neg, \tv, \neg\PropA} = \MCRule{\PropA}{\neg\neg\PropA} \quad
\RuleA_{\neg, \tv, \PropA} = \MCRule{\PropA, \neg\PropA}{}
\quad
\RuleA_{\neg, \uv, \PropA} = \MCRule{\neg\PropA}{\PropA, \neg\PropA}
\quad
\RuleA_{\neg, \uv, \neg\PropA} = \MCRule{\neg\neg\PropA}{\PropA, \neg\PropA}
\quad
\RuleA_{\neg, \bv, \PropA} = \MCRule{\neg\PropA}{\PropA,\neg\PropA}
\end{gather*}
By Theorem~\ref{the:streamlining}, we may remove schemas
$\RuleA_{\neg, \uv, \PropA}$ and 
$\RuleA_{\neg, \bv, \PropA}$ as they are instances of overlap,
as well as replace $\RuleA_{\neg, \uv, \neg\PropA}$
with $\MCRule{\neg\neg\PropA}{\PropA}$, as the latter is
a subschema of the former easily derivable from 
$\RuleA_{\neg, \uv, \neg\PropA}$ itself
and $\RuleA_{\neg, \tv, \PropA}$. The simplified axiomatization 
for $\lukNeg$, then, contains the following rule schemas:
\begin{gather*}
\MCRule{\PropA}{\neg\neg\PropA}\RuleA^{\lukNeg}_{1} \quad
\MCRule{\neg\neg\PropA}{\PropA}\RuleA^{\lukNeg}_{2} \quad
\MCRule{\PropA, \neg\PropA}{}\RuleA^{\lukNeg}_{3}
\end{gather*}
\end{example}

\begin{example}
\label{ex:minmaxaxiomat}
Now consider the matrix of Example~\ref{ex:lukneg}
expanded with connectives $\land$ and $\lor$, interpreted
respectively as $\lukConj$ and $\lukDisj$.
Let us axiomatize $\land$ in detail.
By employing Definition~\ref{def:monadiccalc}, we produce
the following rule schemas:
\begin{gather*}
\RuleA_{\land, \tv, \tv, \PropA} = \MCRule{\PropA, \PropB}{\PropA\land\PropB}
\quad
\RuleA_{\land, \tv, \uv, \PropA} = \MCRule{\PropA,\PropA\land\PropB}{\PropB,\neg\PropB}
\quad
\RuleA_{\land, \tv, \uv, \neg\PropA} = \MCRule{\PropA,\neg(\PropA\land\PropB)}{\PropB,\neg\PropB}\\
\RuleA_{\land, \tv, \bv, \PropA} = \MCRule{\PropA,\neg\PropB,\PropA\land\PropB}{\PropB}
\quad
\RuleA_{\land, \tv, \bv, \neg\PropA} = \MCRule{\PropA,\neg\PropB}{\PropB,\neg(\PropA\land\PropB)}
\quad
\RuleA_{\land, \uv, \tv, \PropA} = \MCRule{\PropB,\PropA\land\PropB}{\PropA,\neg\PropA}\\
\RuleA_{\land, \uv, \tv, \neg\PropA} = \MCRule{\PropB,\neg(\PropA\land\PropB)}{\PropA,\neg\PropA}
\quad
\RuleA_{\land, \uv, \uv, \PropA} = \MCRule{\PropA\land\PropB}{\PropA,\neg\PropA,\PropB,\neg\PropB}
\quad
\RuleA_{\land, \uv, \uv, \neg\PropA} = \MCRule{\neg(\PropA\land\PropB)}{\PropA,\neg\PropA,\PropB,\neg\PropB}\\
\RuleA_{\land, \uv, \bv, \PropA} = \MCRule{\neg\PropB,\PropA\land\PropB}{\PropA,\neg\PropA,\PropB}
\quad
\RuleA_{\land, \uv, \bv, \neg\PropA} = \MCRule{\neg\PropB}{\PropA,\neg\PropA,\PropB,\neg(\PropA\land\PropB)}
\quad
\RuleA_{\land, \bv, \tv, \neg\PropA} = \MCRule{\neg\PropA, \PropB}{\PropA,\neg(\PropA\land\PropB)}\\
\RuleA_{\land, \bv, \tv, \PropA} = \MCRule{\neg\PropA, \PropB, \PropA\land\PropB}{\PropA}
\quad
\RuleA_{\land, \bv, \uv, \neg\PropA} = \MCRule{\neg\PropA}{\PropA,\PropB,\neg\PropB, \neg(\PropA\land\PropB)}
\quad
\RuleA_{\land, \bv, \uv, \PropA} = \MCRule{\neg\PropA,\PropA\land\PropB}{\PropA,\PropB,\neg\PropB}\\
\RuleA_{\land, \bv, \bv, \neg\PropA} = \MCRule{\neg\PropA,\neg\PropB}{\PropA,\PropB,\neg(\PropA\land\PropB)}
\quad
\RuleA_{\land, \bv, \bv, \PropA} = \MCRule{\neg\PropA,\neg\PropB,\PropA\land\PropB}{\PropA,\PropB}
\end{gather*}

\noindent Over these schemas, we apply
Theorem~\ref{the:streamlining}
to obtain a much shorter version of this calculus:
\begin{gather*}
\MCRule{\PropA, \PropB}{\PropA\land\PropB}\RuleA^{\lukConj}_{1}
\quad
\MCRule{\PropA\land\PropB}{\PropA}\RuleA^{\lukConj}_{2}
\quad
\MCRule{\PropA\land\PropB}{\PropB}\RuleA^{\lukConj}_{3}
\quad
\MCRule{\neg(\PropA\land\PropB)}{\neg\PropA,\neg\PropB}\RuleA^{\lukConj}_{4}
\quad
\MCRule{\neg\PropA}{\neg(\PropA\land\PropB)}\RuleA^{\lukConj}_{5}
\quad
\MCRule{\neg\PropB}{\neg(\PropA\land\PropB)}\RuleA^{\lukConj}_{6}
\end{gather*}
\noindent We may then follow the
same procedure for axiomatizing $\lor$,
and the result will be:
\begin{gather*}
\MCRule{\PropA}{\PropA\lor\PropB}\RuleA^{\lukDisj}_{1}
\quad
\MCRule{\PropB}{\PropA\lor\PropB}\RuleA^{\lukDisj}_{2}
\quad
\MCRule{\PropA\lor\PropB}{\PropA,\PropB}\RuleA^{\lukDisj}_{3}
\quad
\MCRule{\neg\PropA,\neg\PropB}{\neg(\PropA\lor\PropB)}\RuleA^{\lukDisj}_{4}
\quad
\MCRule{\neg(\PropA\lor\PropB)}{\neg\PropA}\RuleA^{\lukDisj}_{5}
\quad
\MCRule{\neg(\PropA\lor\PropB)}{\neg\PropB}\RuleA^{\lukDisj}_{6}
\end{gather*}
\end{example}

With axiomatizations of $\lukNeg$, $\lukConj$
and $\lukDisj$ in hand, it is straighforward now
to produce a calculus for the logics $\LThreeName$
and $\SKName$: we just need to axiomatize
$\lukImp$ and $\KleeImp$, which
we do more briefly in the next example.

\begin{example}
According to Definition~\ref{def:monadiccalc}
and the streamlining procedures proposed in
Theorem~\ref{the:streamlining},
$\lukImp$ is axiomatized by the following
rule schemas:
\begin{gather*}
\MCRule{\neg\PropA}{\PropA\to\PropB}\RuleA^{\lukImp}_{1}
\;
\MCRule{\PropB}{\PropA\to\PropB}\RuleA^{\lukImp}_{2}
\;
\MCRule{}{\PropA\to\PropB,\neg\PropB,\PropA}\RuleA^{\lukImp}_{3}
\;
\MCRule{\neg\PropB,\PropA}{\neg(\PropA\to\PropB)}\RuleA^{\lukImp}_{4}\\
\MCRule{\PropA\to\PropB,\neg\PropB}{\neg\PropA}\RuleA^{\lukImp}_{5}
\;
\MCRule{\neg(\PropA\to\PropB)}{\neg\PropB}\RuleA^{\lukImp}_{6}
\;
\MCRule{\neg(\PropA\to\PropB)}{\PropA}\RuleA^{\lukImp}_{7}
\;
\MCRule{\PropA,\PropA\to\PropB}{\PropB}\RuleA^{\lukImp}_{8}
\end{gather*}
In turn, $\KleeImp$ is axiomatized
by:
\begin{gather*}
\MCRule{\neg\PropA}{\PropA\to\PropB}\RuleA^{\KleeImp}_{1}
\;
\MCRule{\PropB}{\PropA\to\PropB}\RuleA^{\KleeImp}_{2}
\;
\MCRule{\neg\PropB, \PropA}{\neg(\PropA\to\PropB)}\RuleA^{\KleeImp}_{3}
\;
\MCRule{\PropA\to\PropB}{\neg\PropA,\PropB}\RuleA^{\KleeImp}_{4}
\;
\MCRule{\neg(\PropA\to\PropB)}{\neg\PropB}\RuleA^{\KleeImp}_{5}
\;
\MCRule{\neg(\PropA\to\PropB)}{\PropA}\RuleA^{\KleeImp}_{6}
\end{gather*}

\noindent Notice that, in this case, instead of 
$\neg\PropA$, we could have
used $\PropA\to\PropA$ to
separate the values
$\bv$ and $\uv$, since
$\bv \KleeImp \bv = \tv$ is designated
and $\uv \KleeImp \uv = \uv$ is non-designated. The produced rules,
after simplification, would be
the following:
\begin{gather*}
    \MCRule{\PropB\to\PropB, \PropA}{(\PropA\to\PropB)\to(\PropA\to\PropB)}
    \quad
    \MCRule{\PropA\to\PropB}{\PropA\to\PropA, \PropB}
    \quad
    \MCRule{\PropB}{\PropA\to\PropB}\\
    \MCRule{(\PropA\to\PropB)\to(\PropA\to\PropB)}{\PropA\to\PropB,\PropB\to\PropB}
    \quad
    \MCRule{(\PropA\to\PropB)\to(\PropA\to\PropB)}{\PropA\to\PropB,\PropA}
    \MCRule{\PropA\to\PropA}{\PropA\to\PropB, \PropA}\quad
    \MCRule{\PropA\to\PropB,\PropA}{\PropB}
\end{gather*}
\end{example}

\begin{example}
Consider now $\HNeg$, which changes the
output under $\uv$ to $\bv$ when compared
with $\lukNeg$. Definition~\ref{def:monadiccalc}
produces the following rules schemas:
\begin{gather*}
\RuleA_{\neg, \tv, \neg\PropA} = \MCRule{\PropA}{\neg\neg\PropA} 
\;
\RuleA_{\neg, \tv, \PropA} = \MCRule{\PropA, \neg\PropA}{}
\;
\RuleA_{\neg, \uv, \PropA} = \MCRule{\neg\PropA}{\PropA, \neg\PropA}
\;
\RuleA_{\neg, \uv, \neg\PropA} = \MCRule{}{\PropA, \neg\PropA,\neg\neg\PropA}
\;
\RuleA_{\neg, \bv, \PropA} = \MCRule{\neg\PropA}{\PropA,\neg\PropA}
\end{gather*}
Notice that $\RuleA_{\neg, \uv, \PropA}$
and $\RuleA_{\neg, \bv, \PropA}$
may be removed as they are instances
of overlap. Also, 
$\MCRule{}{\neg\PropA,\neg\neg\PropA}$
is derivable from the other rules,
so we may use it instead of $\RuleA_{\neg, \uv, \neg\PropA}$. Finally, in this new calculus,
we may derive $\RuleA_{\neg, \tv, \neg\PropA}$,
so it can be removed. The final
axiomatization for $\HNeg$ is then
\begin{gather*}
\MCRule{\PropA, \neg\PropA}{}\RuleA^{\HNeg}_{1}
\quad
\MCRule{}{\neg\PropA,\neg\neg\PropA}\RuleA^{\HNeg}_{2}
\end{gather*}
\end{example}

We have seen many examples of axiomatizations
for three-valued logical matrices having
a single designated value. Now we indicate how to extend our approach to matrices with
two designated values.
Let $\MatA_1 = \langle \{\dv,\uva,\uvb\}, \{\uva,\uvb\} \rangle$. We may obtain a matrix $\MatA_2$ with a single
designated value from $\MatA_1$ by simply
taking the set $\{\dv\}$ as designated:
$\MatA_2 = \langle \{\dv,\uva,\uvb\}, \{\dv\} \rangle$. We know then that, if $\MatA_2$ is monadic, there are quasi-partitions $(\DesSeps{x}^2(\PropA),\NDesSeps{x}^2(\PropA))$ 
for each $x$ and that
$v(\FmA) = x \text{ if{f} }
v(\DesSeps{x}^2(\FmA)) \subseteq \{\dv\} \text{ and }
v(\NDesSeps{x}^2(\FmA)) \subseteq \{\uva, \uvb\}$,
from which, by Definition~\ref{def:monadiccalc},
we obtain an adequate $\{\PropA,\Sep(\PropA)\}$-analytic calculus $\CalcA_2$ for $\MatA_2$.
If we take $(\DesSeps{x}^1(\PropA),\NDesSeps{x}^1(\PropA))$ for each $x$ such that
$\DesSeps{x}^1(\PropA) \SymbDef \NDesSeps{x}^2(\PropA)$
and $\NDesSeps{x}^1(\PropA) \SymbDef \DesSeps{x}^2(\PropA)$,
then
$v(\FmA) = x \text{ if{f} }
v(\DesSeps{x}^1(\FmA)) \subseteq \{\uva, \uvb\}  \text{ and }
v(\NDesSeps{x}^1(\FmA)) \subseteq \{\dv\}$,
thus we have separators for $\MatA_1$.
We may use them in Definition~\ref{def:monadiccalc}
to build a $\{\PropA, \Sep(\PropA)\}$-analytical  calculus $\CalcA_1$ which can be easily
proved to be sound and complete with respect
to $\MatA_1$ by simply adapting
the proofs of Lemma~\ref{lem:comphelper} and Theorem~\ref{the:completeness}.
But we can easily see that the schemas in
$\CalcA_1$ are just the schemas in $\CalcA_2$
turned upside down, and this is how our approach
simply extends to the case of two designated values.

\begin{example}
Consider the three-valued matrix for the Logic of Paradox ($\LPName$),
which has the same signature and interpretation structure as
the logic $\SKName$, but, instead of having $\{\tv\}$
as set of designated values, it has
$\{\tv,\uv\}$. According to what we have presented
above, we may produce an axiomatization
for the matrix with the same interpretations
but with set of designated values
being $\{\bv\}$,
following Definition~\ref{def:monadiccalc}
and Theorem~\ref{the:streamlining},
and then simply turn the rules upside down to axiomatize $\LPName$.
\end{example}

\begin{example}\label{ex:postaxiomat}
    Post's logic $\PostName$ is axiomatized by
    the rules for $\land$ and $\lor$
    in $\LThreeName$
    (see~Example \ref{ex:minmaxaxiomat}) together with
    the following rules for $\neg$,
    again produced using the described
    procedure and simplifications:
    \[
\MCRule{}{\neg\neg\PropA,\neg\PropA,\PropA}\RuleA^{\postNeg}_{1}
\qquad
\MCRule{\neg\neg\PropA,\PropA}{}\RuleA^{\postNeg}_{2}
\qquad
\MCRule{\neg\PropA,\PropA}{}\RuleA^{\postNeg}_{3}
\]
\end{example}

\begin{remark}
The recipe we have described and exemplified throughout
this section is a simplification 
of the general method for axiomatizing
finite monadic matrices given in~\cite{shoesmithsmiley1978},
which was recently generalized
to finite monadic
partial non-deterministic logical matrices in~\cite{marcelinowollic19}.
\end{remark}

\subsection{Axiomatization of non-monadic matrices}
\label{sec:nonmonadic}

As we have been insisting on, non-monadic deterministic logical matrices can always be axiomatized in \SetSet{}~\cite[Theorem 19.12]{shoesmithsmiley1978} using the original language, although
not via our recipe. 
There is no known algorithm for this general case.
By employing separators with contexts~\cite[Section 19.5]{shoesmithsmiley1978}
in an \textit{ad hoc} manner, however, we managed to show \SetSet{} $\Theta$-analytic axiomatizations --- now allowing $\Theta$
to contain formulas of arities other than 1, for instance $\PropA\lor\PropB$
and $\PropA\to\PropB$ --- for fragments of some of the
logics presented in Section~\ref{sec:CaseStudies}.
We present them in the sequel.

\begin{remark}
We should mention that our recipe for monadic matrices can always be applied to non-monadic matrices
provided that the object language is extended with connectives guaranteeing the separation of all truth values. The resulting matrix is then monadic
and the logic
is a conservative extension of the original logic~\cite{marcelino2019}.
\end{remark}

\subsubsection{Disjunctive fragment of Bochvar-Kleene's logic}

Let $\SetSetCR_{\BKName}^{\lor}$ be the \SetSet{} logic determined by
the three-valued logical matrix $\MatA_{\BKName}^\lor$ whose signature contains only $\lor$, interpreted as $\BKleeDisj$ in this case --- the truth table of which we 
repeat below for convenience --- and whose set
of designated values is $\{\tv\}$. The reader can easily verify
that this matrix is not monadic by using the fact that the sets $\{\bv\}$ and 
$\{\uv\}$ are both closed under $\BKleeDisj$. Although our previous
recipe does not apply in this scenario, we are still able
to provide an analytic axiomatization for this matrix.

\begin{table}[H]
    \centering
    \begin{tabular}{c|ccc}
        \toprule
         \BKleeDisj & \bv & \uv & \tv\\
        \midrule
         \bv & \bv & \uv & \tv\\
         \uv & \uv & \uv & \uv \\
         \tv & \tv & \uv & \tv\\
        \bottomrule
    \end{tabular}
\end{table}

\begin{definition}
Let $\CalcA^\lor_{\BKName}$ be the \SetSet{} calculus 
whose rule schemas are the following:
\begin{gather*}
    \MCRule{\PropA \lor \PropB}{\PropA,\PropB}\RuleA^{\lor,1}_{\BKName}
    \quad
    \MCRule{\PropA,\PropB}{\PropA \lor \PropB}\RuleA^{\lor,2}_{\BKName}\\
    \MCRule{\PropA_1\lor\PropB_1, \PropA_2}{\PropA_1\lor\PropA_2}\RuleA^{\lor,3}_{\BKName}
    \quad
    \MCRule{\PropA_2\lor\PropB_1, \PropA_1}{\PropA_1\lor\PropA_2}\RuleA^{\lor,4}_{\BKName}
    \quad
    \MCRule{\PropA_1\lor\PropA_2,\PropB_1}{\PropA_2\lor\PropB_1}\RuleA^{\lor,5}_{\BKName}\\
    \MCRule{\PropA_1\lor\PropB_1,\PropA_2\lor\PropB_1}{(\PropA_1\lor\PropA_2)\lor\PropB_1,\PropA_1}\RuleA^{\lor,6}_{\BKName}
    \quad
    \MCRule{(\PropA_1\lor\PropA_2)\lor\PropB_1,\PropB_1}{\PropA_1\lor\PropB_1}\RuleA^{\lor,7}_{\BKName}
    \quad
    \MCRule{(\PropA_1\lor\PropA_2)\lor\PropB_1,\PropB_1}{\PropA_2\lor\PropB_1}\RuleA^{\lor,8}_{\BKName}
\end{gather*}
\end{definition}

\begin{theorem}
$\CalcA^\lor_{\BKName}$ is sound for $\MatA_{\BKName}^\lor$.
\end{theorem}
\begin{proof}
This is routine.
\end{proof}

\begin{lemma}
Some derivable rules:
\[
    \MCRule{\PropA_1\lor\PropA_2,\PropB_1}{\PropA_1\lor\PropB_1}\RuleA^{\lor,9}_{\BKName}\\
    \quad
    \MCRule{\PropA_1\lor\PropA_2}{\PropA_2\lor\PropA_1}\RuleA^{\lor,10}_{\BKName}\\
\]
\end{lemma}

\begin{theorem}
$\CalcA^\lor_{\BKName}$ is complete for $\MatA_{\BKName}^\lor$
and $\{\PropA \lor \PropB\}$-analytic.
\end{theorem}
\begin{proof}
Let $\AnalyticSetA \SymbDef \{\PropA \lor \PropB\}$.
Suppose that $\FmSetA \NAnaSetSetCR{\CalcA^\lor_{\BKName}}{\AnalyticSetA} \FmSetB$,
then, by cut for sets, there is $\SetCut \subseteq \LangSetA$
such that
$\SetCut \NAnaSetSetCR{\CalcA^\lor_{\BKName}}{\AnalyticSetA} \FmSetComp\SetCut$,
with $\FmSetA \subseteq \SetCut$ and $\FmSetB \subseteq \FmSetComp\SetCut$.
Let $\GenSubs{\AnalyticSetA} \SymbDef \Subf{\FmSetA \cup \FmSetB} \cup \{\FmA \lor \FmB \mid \FmA,\FmB \in \Subf{\FmSetA \cup \FmSetB}\}$.
Then $\TPart \NAnaSetSetCR{\CalcA^\lor_{\BKName}}{\AnalyticSetA} \FPart$
for some partition $(\TPart, \FPart)$ of $\GenSubs{\AnalyticSetA}$,
with $\FmSetA \subseteq \TPart$ and $\FmSetB \subseteq \FPart$.
We consider two cases, defining
for each of them a mapping 
$v : \Subf{\FmSetA\cup\FmSetB}\to\ThreeValuesSet$
and proving that $v(\FmA\lor\FmB) = v(\FmA) \BKleeDisj v(\FmB)$:
\begin{itemize}
    \item (a): for all $\FmA,\FmB$, if $\FmA \in \TPart \cap \Subf{\FmSetA\cup\FmSetB}$
    and $\FmB \in \FPart \cap \Subf{\FmSetA\cup\FmSetB}$,
    then $\FmA\lor\FmB \in \TPart$. Let $v : \Subf{\FmSetA\cup\FmSetB}\to\ThreeValuesSet$ such that
    \[
        v(\FmA) \SymbDef
        \begin{cases}
            \tv & \text{ if } \FmA \in \TPart\\
            \bv & \text{ if } \FmA \in \FPart\\
        \end{cases}
    \]
    \begin{itemize}
        \item if $v(\FmA\lor\FmB)=\tv$,
        then, by $(\RuleA^{\lor,1}_{\BKName})$, either
        $v(\FmA)=\tv$ or $v(\FmB)=\tv$, and we are done,
        as the range of $v$ is $\{\tv,\bv\}$;
        \item if $v(\FmA)=\tv$, in case $v(\FmB)=\tv$, use $(\RuleA^{\lor,2}_{\BKName})$, and in case
        $v(\FmB)=\bv$, use the assumption (a).
        The argument is analogous for $v(\FmB)=\tv$, with
        the help of $(\RuleA^{\lor,10}_{\BKName})$;
        \item if $v(\FmA\lor\FmB)=\bv$, 
        use the contrapositive of (a),
        which gives us that either $\FmA \in \FPart$ 
        or $\FmB \in \TPart$. 
        We want to conclude $\FmA,\FmB \in \FPart$
        and for that it is enough to show $\FmB \in \FPart$.
        Assume that $\FmB \in \TPart$
        and consider two subcases:
        either $\FmA \in \TPart$
        or $\FmA \in \FPart$.
        The first case is absurd by
        $(\RuleA^{\lor,2}_{\BKName})$.
        The second case produces,
        by (a), $\FmB \lor \FmA \in \TPart$,
        but then by 
        $(\RuleA^{\lor,10}_{\BKName})$
        we have $\FmA\lor\FmB \in \TPart$,
        absurd again.
        Thus $\FmB \in \FPart$
        and we are done.
        \item if $v(\FmA)=\bv$ and $v(\FmB)=\bv$, by
        $(\RuleA^{\lor,1}_{\BKName})$ we have $\FmA\lor\FmB \in \FPart$.
    \end{itemize}
    \item for some $\FmB_\tv \in \TPart \cap \Subf{\FmSetA\cup\FmSetB}$
    and $\FmB_{\uv} \in \FPart \cap \Subf{\FmSetA\cup\FmSetB}$,
    we have $\FmB_\tv\lor\FmB_{\uv} \in \FPart$.  
     Let $v : \Subf{\FmSetA\cup\FmSetB}\to\ThreeValuesSet$ such that
     \[
        v(\FmA) \SymbDef
        \begin{cases}
            \tv & \text{ if } \FmA \in \TPart\\
            \uv & \text{ if } \FmA, \FmA\lor\FmB_\tv \in \FPart\\
            \bv & \text{ if } \FmA\in\FPart, \FmA\lor\FmB_\tv \in \TPart\\
        \end{cases}
    \]    
    \begin{itemize}
        \item if $v(\FmA \lor \FmB) = \uv$,
        then $\FmA\lor\FmB, (\FmA\lor\FmB)\lor\FmB_\tv \in \FPart$.
        By $(\RuleA^{\lor,1}_{\BKName})$, either $\FmA \in \FPart$
        or $\FmB \in \FPart$. In the first
        case, by {$(\RuleA^{\lor,6}_{\BKName})$}, we have
        $\FmA\lor\FmB_\tv \in \FPart$
        or $\FmB\lor\FmB_\tv \in \FPart$.
        The former entails that
        $v(\FmA)=\uv$, while the latter, by {$(\RuleA^{\lor,5}_{\BKName})$}, gives that $\FmB\lor\FmB \in \FPart$,
        which, by rule {$(\RuleA^{\lor,1}_{\BKName})$}, gives $\FmB\in\FPart$,
        and thus $v(\FmB)=\uv$.
        In the second case, with $\FmB\in\FPart$,
        by rule {$(\RuleA^{\lor,3}_{\BKName})$}, $\FmA\lor\FmB_\tv \in \FPart$.
        As $\FmB_\tv \in \TPart$,
        by rule {$(\RuleA^{\lor,2}_{\BKName})$} we have $\FmA \in \FPart$,
        and then $v(\FmA)=\uv$, as desired;
        \item if $v(\FmA) = \uv$, then
        $\FmA, \FmA\lor\FmB_\tv \in \FPart$, and, by rule {$(\RuleA^{\lor,9}_{\BKName})$}, we have $\FmA\lor\FmB\in\FPart$; hence, by rule {$(\RuleA^{\lor,7}_{\BKName})$}, we obtain $(\FmA\lor\FmB)\lor\FmB_\tv\in\FPart$, thus $v(\FmA\lor\FmB)=\uv$. The case $v(\FmB)=\uv$ is analogous, using
        rules {$(\RuleA^{\lor,5}_{\BKName})$} and {$(\RuleA^{\lor,8}_{\BKName})$} instead;
        \item if $v(\FmA\lor\FmB)=\bv$, 
        then $\FmA\lor\FmB\in\FPart$ and $(\FmA\lor\FmB)\lor\FmB_\tv \in \TPart$.
        By {$(\RuleA^{\lor,7}_{\BKName})$} and {$(\RuleA^{\lor,8}_{\BKName})$}, $\FmA\lor\FmB_\tv \in \TPart$ and $\FmB\lor\FmB_\tv \in \TPart$
        respectively.
        By {$(\RuleA^{\lor,3}_{\BKName})$} and {$(\RuleA^{\lor,4}_{\BKName})$}, then, we have
        $v(\FmB)=\bv$ and $v(\FmA)=\bv$,
        as desired;
        \item if $v(\FmA) = \bv$ and $v(\FmB)=\bv$, then $\FmA,\FmB\in\FPart$ and
        $\FmA\lor\FmB_\tv,\FmB\lor\FmB_\tv \in \TPart$. By {$(\RuleA^{\lor,1}_{\BKName})$}, we have $\FmA\lor\FmB\in\FPart$,
        and, by {$(\RuleA^{\lor,6}_{\BKName})$}, we get $(\FmA\lor\FmB)\lor\FmB_\tv \in \TPart$,
        entailing that $v(\FmA\lor\FmB)=\bv$;
        \item if $v(\FmA\lor\FmB)=\tv$,
        by {$(\RuleA^{\lor,9}_{\BKName})$} and {$(\RuleA^{\lor,5}_{\BKName})$}, we have $\FmA\lor\FmB_\tv \in \FPart$ and $\FmB\lor\FmB_\tv \in \FPart$, thus $v(\FmA) \neq \uv$
        and $v(\FmB)\neq\uv$.
        Then, by {$(\RuleA^{\lor,1}_{\BKName})$}, we have $\FmA\in \TPart$
        or $\FmB\in\TPart$, and hence,
        in either case, $v(\FmA\lor\FmB)=v(\FmA)\BKleeDisj v(\FmB)$;
        \item if $v(\FmA) = \tv$ and $v(\FmB) \neq \uv$, we have $\FmA \in \TPart$ and either $\FmB \in \TPart$
        or $\FmB\lor\FmB_\tv \in \FPart$.
        In the first case, by {$(\RuleA^{\lor,2}_{\BKName})$}, we get
        $\FmA\lor\FmB \in \TPart$,
        and, in the second, by {$(\RuleA^{\lor,4}_{\BKName})$}, 
        we have also $\FmA\lor\FmB \in \TPart$,
        and we are done.
        The case $v(\FmB) = \tv$ and $v(\FmA) \neq \uv$ is analogous.
    \end{itemize}
\end{itemize}
\end{proof}

\subsubsection{Implicative fragment of \L ukaziewicz's logic}

Let $\SetSetCR_{\LThreeName}^{\to}$ be the \SetSet{} logic determined by
the three-valued logical matrix $\MatA_{\LThreeName}^{\to}$ whose signature contains only $\to$, interpreted as $\lukImp$ in this case --- the truth table of which we 
repeat below for convenience --- and whose set
of designated values is $\{\tv\}$. The reader can easily verify
that this matrix is not monadic,
as applications of $\lukImp$
on the same value always yield $\tv$.
As in the previous section, despite this fact,
we will axiomatize this logic in \SetSet{}.

\[
    \begin{tabular}{c|ccc}
        \toprule
         \lukImp & \bv & \uv & \tv\\
        \midrule
         \bv & \tv & \tv & \tv\\
         \uv & \uv & \tv & \tv\\
         \tv & \bv & \uv & \tv\\
        \bottomrule
    \end{tabular}
\]

\begin{definition}
Let $\CalcA^\to_{\LThreeName}$ be the \SetSet{} calculus 
whose rule schemas are the following:
\begin{gather*}
    \MCRule{\PropB}{\PropA\to\PropB}\RuleA_{\LThreeName}^{1,\to}
    \quad
    \MCRule{\PropA_1\to\PropA_2, \PropA_1}{\PropA_2}\RuleA_{\LThreeName}^{2,\to}\quad
    \MCRule{(\PropA_1 \to \PropA_2) \to \PropB_0}{\PropA_2 \to \PropB_0}\RuleA_{\LThreeName}^{3,\to}\quad
    \MCRule{\PropA_2 \to \PropB_0, \PropA_1}{(\PropA_1\to\PropA_2)\to\PropB_0}\RuleA_{\LThreeName}^{4,\to}
    \\
    \MCRule{\PropA_1 \to \PropA_2, \PropA_2 \to \PropB_0}{\PropA_1 \to \PropB_0}\RuleA_{\LThreeName}^{5,\to}\quad
    \MCRule{(\PropA_1 \to \PropA_2) \to \PropB_0, \PropA_2 \to \PropB_0}{\PropA_1 \to \PropB_0, \PropA_1}\RuleA_{\LThreeName}^{6,\to}\\
    \MCRule{\PropA_1 \to \PropB_0}{\PropA_1\to\PropA_2,\PropB_{\half} \to \PropB_0, \PropB_\half}\RuleA_{\LThreeName}^{7,\to}\quad
    \MCRule{}{\PropA_1 \to \PropA_2, \PropA_2 \to \PropB_0, \PropA_1}\RuleA_{\LThreeName}^{8,\to}\\
\end{gather*}
\end{definition}

\begin{theorem}
$\CalcA^\to_{\LThreeName}$ is sound for $\MatA_{\LThreeName}^\to$.
\end{theorem}
\begin{proof}
This is routine.
\end{proof}
\begin{theorem}
\label{the:lukimpcomp}
$\CalcA^\to_{\LThreeName}$ is complete for $\MatA_{\LThreeName}^\to$
and $\{\PropA \to \PropB\}$-analytic.
\end{theorem}
\begin{proof}
Let $\AnalyticSetA \SymbDef \{\PropA \to \PropB\}$.
Suppose that $\FmSetA \NAnaSetSetCR{\CalcA^\to_{\LThreeName}}{\AnalyticSetA} \FmSetB$,
then, by cut for sets, there is $\SetCut \subseteq \LangSetA$
such that
$\SetCut \NAnaSetSetCR{\CalcA^\to_{\LThreeName}}{\AnalyticSetA} \FmSetComp\SetCut$,
with $\FmSetA \subseteq \SetCut$ and $\FmSetB \subseteq \FmSetComp\SetCut$.
Let $\GenSubs{\AnalyticSetA} \SymbDef \Subf{\FmSetA \cup \FmSetB} \cup \{\FmA \to \FmB \mid \FmA,\FmB \in \Subf{\FmSetA \cup \FmSetB}\}$.
Then $\TPart \NAnaSetSetCR{\CalcA^\to_{\LThreeName}}{\AnalyticSetA} \FPart$
for some partition $(\TPart, \FPart)$ of $\GenSubs{\AnalyticSetA}$,
with $\FmSetA \subseteq \TPart$ and $\FmSetB \subseteq \FPart$.
We consider two cases, defining
for each of them a mapping 
$v : \Subf{\FmSetA\cup\FmSetB}\to\ThreeValuesSet$
and proving that $v(\FmA\to\FmB) = v(\FmA) \lukImp v(\FmB)$:
\begin{itemize}
    \item (a): for every $\FmA,\FmB \in \FPart \cap \Subf{\FmSetA\cup\FmSetB}$, we have
    $\FmA \to \FmB \in \TPart$.
    Let us define $v : \Subf{\FmSetA\cup\FmSetB}\to\ThreeValuesSet$ such that
    \[
        v(\FmA) \SymbDef
        \begin{cases}
            \tv & \text{ if } \FmA \in \TPart\\
            \bv & \text{ if } \FmA \in \FPart\\
        \end{cases}
    \]
    \begin{itemize}
        \item if $v(\FmA) = \tv$ and $v(\FmB)=\bv$: 
        by rule $(\RuleA_{\LThreeName}^{2,\to})$, we
        must have $\FmA \to \FmB \in \FPart$, that is,
        $v(\FmA\to\FmB)=\bv$;
        \item if $v(\FmA\to\FmB)=\bv$:
        by $(\RuleA_{\LThreeName}^{1,\to})$, we have
        $\FmB \in \FPart$ and, by (a), we must have
        $\FmA \in \TPart$;
        \item if $v(\FmA) = \bv$: by {$(\RuleA_{\LThreeName}^{8,\to})$}, we have
        that (b): $\FmA\to\FmB \in \TPart$ or 
        (c): $\FmB\to\FmA \in \TPart$.
        Case (b), we are done. Case (c), by $(\RuleA_{\LThreeName}^{2,\to})$,
        we have $\FmB \in \FPart$,
        and then, by (a), we have $\FmA\to\FmB \in \TPart
        $;
        \item if $v(\FmB) = \tv$: by rule $(\RuleA_{\LThreeName}^{1,\to})$, we
        must have $\FmA \to \FmB \in \TPart$;
        \item if $v(\FmA\to\FmB) = \tv$:
        by $(\RuleA_{\LThreeName}^{2,\to})$, we must have
        either $\FmA \in \FPart$ or $\FmB \in \TPart$.
    \end{itemize}
    \item there are $\FmA_{\uv}, \FmB_\bv \in \FPart\cap\Subf{\FmSetA\cup\FmSetB}$ with
    $\FmA_{\uv} \to \FmB_\bv \in \FPart$. In this case, let $v : \Subf{\FmSetA\cup\FmSetB}\to\ThreeValuesSet$ such that
    \[
        v(\FmA) \SymbDef
        \begin{cases}
            \tv & \text{ if } \FmA \in \TPart\\
            \uv & \text{ if } \FmA, \FmA\to\FmB_\bv \in \FPart\\
            \bv & \text{ if } \FmA \in \FPart, \FmA\to\FmB_\bv \in \TPart\\
        \end{cases}
    \]
    \begin{itemize}
        \item if $v(\FmA) = \tv$ and $v(\FmB)=\bv$: 
        we have $\FmB \in \FPart, \FmB \to \FmB_\bv \in \TPart$,
        then, by {$(\RuleA_{\LThreeName}^{4,\to})$}, we have
        $(\FmA \to \FmB)\to\FmB_\bv \in \TPart$,
        and, by {$(\RuleA_{\LThreeName}^{1,\to})$}, $\FmA \to \FmB \in \FPart$,
        thus $v(\FmA\to\FmB) = \bv$;
        \item if $v(\FmA\to\FmB)=\bv$: by {$(\RuleA_{\LThreeName}^{3,\to})$},
        we have $\FmB \to \FmB_\bv \in \TPart$
        and by {$(\RuleA_{\LThreeName}^{2,\to})$}, we have $\FmB \in \FPart$, thus
        $v(\FmB)=\bv$; by {$(\RuleA_{\LThreeName}^{6,\to})$}, we have either
        $\FmA \in \TPart$ or $\FmA \to \FmB_\bv \in \TPart$.
        In the first case, we are done; the second
        case leads to an absurd, by {$(\RuleA_{\LThreeName}^{7,\to})$},
        since $\FmA_{\uv} \to \FmB_\bv, \FmB_\bv \in \FPart$;
        \item if $v(\FmA) = \bv$: by rule {$(\RuleA_{\LThreeName}^{8,\to})$},
        we have either $\FmA\to\FmB \in \TPart$
        or $\FmB\to\FmB_{\uv} \in \TPart$.
        In the first case, we are done.
        In the second, we must have $\FmA\to\FmB \in \TPart$
        by rule {$(\RuleA_{\LThreeName}^{7,\to})$};
        \item if $v(\FmB) = \tv$: by rule {$(\RuleA_{\LThreeName}^{1,\to})$}, we have
        $\FmA\to\FmB \in \TPart$;
        \item if $v(\FmA\to\FmB) = \tv$: then $\FmA\to\FmB \in \TPart$. By {$(\RuleA_{\LThreeName}^{2,\to})$}, we have either (b): $\FmA \in \FPart$
        or (c): $\FmB \in \TPart$. In case (b), by {$(\RuleA_{\LThreeName}^{5,\to})$}, we have either
        (d): $\FmA\to\FmB_\bv \in \TPart$ or
        (e): $\FmB\to\FmB_\bv \in \FPart$.
        Case (d), we have $v(\FmA) = \bv$, as desired.
        Case (e), if $\FmB \in \TPart$, we are done,
        and if $\FmB \in \FPart$, we will have $v(\FmB)=\uv$,
        and we are also done.
        Case (c), that is, $\FmB \in \TPart$, $v(\FmB)=\tv$ and we are done;
        \item if $v(\FmA\to\FmB) = \uv$: then
        $\FmA\to\FmB,(\FmA\to\FmB)\to\FmB_\bv \in \FPart$.
        Then, by {$(\RuleA_{\LThreeName}^{4,\to})$}, either (b): $\FmA \in \FPart$ or
        (c): $\FmB\to\FmB_\bv \in \FPart$.
        Case (b), by {$(\RuleA_{\LThreeName}^{8,\to})$}, we have $\FmB\to\FmB_\bv \in \TPart$,
        and, by {$(\RuleA_{\LThreeName}^{2,\to})$}, we have $\FmB\in\FPart$, meaning
        that $v(\FmB)=\bv$. Moreover, by {$(\RuleA_{\LThreeName}^{7,\to})$}, we have $\FmA\to\FmB_\bv\in\FPart$, hence $v(\FmA)=\uv$,
        as desired.
        Case (c), by {$(\RuleA_{\LThreeName}^{8,\to})$} we have $\FmA \in \TPart$, i.e. $v(\FmA)=\tv$.
        Also, by {$(\RuleA_{\LThreeName}^{1,\to})$} and the fact that $\FmA\to\FmB \in \FPart$, we
        have $\FmB \in \FPart$, and then $v(\FmB)=\uv$, as desired;
        \item if $v(\FmA) = \uv$ and $v(\FmB)=\bv$:     then
        $\FmA,\FmA\to\FmB_\bv \in \FPart$
        and $\FmB \in \FPart$ and $\FmB\to\FmB_\bv \in \TPart$.
        By {$(\RuleA_{\LThreeName}^{5,\to})$}, we have $\FmA\to\FmB \in \FPart$, and, 
        by {$(\RuleA_{\LThreeName}^{6,\to})$}, we have $(\FmA\to\FmB)\to\FmB_\bv \in \FPart$,
        we have $v(\FmA\to\FmB)=\uv$.
        \item if $v(\FmA) = \tv$ and $v(\FmB)=\uv$:
        then $\FmA \in \TPart$, $\FmB \in \FPart$
        and $\FmB\to\FmB_\bv \in \TPart$. By {$(\RuleA_{\LThreeName}^{4,\to})$},
        we have $(\FmA\to\FmB)\to\FmB_\bv \in \FPart$,
        and, by {$(\RuleA_{\LThreeName}^{1,\to})$}, we have $\FmA\to\FmB \in \FPart$,
        thus $v(\FmA\to\FmB)=\uv$.
    \end{itemize}
\end{itemize}
\end{proof}

\subsubsection{Implicative fragment of G\"odel's logic}

The last example in this series of
three-valued logics determined by
non-monadic matrices concerns
$\SetSetCR_{\GThreeName}^{\to}$, the \SetSet{} logic determined by
the three-valued logical matrix $\MatA_{\GThreeName}^{\to}$ --- we repeat the
interpretation of $\to$ below for convenience.
In this case, we will be able to take advantage
of the calculus introduced in the previous
subsection in order to obtain the desired
axiomatization.

\[
    \begin{tabular}{c|ccc}
        \toprule
             \HImp & \bv & \uv & \tv\\
        \midrule
         \bv & \tv & \tv & \tv\\
         \uv & \bv & \tv & \tv\\
         \tv & \bv & \uv & \tv\\
        \bottomrule
    \end{tabular}
\]

\begin{definition}
Let $\CalcA^\to_{\GThreeName}$ be the \SetSet{} calculus 
whose rule schemas are the same as those of $\CalcA^\to_{\LThreeName}$,
except that we replace $(\RuleA_{\LThreeName}^{6,\to})$
by the following:
\begin{gather*}
\MCRule{\PropA_2 \to \PropB_0}{(\PropA_1\to\PropA_2)\to\PropB_0, \PropA_1\to\PropB_0}(\RuleA_{\GThreeName}^{6,\to})
\end{gather*}
\end{definition}

\begin{theorem}
$\CalcA^\to_{\GThreeName}$ is sound for $\MatA_{\GThreeName}^\to$.
\end{theorem}
\begin{proof}
This is routine.
\end{proof}
\begin{theorem}
$\CalcA^\to_{\GThreeName}$ is complete for $\MatA_{\GThreeName}^\to$
and $\{\PropA \to \PropB\}$-analytic.
\end{theorem}
\begin{proof}
The proof is very similar to the one of Theorem~\ref{the:lukimpcomp},
di{f}fering only in some sub-cases of the second item:
\begin{itemize}
    \item if $v(\FmA\to\FmB)=\bv$: then 
    $\FmA\to\FmB \in \FPart$
    and $(\FmA\to\FmB)\to\FmB_\bv \in \TPart$.
    By {$(\RuleA_{\GThreeName}^{3,\to})$},
        we have $\FmB \to \FmB_\bv \in \TPart$
        and by {$(\RuleA_{\GThreeName}^{2,\to})$}, we have $\FmB \in \FPart$, thus
        $v(\FmB)=\bv$; also, by {$(\RuleA_{\GThreeName}^{7,\to})$}, we have
        $\FmA\to\FmB_\bv \in \FPart$.
        Then, if $\FmA \in \TPart$,
        we are done, and if $\FmA \in \FPart$,
        we have $v(\FmA)=\uv$, as desired.
    \item if $v(\FmA\to\FmB)=\uv$:
    then
        $\FmA\to\FmB,(\FmA\to\FmB)\to\FmB_\bv \in \FPart$.
        Then, by {$(\RuleA_{\GThreeName}^{4,\to})$}, either (b): $\FmA \in \FPart$ or
        (c): $\FmB\to\FmB_\bv \in \FPart$.
        Case (b), by {$(\RuleA_{\GThreeName}^{8,\to})$}, we have $\FmB\to\FmB_\bv \in \TPart$. 
        By {$(\RuleA_{\GThreeName}^{7,\to})$}, we have $\FmA\to\FmB_\bv\in\FPart$,
        but, by {$(\RuleA_{\GThreeName}^{6,\to})$}, we have $\FmA\to\FmB_\bv \in \TPart$,
        an absurd.
        Case (c), by {$(\RuleA_{\GThreeName}^{8,\to})$} we have $\FmA \in \TPart$, i.e. $v(\FmA)=\tv$.
        Also, by {$(\RuleA_{\GThreeName}^{1,\to})$} and the fact that $\FmA\to\FmB \in \FPart$, we
        have $\FmB \in \FPart$, and then $v(\FmB)=\uv$, as desired;
    \item if $v(\FmA)=\uv$ and $v(\FmB)=\bv$:
    then
        $\FmA,\FmA\to\FmB_\bv \in \FPart$
        and $\FmB \in \FPart$ and $\FmB\to\FmB_\bv \in \TPart$.
        By {$(\RuleA_{\GThreeName}^{5,\to})$}, we have $\FmA\to\FmB \in \FPart$, and, 
        by {$(\RuleA_{\GThreeName}^{6,\to})$}, we have $(\FmA\to\FmB)\to\FmB_\bv \in \TPart$,
        we have $v(\FmA\to\FmB)=\bv$.
\end{itemize}
\end{proof}

\section{Some correspondences between 3-labelled and multiple-conclusion Hilbert-style calculi}
\label{ComparisonBetween3labelledAndMultipleConclusionCalculi}

As to the comparison between three-labelled and \SetSet{} Hilbert-style calculi, we first observe the following: \SetSet{} Hilbert-style calculi manipulate sets of formulas according to the Tarskian postulates (cf.~Definition \ref{def:SetSet}), so the structural rules of three-labelled calculi (Identity, Weakening and Cut --- see Section~\ref{sec:3LabelledCalculi}) are implicitly assumed.
Therefore, we can confine ourselves to consider logical rules. 
For simplicity, 
we will also limit ourselves in this section to matrices with
the single element $\tv$ designated, as adapting the results
to the other cases is not a difficult task.
The following
shows a correspondence between
the generating subprocedures of each formalism. 
To allow for a more general
presentation, we use the
subprocedure of the \SetSet{} formalism for the non-deterministic case~\cite{marcelino2019}, which is
equivalent to the one we
presented in Section~\ref{def:monadiccalc}:

\begin{theorem}
\label{thm:FullMultipleIFFFullLabelledCalcluli}
Consider the three-valued monadic matrix 
$\MatA \SymbDef \langle \AlgA, \{\tv\} \rangle$.
For each $k \in \omega$,
$\conn \in \Sigma^k$, and for all $\LabelValA_1,\ldots,\LabelValA_n \in \ThreeValuesSet$,
the 3-labelled rule
\begin{center}
\AXC{$(\CtxLabAZero \mid \CtxLabAHalf \mid \CtxLabAOne)[\,\LabelValA_1 : \FmA_1\,] \quad \cdots \quad (\CtxLabAZero \mid \CtxLabAHalf \mid \CtxLabAOne)[\,\LabelValA_n : \FmA_n\,]$}
\RL{(R)}
\UIC{$(\CtxLabAZero \mid \CtxLabAHalf \mid \CtxLabAOne)[\,\LabelValA : \conn(\FmA_1,\ldots,\FmA_n)\,]$}
\DP
\end{center}
is generated by the sub-procedure for 3-labelled calculi
if, and only if, the \SetSet{} rules
\[
(\mathsf{r}_1)\frac{
\bigcup_{i=1}^n\Omega_{\LabelValA_i}(\FmA_i) \cup
\Omega_{\bar y_1}(\conn(\FmA_1,\ldots,\FmA_n))
}{
\bigcup_{i=1}^n\mho_{\LabelValA_i}(\FmA_i) \cup
\mho_{\bar y_1}(\conn(\FmA_1,\ldots,\FmA_n))
}
\qquad
(\mathsf{r}_2)\frac{
\bigcup_{i=1}^n\Omega_{\LabelValA_i}(\FmA_i) \cup
\Omega_{\bar y_2}(\conn(\FmA_1,\ldots,\FmA_n))
}{
\bigcup_{i=1}^n\mho_{\LabelValA_i}(\FmA_i) \cup
\mho_{\bar y_2}(\conn(\FmA_1,\ldots,\FmA_n))
}
\]
are generated by the sub-procedure for \SetSet{} Hilbert-style calculi in the non-deterministic setting (cf.\cite{marcelino2019}),
where $\ThreeValuesSet = \{\bar y_1, \bar y_2,  \conn_\AlgA(\LabelValA_1,\ldots,\LabelValA_n)\}$.
\end{theorem}
\begin{proof}
From the left to the right, assume that
(R) was generated by the sub-procedure for
3-labelled calculi, meaning that
 $\AlgInterp{\conn}{\AlgA}(\LabelValA_1,\ldots,\LabelValA_k) = y$.
Then the procedure for 
generating \SetSet{} calculi
in the non-deterministic setting
produces one rule for each
$z \neq \AlgInterp{\conn}{\AlgA}(\LabelValA_1,\ldots,\LabelValA_k)$,
whose shape is given precisely by $(\mathsf{r}_1)$
and $(\mathsf{r}_2)$, where $\{\bar y_1, \bar y_2\} = \ThreeValuesSet\setminus\{\AlgInterp{\conn}{\AlgA}(\LabelValA_1,\ldots,\LabelValA_k)\}$.
The converse direction is similar:
if $(\mathsf{r}_1)$ and $(\mathsf{r}_2)$ are produced,
then a single
rule will be generated by the 3-labelled
sub-procedure bearing witness to the
fact that $y = \AlgInterp{\conn}{\AlgA}(\LabelValA_1,\ldots,\LabelValA_k) \not\in \{\bar y_1, \bar y_2\}$.
Notice that, in this way, if $\conn$ is $n$-ary, a
single 3-labelled rule per entry is produced ($3^n$ rules at the end),
while two \SetSet{} rules are produced
per entry (so, $2\cdot3^n$ rules)
by the \SetSet{} sub-procedure. 
Thus, a 3-labelled rule
corresponds to two \SetSet{} rules
(and vice-versa).
\end{proof}

\begin{example}
\label{ex:translation-rules-procedures}
Here we show some correspondences
between rules of $\ThreeLabStreamName{\LThreeName}$ and rules of $\smName{\LThreeName}$
according with the
previous theorem:
\begin{itemize}
    \item 
    \begin{math}
    \AXC{$\CtxLabAZero, \FmA \mid \CtxLabAHalf \mid \CtxLabAOne$}
    \RL{\fns $\aneg^\bv_\tvb$}
    \UIC{$\CtxLabAZero \mid \CtxLabAHalf \mid \CtxLabAOne, \aneg \FmA$}
    \DP
    \end{math}
    corresponds to
    $\AXC{$\neg\neg\FmA, \neg\FmA$}
    \RL{\fns $\aneg^\bv_{\tvb\bv}$}
    \UIC{$\neg\FmA, \FmA$}
    \DP$
    and
    $\AXC{$\aneg\FmA$}
    \RL{\fns $\aneg^\bv_{\tvb\uv}$}
    \UIC{$\aneg\aneg\FmA, \aneg\FmA, \FmA$}
    \DP$, and vice-versa.
    \item
    \begin{math}
    \AXC{$\CtxLabAZero \mid \CtxLabAHalf, \FmA \mid \CtxLabAOne$}
\AXC{$\CtxLabAZero \mid \CtxLabAHalf \mid \CtxLabAOne, \FmB$}
\RL{\fns $\aand^{\uv\tv}_\uvg$}
\BIC{$\CtxLabAZero \mid \CtxLabAHalf, \FmA \aand \FmB \mid \CtxLabAOne$}
\DP
    \end{math}
    corresponds to
    $\AXC{$\neg(\FmA \land \FmB), \FmB$}
\RL{\fns $\aand^{\uv\tv}_{\uvg\bv}$}
\UIC{$\FmA \aand \FmB, \neg\FmA, \FmA$}
\DP$

and
$\AXC{$\FmA \land \FmB, \FmB$}
\RL{\fns $\aand^{\uv\tv}_{\uvg\tv}$}
\UIC{$\neg\FmA, \FmA$}        
\DP$, and vice-versa.
    \item
    $\AXC{$\CtxLabAZero, \FmA \mid \CtxLabAHalf \mid \CtxLabAOne$}
\AXC{$\CtxLabAZero \mid \CtxLabAHalf, \FmB \mid \CtxLabAOne$}
\RL{\fns $\ararr^{\bv\uv}_\tvb$}
\BIC{$\CtxLabAZero \mid \CtxLabAHalf \mid \CtxLabAOne, \FmA \ararr \FmB$}
\DP$ corresponds to
$\AXC{$\aneg(\FmA \ararr \FmB), \neg\FmA$}
\RL{\fns $\ararr^{\bv\uv}_{\tvb\bv}$}
\UIC{$\FmA \to \FmB, \neg\FmB, \FmA, \FmB$}
\DP$ 

and
$\AXC{$\neg\FmA$}
\RL{\fns $\ararr^{\bv\uv}_{\tvb\uv}$}
\UIC{$\FmA \to \FmB, \aneg(\FmA \ararr \FmB), \neg\FmB, \FmA, \FmB$}
\DP$,
and vice-versa.
\end{itemize}

\end{example}

\vspace{0.5cm}

The above result just states that, modulo monadicity,
both generating procedures work by similar mechanisms,
characterizing the interpretations of the three-valued
matrix at hand.
An interesting problem is whether there
is an effective translation between
proofs of each generated calculi.
The next results show that, from 
\SetSet{} proofs to 3-labelled proofs,
we may explore the fact that
\SetSet{} proofs internalize
proof-by-cases~\cite[Chapter 20]{shoesmithsmiley1978},
in order to provide an effective translation.
We first show how to translate \SetSet{}
rules of inference into 3-labelled rules, and then explain how to use
the latter in transforming proofs between these formalisms
preserving $\Theta$-analyticity.
From now on, we will write \SetSet{} rules
schematically (using variables for formulas
instead of propositional variables),
in order to ease the comparison with 
3-labelled rules.

\begin{lemma}
\label{lem:Derivability-MultipleToLabelledCalculi}
Let $\LName$ be the \SetSet{} logic
determined by the three-valued matrix 
$\MatA \SymbDef \langle \AlgA, \{\tv\} \rangle$ and
$\CalcA$ be a \SetSet{} axiomatization for $\MatA$. For every \emph{primitive} rule  $\mathsf{r} \SymbDef \MCRule{\FmA_1,\ldots,\FmA_m}{\FmB_1,\ldots,\FmB_n} \in \CalcA$, the 3-sequent $\FmA_1, \ldots, \FmA_m \mid \FmA_1, \ldots, \FmA_m \mid \FmB_1, \ldots, \FmB_n$
has a cut-free derivation in $\ThreeLabStreamName{\LName}$.
\end{lemma}
\begin{proof}
It follows directly from the facts that
$\RuleA$ is a consecution of $\LName$
and that
$\ThreeLabStreamName{\LName}$ is adequate with respect
$\MatA$ and cut is admissible in it.
%
%
%
%
%
%
%
%
\end{proof}

\begin{lemma}
\label{thm:derivable-rule-mc}
Let $\LName$ be the logic
determined by the three-valued matrix 
$\MatA \SymbDef \langle \AlgA, \{\tv\} \rangle$ and
$\CalcA$ be a \SetSet{} axiomatization for $\MatA$. For every \emph{primitive} rule  $\mathsf{r} \in \CalcA$, the image of the translation $\lf \cdot \rf$
given by
\begin{center}
\begin{tabular}{c}
$\left\lf
\AXC{$\FmA_1,\ldots,\FmA_m$}
\RL{$\mathsf{r}$}
\UIC{$\FmB_1, \ldots, \FmB_n$}
\DP
\right\rf$
\ \  $:=$ \\
 \\
\AXC{$F, \CtxLabAZero, \FmB_1 \mid F, \CtxLabAZero, \FmB_1 \mid \CtxLabAOne$}
\AXC{$\cdots$}
\AXC{$F, \CtxLabAZero, \FmB_n \mid F, \CtxLabAZero, \FmB_n \mid \CtxLabAOne$}
\RL{$\lf\mathsf{r}\rf$}
\TIC{$F, \CtxLabAZero \mid F, \CtxLabAZero \mid \CtxLabAOne$}
\DP
 \\
\end{tabular}
\end{center}
where $F \SymbDef \{\FmA_1,\ldots,\FmA_m\}$
is a derivable rule  of $\ThreeLabStreamName{\LName}$.
\end{lemma}
\begin{proof}
By Lemma~\ref{lem:Derivability-MultipleToLabelledCalculi},
$\FmA_1, \ldots, \FmA_m \mid \FmA_1, \ldots, \FmA_m \mid \FmB_1, \ldots, \FmB_n$
is derivable in $\ThreeLabStreamName{\LName}$.
The result easily follows by applying cuts using this 3-sequent and the 3-sequents in the premises of $\lf\mathsf{r}\rf$.
\end{proof}
%

\begin{theorem}
There is an effective procedure to translate
$\CalcA$-proofs into
proofs in $\ThreeLabStreamName{\LName}$.
Moreover, the translation preserves $\Theta$-analyticity.
\end{theorem}
\begin{proof}
We will prove by induction on the
structure of $\CalcA$-derivations $\TreeA$ that
$S(\TreeA) \SymbDef $ ``for all $\FmSetA,\FmSetB \subseteq \LangSetA$,
if $\TreeA$ is an $\CalcA$-proof of 
$(\FmSetA,\FmSetB)$, then
$\TreeA$ can be effectively translated into a proof 
$\lf \TreeA \rf$ of the same statement in $\ThreeLabStreamName{\LName}$
preserving $\Theta$-analyticity''.
For the base case, suppose that
$\TreeA$ has a single node labelled with $\FmSetD$. Since it is a
proof of $(\FmSetA,\FmSetB)$, we must have
$\FmSetD \subseteq \FmSetA$ and
$I \SymbDef {\FmSetD} \cap \FmSetB \neq \EmptySet$.
Consider an enumeration $\FmA_1,\ldots,\FmA_n,\ldots$
of the formulas in $I$.
Then take $\lf \TreeA \rf$ to be the tree
\begin{figure}[H]
    \centering
    \small
    \AxiomC{}
    \RL{\fns Id$_{\FmA_1}$}
    \UIC{$\FmA_1 \mid \FmA_1 \mid \FmA_1$}
    \RL{$W^\ast$}
    \UnaryInfC{$I \mid I \mid I$}
    \RL{$W^\ast$}
    \UnaryInfC{$\FmSetD \mid \FmSetD \mid \FmSetB$}
    \RL{$W^\ast$}
    \UnaryInfC{$\FmSetA \mid \FmSetA \mid \FmSetB$}
    \DP
\end{figure}
\noindent which is clearly a proof of 
$(\FmSetA,\FmSetB)$ in $\ThreeLabStreamName{\LName}$
that is $\Theta$-analytic if $\TreeA$ is
$\Theta$-analytic.
Suppose now that $\TreeA$ has its root labelled
with $\FmSetD$ and has
subtrees $\TreeA_1,\ldots,\TreeA_n$, which resulted
from an application of the rule instance $\RuleA^\sigma
\SymbDef \MCRule{\FmB_1,\ldots,\FmB_m}{\FmC_1,\ldots,\FmC_n}$.
Note that
$\TreeA_i$ is an $\CalcA$-proof of
$(\FmSetD \cup \{\FmC_i\},\FmSetB)$.
Assume that (IH): $S(\TreeA_i)$ holds for every
$1 \leq i \leq n$. Therefore, each $\TreeA_i$
has a translation to a proof $\lf \TreeA_i \rf$ 
in $\ThreeLabStreamName{\LName}$ of the statement 
$(\FmSetD \cup \{\FmC_i\},\FmSetB)$.
Since $\Gamma \subseteq \FmSetA$
and $\{\FmB_1,\ldots,\FmB_m\} \subseteq \Gamma$,
the following tree is the translation we are looking for
(recall the definition of
$\lf \RuleA \rf$ in Lemma~\ref{thm:derivable-rule-mc}):

\begin{figure}[H]
    \centering
    \small
    \AxiomC{$\lf \TreeA_1 \rf$}
    \noLine
    \UnaryInfC{$\FmSetD,\FmC_1 \mid \FmSetD,\FmC_1 \mid \FmSetB$}
    \AxiomC{$\cdots$}
    \AxiomC{$\lf \TreeA_n \rf$}
    \noLine
    \UnaryInfC{$\FmSetD, \FmC_n \mid \FmSetD, \FmC_n \mid \FmSetB$}
    \RL{$\lf \RuleA \rf$}
    \TIC{$\FmSetD \mid \FmSetD \mid \FmSetB$}
    \RL{$W^\ast$}
    \UnaryInfC{$\FmSetA \mid \FmSetA \mid \FmSetB$}
    \DP
\end{figure}

\noindent By Lemma~\ref{lem:Derivability-MultipleToLabelledCalculi} and the proof of Lemma~\ref{thm:derivable-rule-mc}, we observe
that only formulas in $\{\FmB_1,\ldots,\FmB_m\} \cup \{\FmC_1,\ldots,\FmC_n\}$
are used in the derivation of the applied instance of $\lf \RuleA \rf$,
and since all these formulas appeared in labels of the nodes of $\TreeA$, they are
$\Theta$-subformulas of $(\FmSetA,\FmSetB)$.
By (IH), then, we conclude that $\lf \RuleA \rf$ is $\Theta$-analytic.
\end{proof}

From 3-labelled proofs to \SetSet{} proofs,
we leave the task for further research. A starting point
would be to impose the monadicity requirement
for the axiomatized matrix and search for
a finite collection of \SetSet{} rules that capture
the application of a 3-labelled rule, using
the fact that each of such rules represent
a collection of entries of the interpretation
of a connective
of the logical matrix. 

\section{Final considerations}
\label{sec:Conclusion}

When it comes to the
automatic generation of
proof systems for propositional
three-valued logics, 
the formalism of 3-labelled calculi 
is powerful enough to deliver finite and analytic
systems for any of such logics.
The mechanism behind this general applicability is fairly simple: each truth value has its own
place in the manipulated sequents, in such a
way that the meta-conjunctions and meta-disjunctions
that characterize each truth table
have a direct correspondence in terms of
3-labelled logical rules (see~\cite{BaaFerZac98}
for more details on this matter).
Enriching the structure of the sequents, however, may bring some disadvantages, namely
a greater distancing between the rules of the proof system and the logic being 
captured, as well as the extra effort in dealing
with the new metalinguistic resources.
One of the research lines in which 
these aspects may become practical problems is
in the combination of logics: putting two
sequent systems together may cause unwanted
interactions in the combined system~\cite{joaocarlossergiomerging2017}.

An alternative to increasing the metalanguage resides in investing on Hilbert-style
systems, which keep the metalanguage at a minimum,
the rules of inference being just a collection
of consecutions of the underlying logic.
When it comes to system generation, proof search and decidability, 
these systems have been put aside due to the
lack of modularity and analyticity results. After
the works of C. Caleiro and S. Marcelino~\cite{marcelino2019,marcelinowollic19}, we learned that the 
simple move to \SetSet{} Hilbert-style systems~\cite{shoesmithsmiley1978}
allows us to get modular and analytic 
systems for a wide class of many-valued logics.
The price for abandoning the metalanguage
present in 3-sequents is the sufficient
expressiveness requirement (monadicity),
which is demanded for the automatic
proof system generation in the
\SetSet{} setting. 
Summing up, as we have seen in Section~\ref{ComparisonBetween3labelledAndMultipleConclusionCalculi},
the choice of proof formalism for a
given application depends
essentially on a trade-off between
the acceptable usage of extra metalanguage
(or the proximity of the rules  to the
logic of interest)
and the complexity of the formulas
present in the rules of inference
(cf. Example~\ref{ex:translation-rules-procedures}).

It is worth emphasizing that
every three-valued logic (in the sense of being
determined by a single deterministic three-valued matrix)
has a finite \SetSet{} axiomatization~\cite{shoesmithsmiley1978}, even though
there is no general procedure for producing it
in the case non-monadic matrices (provided we want to use the same language --- if one is allowed to expand the language, the problem is solved~\cite[Section 4]{marcelino2019}). The best we can
do, for now, is to proceed in an \emph{ad hoc}
manner, as we did for some fragments of
well-known three-valued logics in Subsection~\ref{sec:nonmonadic}.
A promising line of investigation, therefore, consists in
pursuing a general procedure covering non-monadic cases and delivering
analytic \SetSet{} calculi over the original language.

Finally, despite the fact that the three-valued
logics studied in this chapter are
all determined by single deterministic logical
matrices and the generation procedures were presented for these cases only, such procedures have all been generalized
to (partial) non-deterministic matrices~\cite{marcelinowollic19,avronbeata2005}.
These are structures that allow for
(possibly empty) sets of truth values as
outputs of the interpretations of the
connectives. With partial matrices, for example,
one may capture by single monadic matrices
three-valued logics determined by 
a family of deterministic three-valued matrices~\cite[Example 3]{marcelino2019}.
In this way, taking the study of three-valued logics
to partial non-deterministic semantics 
is a natural step towards the discovery
of new logics and applications.

    

\bibliographystyle{plain}
\bibliography{main}


\end{document}